\documentclass[a4paper,11pt,reqno]{article}
   
\usepackage[utf8]{inputenc}
\usepackage[T1]{fontenc}
\usepackage{lmodern}
\usepackage[english]{babel}
\usepackage{microtype}

\usepackage{amsmath,amssymb,amsfonts,amsthm}
\usepackage{mathtools,accents}
\usepackage{mathrsfs}
\usepackage{aliascnt}
\usepackage{braket}
\usepackage{bm}
\usepackage{esint}

\usepackage[a4paper,margin=3cm]{geometry}
\usepackage[citecolor=blue,colorlinks]{hyperref}

\usepackage{enumerate}
\usepackage{xcolor}
\makeatletter
\g@addto@macro\@floatboxreset\centering
\makeatother

\makeatletter
\def\newaliasedtheorem#1[#2]#3{
  \newaliascnt{#1@alt}{#2}
  \newtheorem{#1}[#1@alt]{#3}
  \expandafter\newcommand\csname #1@altname\endcsname{#3}
}
\makeatother

\numberwithin{equation}{section}

\newtheoremstyle{slanted}{\topsep}{\topsep}{\slshape}{}{\bfseries}{.}{.5em}{}

\theoremstyle{plain}
\newtheorem{theorem}{Theorem}[section]
\newaliasedtheorem{proposition}[theorem]{Proposition}
\newaliasedtheorem{lemma}[theorem]{Lemma}
\newaliasedtheorem{corollary}[theorem]{Corollary}
\newaliasedtheorem{counterexample}[theorem]{Counterexample}

\theoremstyle{definition}
\newaliasedtheorem{definition}[theorem]{Definition}
\newaliasedtheorem{question}[theorem]{Question}
\newaliasedtheorem{example}[theorem]{Example}
\newaliasedtheorem{conjecture}[theorem]{Conjecture}

\theoremstyle{remark}
\newaliasedtheorem{remark}[theorem]{Remark}
\newaliasedtheorem{comment}[theorem]{Comment}

\newcommand{\setR}{\mathbb{R}}

\let\phi\varphi

\newcommand{\di}{\mathop{}\!\mathrm{d}}

\newcommand{\loc}{{\rm loc}}

\newcommand{\res}{\mathop{\hbox{\vrule height 7pt width .5pt depth 0pt
\vrule height .5pt width 6pt depth 0pt}}\nolimits}

\DeclareMathOperator{\supp}{supp}

\newcommand{\Ch}{{\sf Ch}}
\newcommand{\Leb}{\mathrm{Leb}}

\DeclareMathOperator{\Lip}{Lip}

\newcommand{\dist}{\mathsf{d}}

\newcommand{\meas}{\mathfrak{m}}

\DeclareMathOperator{\CD}{CD}
\DeclareMathOperator{\RCD}{RCD}

\newfont{\tmpf}{cmsy10 scaled 2500}

\begin{document}
\title{Sobolev mappings between RCD spaces and applications to harmonic maps: a heat kernel approach }
\author{
Shouhei Honda\thanks{Tohoku University, \url{shouhei.honda.e4@tohoku.ac.jp}} and  
Yannick Sire\thanks{Johns Hopkins University, \url{ysire1@jhu.edu}} } \maketitle

\begin{abstract} 
In this paper, we investigate a Sobolev map $f$ from a finite dimensional RCD space $(X, \dist_X, \meas_X)$ to a finite dimensional non-collapsed compact RCD space $(Y, \dist_Y, \mathcal{H}^N)$. It is proved that if the image $f(X)$ is smooth in a weak sense (which is satisfied if the pushforward measure $f_{\sharp}\meas_X$ is absolutely continuous with respect to the Hausdorff measure $\mathcal{H}^N$, or if $(Y, \dist_Y, \mathcal{H}^N)$ is smooth in a weak sense), then the pull-back $f^*g_Y$ of the Riemannian metric $g_Y$ of $(Y, \dist_Y, \mathcal{H}^N)$ is well-defined as an $L^1$-tensor on $X$, the minimal weak upper gradient $G_f$ of $f$ can be written by using $f^*g_Y$, and it coincides with the local slope $\mathrm{Lip}f$ for $\meas_X$-almost everywhere points in $X$ when $f$ is Lipschitz. In particular the last statement gives a nonlinear analogue of Cheeger's differentiability theorem for Lipschitz functions on metric measure spaces. 
Moreover these results allow us to define the energy of $f$. It is also proved that the energy coincides with the Korevaar-Schoen energy up to by multiplying a dimensional positive constant.
In order to achieve this, we use a smoothing of $g_Y$ via the heat kernel embedding $\Phi_t:Y \hookrightarrow L^2(Y, \mathcal{H}^N)$, which is established by Ambrosio-Portegies-Tewodrose and the first named author \cite{AHPT}. Moreover we improve the regularity of $\Phi_t$, which plays a key role to get the above results. As an application, we show that $(Y, \dist_Y)$ is isometric to the $N$-dimensional standard unit sphere in $\mathbb{R}^{N+1}$ and $f$ is a minimal isometric immersion if and only if $(X, \dist_X, \meas_X)$ is non-collapsed up to a multiplication of a constant to $\meas_X$, and $f$ is an eigenmap whose eigenvalues coincide with the essential dimension of $(X, \dist_X, \meas_X)$, which gives a positive answer to a remaining problem from a previous work \cite{honda20} by the first named author. This approach, using the heat kernel embedding instead of using Nash's one, to the study of energies of maps between possibly singular spaces seems new even for closed Riemannian manifolds.
\end{abstract}

\tableofcontents
\section{Introduction}
\subsection{Energy via Nash embedding}\label{1.1}
The study of \textit{harmonic maps} between Riemannian manifolds is a central topic in geometric analysis. A standard way for defining (smooth) harmonic maps is as follows. Let $(M, g_M)$ and $(N, g_N)$ be finite dimensional Riemannian manifolds and let $f:M \to N$ be a smooth map. Nash's embedding theorem allows  to find a smooth \textit{isometric} embedding
\begin{equation}
\Phi:N \to \mathbb{R}^k
\end{equation}
for some $k \in \mathbb{N}$, that is, $\Phi^*g_{\mathbb{R}^k}=g_N$. The \textit{energy of $f$} is defined by
\begin{equation}\label{smooth energy}
\mathcal{E}_{M, N}(f):=\frac{1}{2}\int_M\left| \dist (\Phi \circ f)\right|^2 \di \mathrm{vol}_{g_M},
\end{equation} 
where $\mathrm{vol}_{g_M}$ denotes the Riemannian volume measure of $(M, g_M)$. Note that (\ref{smooth energy}) does not depend on the choice of $\Phi$ because $\left| \dist (\Phi \circ f)\right|^2$ coincides with $\langle g_M, f^*g_N\rangle$. Then $f$ is said to be \textit{harmonic} if $f$ is a critical point of (\ref{smooth energy}) under any compactly supported smooth perturbations $f_t$ of $f$
\begin{equation}
\frac{\di}{\di t}\Big|_{t=0}\mathcal{E}_{M, N}(f_t)=0.
\end{equation}

The purpose of this paper is to provide a similar theory for non-smooth spaces with Ricci curvature bounded below, so-called \textit{$\RCD$-metric measure spaces}, which
are introduced in \cite{AmbrosioGigliSavare14} by Ambrosio-Gigli-Savar\'e (when $N=\infty$), \cite{AmbrosioMondinoSavare} by Ambrosio-Mondino-Savar\'e (treating $\RCD^*(K, N)$ spaces), \cite{Gigli13, Gigli1} by Gigli (treating the infinitesimally Hilbertian condition), and \cite{ErbarKuwadaSturm} by Erbar-Kuwada-Sturm (treating $\RCD^*(K, N)$ spaces),
after the introduction of $\CD(K, N)$ space introduced in \cite{LottVillani} by Lott-Villani and \cite{Sturm06a, Sturm06b} by Sturm, independently.

 Naively a metric measure space $(X, \dist_X, \meas_X)$ is said to be an \textit{$\RCD(K, N)$ space}, or an \textit{$\RCD$ space} for short, if 
\begin{itemize}
\item $(X, \dist_X)$ is a complete separable metric space, $\meas_X$ is a Borel measure on $X$ which is finite on each bounded set, the heat flow is linear, the Ricci curvature is bounded below by $K$ and the dimension is bounded above by $N$, in a synthetic sense.
\end{itemize}
See Definition \ref{def:rcd} for the precise definition {\color{blue}and  \cite{A} for a nice survey}. We say that an $\RCD$ space is \textit{finite dimensional} if $N$ can be taken as a finite number. Typical examples of $\RCD$ spaces are weighted Riemannian manifolds $(M, \dist_{g_M}, e^{-\phi}\mathrm{vol}_{g_M})$ with Bakry-\'Emery Ricci curvature bounded below, and in their measured Gromov-Hausdorff limit spaces, where $\dist_{g_M}$ denotes the induced distance by $g_M$. In fact the Gaussian space $(\mathbb{R}^n, \dist_{\mathbb{R}^n}, e^{-|x|^2/2}\mathcal{L}^n)$ is an $\RCD(1, \infty)$ space, but it is not finite dimensional because of the exponential decay of the weight. On the other hand, if $M$ is closed, then $(M, \dist_{g_M}, e^{-\phi}\mathrm{vol}_{g_M})$ is always a finite dimensional $\RCD$ space.

Let us consider a map between two $\RCD$ spaces $(X, \dist_X, \meas_X), (Y, \dist_Y, \meas_Y)$
\begin{equation}
f:X \to Y.
\end{equation}
Then the main difficulties to establish the above are:
\begin{enumerate}
\item  When we want to find a good definition of the energy density $|\dist (\Phi \circ f)|^2$ along a similar way in this setting, we do not know a nice isometric embedding result as Nash's one.
\item When we want to find a good definition of the pull-back $f^*g_Y$, although the Riemannian metrics $g_X, g_Y$ are still well-defined in a weak sense, they make sense up to negligible sets. In particular we do not know how to define the pull-back $f^*g_Y$ when $f(X)$ is $\meas_Y$-negligible.
\end{enumerate}
In order to overcome these difficulties, we adopt the \textit{heat kernel embedding $\Phi_t$ of $Y$ into $L^2$} as discussed below. It is worth pointing out that there are many fundamental works on Sobolev maps from metric measure spaces to metric spaces, for example, \cite{GS} by Gromov-Schoen,  \cite{KS} by Korevaar-Schoen, \cite{jost} by Jost,  \cite{GPS} by Gigli-Pasqualetto-Soultanis, \cite{GT1, GT2} by Gigli-Tyulenev, \cite{Hajlasz} by Hajlasz, and \cite{KuwaeShioya} by Kuwae-Shioya.

Our goal is to introduce a natural energy for Sobolev maps between metric spaces so that the refined theory in the smooth case can be carried over to the non-smooth setting, like bubbling phenomena, rigidity results, geometric  heat flows associated to harmonic maps and their blow-up analysis, Ginzburg-Landau-type approximations just to name a few. The seminal paper by Gromov-Schoen \cite{GS} was pioneering in using harmonic maps into Bruhat-Tits buildings to obtain rigidity results. A feature of the results obtained in the last years are that the target of the maps are Non-Positively Curved (NPC), which is an important class of metric spaces (see nevertheless the results in \cite{breiner} in the case of CAT$(1)$ spaces, i.e. positively curved in the sense of Alexandrov). In this direction, our contribution is to develop a theory of Sobolev maps including all the previous cases and which is very natural from the viewpoint of analysis.

More precisely we will study the asymptotic behavior of
\begin{equation}\label{asbaruabaruus}
(\Phi_t \circ f)^*g_{L^2}
\end{equation}
as $t \to 0^+$. Since the embedding $\Phi_t$ plays the role of a \textit{smoothing} of $(Y, \dist_Y, \mathcal{H}^N)$, it is expected from the asymptotic behavior (\ref{asbaruabaruus}) that up to normalization,  (\ref{asbaruabaruus}) converges to the pull-back $f^*g_Y$. In order to do this, we need to improve regularity results for $\Phi_t$ obtained in \cite{AHPT} by Ambrosio-Portegies-Tewodrose and the first named author. This is a main idea of the paper.

On the other hand the approach provided in the present paper is new even in the smooth setting. Let us introduce the details below.
\subsection{Heat kernel embedding}
Let $(M, g_M)$ be a closed $m$-dimensional Riemannian manifold. Then B\'erard-Besson-Gallot proved in \cite{BerardBessonGallot} that for any $t \in (0, \infty)$ the map $\Phi_t:M \to L^2(M, \mathrm{vol}_{g_M})$ defined by
\begin{equation}\label{asrarwxdsrgsdd}
\Phi_t(x):=\left( y \mapsto p_M(x, y, t)\right)
\end{equation} 
is a smooth embedding with the following asymptotic expansion
\begin{equation}\label{asymptoti}
c_mt^{(m+2)/2}\Phi_t^*g_{L^2}=g_M+\frac{2t}{3}\left( \mathrm{Ric}^{g_M}-\frac{1}{2}\mathrm{Scal}^{g_M}g_M\right) +O(t^2),\quad (t \to 0^+),
\end{equation}
where $p_M(x, y, t)$ denotes the heat kernel of $(M, g_M)$ and $c_m=4(8\pi)^{m/2}$.
Since (\ref{asymptoti}) is satisfied uniformly on $M$ (cf. \cite{HondaZhu}), in particular, letting
\begin{equation}\label{renormalized}
\tilde{\Phi}_t:=c_m^{1/2}t^{(m+2)/4}\Phi_t,
\end{equation}
we have
\begin{equation}\label{manifoldasym}
\|g_M-\tilde{\Phi}_t^*g_{L^2}\|_{L^{\infty}} \to 0
\end{equation}
which means that $\tilde{\Phi}_t$ is \textit{almost} isometric when $t$ is small.

Next let us introduce a finite dimensional reduction of the above observation. For that, we denote by
\begin{equation}
0=\lambda_0^{M}\le \lambda_1^{M}\le \cdots \to \infty
\end{equation}
the spectrum of the minus Laplacian $-\Delta_{M}f=-\langle \mathrm{Hess}_f^{g_M}, g_M\rangle$ of $(M, g_M)$ counted with multiplicities and denote by $\{\phi_i^{M}\}_i$ corresponding eigenfunctions with $\|\phi_i^{M}\|_{L^2}=1$. Note that standard spectral theory proves that $\{\phi_i^{M}\}_i$ is an $L^2$-orthonormal basis of $L^2(M, \mathrm{vol}_{g_M})$. For any $l \in \mathbb{N}$, let us denote by $\tilde{\Phi}_t^l:M \to \mathbb{R}^l$ the truncated map of $\tilde{\Phi}_t$ by $\{\phi_i^{M}\}_{i=1}^l$ defined by the composition of the following maps
\begin{equation}\label{composition}
\tilde{\Phi}_t^l: M \to L^2(M, \mathrm{vol}_{g_M})=\bigoplus_i\mathbb{R}\phi_i^{M} \stackrel{\pi_l}{\to} \bigoplus_{i=1}^l\mathbb{R}\phi_i^{M} \simeq \mathbb{R}^l,
\end{equation}
where $\pi_l$ is the canonical projection.
It follows from a direct calculation that 
\begin{equation}
\tilde{\Phi}_t^l(x):=\left( c_m^{1/2}t^{(m+2)/4}e^{-\lambda_i^Mt}\phi_i^M(x)\right)_{i=1}^l.
\end{equation}
Under this notation, Portegies proved in \cite{P} that for any $\epsilon \in (0, 1)$ there exists $t_0 \in (0, 1)$ such that for any $t \in (0, t_0]$ there exists $l_0 \in \mathbb{N}$ such that the following hold for any $l \in \mathbb{N}_{\ge l_0}$.
\begin{itemize}
\item The map $\tilde{\Phi}_t^l$ is a smooth embedding.
\item For any $x \in M$ there exists $r \in (0, 1)$ such that $\tilde{\Phi}_t^l|_{B_r(x)}$ is a $(1\pm \epsilon)$-bi-Lipschitz embedding, that is, 
\begin{equation}
(1-\epsilon)\dist_{g_M}(y, z)\le \left|\tilde{\Phi}_t^l(y)-\tilde{\Phi}_t^l(z)\right|_{\mathbb{R}^l}\le (1+\epsilon)\dist_{g_M}(y, z), \quad \forall y,\,\,\forall z \in B_r(x).
\end{equation}
\end{itemize} 
More precisely he established a quantitative version of these results (see also Remark \ref{remarkporte}).

From now on let us discuss on the non-smooth analogue of the above observation. Let us fix a finite dimensional compact $\RCD$ space $(X, \dist_X, \meas_X)$. Then it is proved in \cite{AHPT} that the following hold for any $t \in (0, \infty)$;
\begin{enumerate}
\item the map $\Phi_t:X \to L^2(X, \meas_X)$ and the pull-back $\Phi_t^*g_{L^2}$ are well-defined:
\item the map $\Phi_t$ is Lipschitz and a homeomorphism onto its image $\Phi_t(X)$;
\item for any $p \in [1, \infty)$ we have
\begin{equation}\label{mainasymp}
\|\tilde{c}_mt\meas_X(B_{\sqrt{t}}(\cdot))\Phi_t^*g_{L^2}-g_X\|_{L^p} \to 0,\quad (t \to 0^+),
\end{equation}
where $\tilde{c}_m:=\omega_m^{-1} \cdot c_m$ and $\omega_m$ denotes the volume of a unit ball in $\mathbb{R}^m$.
\end{enumerate}
Then it is natural to ask the following.
\begin{enumerate}
\item[(Q1)] Can (\ref{mainasymp}) be improved to the case when $p=\infty$, that is, 
\begin{equation}
\|\tilde{c}_mt\meas_X(B_{\sqrt{t}}(\cdot))\Phi_t^*g_{L^2}-g_X\|_{L^{\infty}} \to 0,\quad (t \to 0^+)?
\end{equation}
\item[(Q2)] Is $\Phi_t$ a bi-Lipschitz embedding?
\end{enumerate}
However it is shown in \cite{AHPT} that both questions (Q1) and (Q2) have negative answers.
In fact, for example, the metric measure space $([0, \pi], \dist_{[0, \pi]}, \mathcal{H}^1)$, which is a \textit{non-collapsed $\RCD(0, 1)$ space}, satisfies that $\Phi_t^{-1}$ is not Lipschitz for any $t \in (0, \infty)$ and that 
\begin{equation}\label{lowerpositi}
\lim_{t \to 0^+}\|\tilde{c}_1t\mathcal{H}^1(B_{\sqrt{t}}(\cdot))\Phi_t^*g_{L^2}-g_X\|_{L^{\infty}}=\lim_{t \to 0^+}\|\tilde{c}_1t^{3/2}\Phi_t^*g_{L^2}-g_X\|_{L^{\infty}}>0
\end{equation}
is satisfied.

Therefore let us ask the following.
\begin{enumerate}
\item[(Q3)] When do (Q1) and (Q2) have positive answers?
\end{enumerate}
The first main result of the paper is to give a complete answer to this question (Q3). It is worth pointing out that if $(X, \dist_X, \meas_X)$ is \textit{non-collapsed}, that is, it is an $\RCD(K, m)$ space for some $K \in \mathbb{R}$ and some $m \in \mathbb{N}$ with $\meas_X=\mathcal{H}^m$ (thus ``non-collapsed'' always implies the finite dimensionality), then, under the same notation as in (\ref{renormalized}),  (\ref{mainasymp}) is equivalent to
\begin{equation}
\|g_X-\tilde{\Phi}_t^*g_{L^2}\|_{L^p} \to 0, \quad (t \to 0^+)
\end{equation}
because the Bishop and Bishop-Gromov inequality imply the $L^q$-strong convergence of $\mathcal{H}^m(B_r(\cdot))/\omega_mr^m$ to $1$ as $r \to 0^+$ for any $q \in [1, \infty)$.
Compare with (\ref{manifoldasym}).
The following is a main result of the paper.
\begin{theorem}[$=$Theorem \ref{prop:portegies}]\label{themequivalence}
Let $(X, \dist_X, \mathcal{H}^m)$ be a non-collapsed compact $\RCD$ space. Then the following four conditions are equivalent.
\begin{enumerate}
\item We have 
\begin{equation}\label{linfty}
\|g_X -\tilde{\Phi}_t^*g_{L^2}\|_{L^{\infty}} \to 0,\quad (t \to 0^+).
\end{equation}
\item We have 
\begin{equation}\label{linfty2}
\|g_X -\tilde{c}_mt\mathcal{H}^m(B_{\sqrt{t}}(\cdot))g_t\|_{L^{\infty}} \to 0,\quad (t \to 0^+).
\end{equation}
\item For any sufficiently small $t \in (0, 1)$, $\Phi_t$ is a bi-Lipschitz embedding. More strongly, for any $\epsilon \in (0, 1)$ there exists $t_0 \in (0, 1)$ such that $\tilde{\Phi}_t$ is a locally $(1 \pm \epsilon)$-bi-Lipschitz embedding for any $t \in (0, t_0]$.
\item For any sufficiently small $t \in (0, 1)$, $\Phi_t^l$ is a bi-Lipschitz embedding for any sufficiently large $l$. More strongly, for any $\epsilon \in (0, 1)$ there exists $t_0 \in (0, 1)$ such that for any $t \in (0, t_0]$ there exists $l_0 \in \mathbb{N}$ such that $\tilde{\Phi}_t^l$ is a locally $(1 \pm \epsilon)$-bi-Lipschitz embedding for any $l \in \mathbb{N}_{\ge l_0}$.
\end{enumerate}
\end{theorem}
Let us emphasize that this result not only give a complete answer to (Q3), but also provides a complete relationship between (Q1) and (Q2). Moreover under assuming that (1) is satisfied (thus (2), (3) and (4) are satisfied) in the theorem, we will be able to prove that $X$ has no singular set (Proposition \ref{prop:regular}). In particular the intrinsic Reifenberg theorem \cite{CheegerColding1} by Cheeger-Colding allows us to prove that $X$ is bi-H\"older homeomorphic to a closed Riemannian manifold. For this reason, let us say that $(X, \dist_X, \mathcal{H}^m)$ is \textit{weakly smooth} if (1) in Theorem \ref{themequivalence} is satisfied (Definition \ref{weaklysmoothdef}).

Recall (\ref{lowerpositi}) with the fact that the singular set of $([0, \pi], \dist_{[0, \pi]}, \mathcal{H}^1)$ is $\{0, \pi\}$. Thus Theorem \ref{themequivalence} reproves (\ref{lowerpositi}). Although this is stated only for non-collapsed $\RCD$ spaces, we will give similar bi-Lipschitz properties of $\Phi_t$ for general finite dimensional compact $\RCD$ spaces in the appendix \ref{biliplp}, of independent interest.

Using (a weaker form of) Theorem \ref{themequivalence}, we will establish the desired energy as discussed in subsection \ref{1.1}. Let us explain them in the next section.
\subsection{Energy via heat kernel embedding}
Let us fix a finite dimensional (not necessary compact) $\RCD$ space $(X, \dist_X, \meas_X)$ and a finite dimensional compact $\RCD$ space $(Y, \dist_Y, \meas_Y)$.
In this paper a Borel map $f:X \to Y$ is said to be \textit{weakly smooth} if $\phi \circ f$ is a $H^{1, 2}$-Sobolev function on $(X, \dist_X, \meas_X)$ for any eigenfunction $\phi$ of $(Y, \dist_Y, \meas_Y)$. For such a map $f$, we define the \textit{approximate energy}, denoted by $\mathcal{E}_{X, Y, t}(f)$, by
\begin{equation}
\mathcal{E}_{X, Y, t}(f):=\frac{1}{2}\int_X\langle (\Phi_t \circ f)^*g_{L^2(Y, \meas_Y)}, g_X\rangle \di \meas_X.
\end{equation}
We say that $f$ is a \textit{$0$-Sobolev map} if 
\begin{equation}
\limsup_{t \to 0^+}\int_Xt\meas_Y(B_{\sqrt{t}}(f(\cdot )))\langle (\Phi_t \circ f)^*g_{L^2(Y, \meas_Y)}, g_X\rangle \di \meas_X<\infty.
\end{equation}
It is expected from Theorem \ref{themequivalence} that this notion plays a nonlinear analogue of Sobolev functions (at least in the case when $(Y, \dist_Y, \meas_Y)$ is non-collapsed).

On the other hand, as mentioned in the first section, it is well-known that there is a notion of Sobolev maps from a metric measure space to a metric space (see Definition \ref{defsoble} for the precise definition we will adopt).
Therefore it is natural to ask;
\begin{enumerate}
\item[(Q4)] Is there any relationship between $0$-Sobolev maps and Sobolev maps?
\end{enumerate}
The second main result of the paper gives an answer to this question (Q4) under assuming a kind of weak smoothness of the image $f(X)$.
In order to simplify our explanation, we here introduce the result under stronger assumptions (\ref{stronger}) or (\ref{absolutecont}). See Theorems \ref{propcompatr} and \ref{asbairanwi} for more general (localized) results.
\begin{theorem}\label{thm:cheeger}
Let $(X, \dist_X, \meas_X)$ be a finite dimensional $\RCD$ space and let $(Y, \dist_Y, \mathcal{H}^n)$ be a non-collapsed compact $\RCD$ space. Let $f:X \to Y$ be a Borel map. Assume that either
\begin{equation}\label{stronger}
\|g_Y -\tilde{\Phi}_t^*g_{L^2}\|_{L^{\infty}} \to 0,\quad (t \to 0^+).
\end{equation}
or
\begin{equation}\label{absolutecont}
f_{\sharp}\meas_X \ll \mathcal{H}^n
\end{equation}
holds. Then the following two conditions are equivalent.
\begin{enumerate}
\item $f$ is a $0$-Sobolev map.
\item $f$ is a Sobolev map.
\end{enumerate}
Moreover if (1) holds (thus (2) also holds), then the sequence $(\tilde{\Phi}_t\circ f)^*g_{L^2}$ $L^1$-converges to a tensor $f^*g_Y$, called the \textit{pull-back by $f$}, as $t \to 0^+$
and that  $f$ is a Lipschitz-Lusin map with
\begin{equation}
G_f(x)=\mathrm{Lip} \left( f|_D\right)(x), \quad \text{for $\meas_X$-a.e. $x \in D$}
\end{equation}
whenever the restriction of $f$ to a Borel subset $D$ of $X$ is Lipschitz, where $G_f$ is the minimal $2$-weak upper gradient of $f$ (see Definition \ref{defsoble}) and $\mathrm{Lip}$ denotes the local slope (see (\ref{localslope})).
Furthermore for $\meas_X$-a.e. $x \in X$, $G_f^2$ coincides with the best bound of $f^*g_Y$ as a bilinear form.
\end{theorem}
We will prove some compactness results for such Sobolev maps (Theorems \ref{propcompact} and \ref{corcompactness}).
Note that a map is said to be \textit{Lipschitz-Lusin} if there exists a sequence of Borel subsets $D_i$ such that the complement of the union of $\{D_i\}_i$ is null with respect to the reference measure and that the restriction of the map to each $D_i$ is Lipschitz. See Definition \ref{def:lipschitzlusin}.
It is worth pointing out that this theorem can be regarded as a nonlinear analogue of Cheeger's differentiability theorem \cite{Cheeger} which states that for a PI metric measure space $(Z, \dist_Z, \meas_Z)$ (that is, a Poincar\'e inequality and the volume doubling conddition are satisfied), any Sobolev function $f:Z \to \mathbb{R}$ is Lipschitz-Lusin with 
\begin{equation}
|\nabla f|=\mathrm{Lip}\left( f|_D\right), \quad \text{for $\meas_Z$-a.e. $x \in D$}
\end{equation}
whenever the restriction of $f$ to a Borel subset $D$ of $Z$ is Lipschitz, where $|\nabla f|$ is the minimal relaxed slope of $f$ (or equivalently, the minimal $2$-weak upper gradient of $f$).

Theorem \ref{thm:cheeger} allows us to define the \textit{energy density} of such a $f$ by
\begin{equation}
e_Y(f):=\langle f^*g_Y, g_X\rangle
\end{equation}
with the \textit{energy}
\begin{equation}
\mathcal{E}_{X, Y}(f):=\frac{1}{2}\int_Xe_Y(f)\di \meas_X.
\end{equation}
Note that it also follows from Theorem \ref{thm:cheeger} that $G_f \le e_Y(f) \le \sqrt{m}G_f$ holds for $\meas_X$-a.e. on $X$, where $m$ denotes the essential dimension of $(X, \dist_X, \meas_X)$ defined by Bru\`e-Semola in \cite{BrueSemola} (see also Theorem \ref{thmRCD decomposition}).

Let us recall here that there is a well-known canonical energy, so-called \textit{Korevaar-Schoen energy}, defined in \cite{KS} (see Definition \ref{defksenery}). Thus it is natural to ask the following.
\begin{enumerate}
\item[(Q5)] Does the energy $\mathcal{E}_{X, Y}(f)$ coincides with the Korevaar-Schoen's one?
\end{enumerate}
Under the same assumptions as in Theorem \ref{thm:cheeger}, we can also prove the following compatibility result which gives a complete answer to (Q5).
\begin{theorem}[Compatibility with the Korevaar-Schoen energy]\label{ksenergycompatible}
Under the same assumptions as in Theorem \ref{thm:cheeger}, for any ($0$-) Sobolev map $f:X \to Y$, the energy $\mathcal{E}_{X, Y}(f)$ coincides with the Korevaar-Schoen energy $\mathcal{E}_{X, Y}^{KS}(f)$ up to multiplying by a dimensional positive constant.
\end{theorem}
See Theorem \ref{koseergy} for a more general statement.

Let us here emphasize that one of the advantages using heat kernel embeddings (instead of using Nash's one in the smooth setting) is that the embedding map $\Phi_t$ behaves nicely with respect to measured Gromov-Hausdorff convergence as discussed in \cite{AHPT}.
In fact we will study behaviors of energies with respect to the measured Gromov-Hausdorff convergence (Theorems \ref{prop:lowersemiapp}, \ref{hybfgaty} and \ref{thmcompacatara}). In a forthcoming work, we will fully exploit Theorem \ref{ksenergycompatible} to prove existence and bubbling phenomena for harmonic maps between RCD spaces. An important issue in harmonic map theory is their regularity  (see e.g. \cite{ZZ} for an optimal result between Alexandrov spaces). We plan to address this issue in our framework too. We note that the recent papers by Mondino-Semola \cite{MondinoSemola} and Gigli \cite{gigli22} proved Lipschitz regularity of harmonic maps into CAT$(0)$ spaces, whenever the energy is the Korevaar-Schoen one. 

On the other hand by Theorem \ref{thm:cheeger}, we can define that $f$ is an \textit{isometric immersion into $Y$} if 
\begin{equation}\label{isometricimme}
f^*g_Y=g_X.
\end{equation} 
It is worth pointing out that in general, the equality (\ref{isometricimme}) does not imply the local bi-Lipschitz embeddability of $f$, which is a different point from the smooth setting (Remark \ref{asahoarbasj}). However under assuming some regularity for the map $f$, we can realize such a bi-Lipschitz embeddability from (\ref{isometricimme}) (Corollary \ref{prop:iso2}). This observation leads us to study \textit{minimal isometric immersions} from $(X, \dist_X, \meas_X)$ into spheres. Let us explain it in the next section.

\subsection{Minimal isometric immersion into sphere}
Let us recall a fundamental result in submanifold theory, so-called Takahashi's theorem \cite{takahashi}, which states that for a closed $m$-dimensional Riemannian manifold $(M, g_M)$ and a smooth isometric immersion
\begin{equation}
f:M \to \mathbb{S}^k(1):=\{x \in \mathbb{R}^{k+1};|x|=1\},
\end{equation}
the following two conditions are equivalent.
\begin{enumerate}
\item $f$ is minimal (thus it is equivalent to be harmonic because of $f^*g_{\mathbb{S}^k(1)}=g_M$).
\item $f$ is an eigenmap with the eigenvalue $m$, that is, $\Delta_Mf_i+mf_i\equiv 0$ holds for any $i=1, \ldots, k+1$, where $f=(f_i)_i$.
\end{enumerate}
Let us generalize this to $\RCD$ spaces as follows. A Borel map $f$ from a finite dimensional compact $\RCD$ space $(X, \dist_X, \meas_X)$ to $\mathbb{S}^k(1)$ is said to be a \textit{minimal isometric immersion} if it is a $0$-Sobolev map (or equivalently, Sobolev map by Theorem \ref{absolutecont}) with $f^*g_Y=g_X$ and 
\begin{equation}
\frac{\di}{\di t}\Big|_{t=0}\mathcal{E}_{X, \mathbb{S}^k(1)}(f_t)=0
\end{equation}
for any map $(-\epsilon, \epsilon) \times X \to \mathbb{S}^k(1), (t, x) \mapsto f_t(x)=(f_{t, i}(x))_i$, satisfying that $f_0=f$ holds, that $f_{t, i}$ is in the $H^{1, 2}$-Sobolev space of $(X, \dist_X, \meas_X)$ holds for all $t, i$ and that the map $t \mapsto f_{t, i}$ is differentiable at $t=0$ in $H^{1, 2}$.

The third main result is the following which gives a generalization of Takahashi's theorem to the $\RCD$-setting.
\begin{theorem}[=Theorem \ref{theoremtakahashi}]\label{takahashitheorem}
Let $(X, \dist_X, \meas_X)$ be a finite dimensional compact $\RCD$ space whose essential dimension is $m$. For any map $f:X \to \mathbb{S}^k(1)$, the following two conditions are equivalent.
\begin{enumerate}
\item $f$ is a minimal isometric immersion.
\item We see that $\meas_X=c\mathcal{H}^m$ for some $c \in (0, \infty)$, that $(X, \dist_X, \mathcal{H}^m)$ is a finite dimensional non-collapsed $\RCD$ space and that $f$ is an eigenmap with $\Delta_Xf_i+mf_i=0$ for any $i$.
\end{enumerate}
In particular if the above conditions (1) and (2) hold, then $X$ is bi-H\"older homeomorphic to an $m$-dimensional closed manifold, $f$ is $1$-Lipschitz and for all $\epsilon \in (0,1)$ and $x \in X$, there exists $r \in (0, 1)$ such that $f|_{B_r(x)}$ is a $(1 \pm \epsilon)$-bi-Lipschitz embedding.
\end{theorem}
In the next section we will explain how to achieve these results.
\subsection{Outline of proofs}
Let us first introduce a sketch of the proof of Theorem \ref{themequivalence}.  
We first assume that (1) holds. Fixing a small $t \in (0, 1)$, we can find a large $l \in \mathbb{N}$ such thats
\begin{equation}
\|g_X-(\tilde{\Phi}_t^l)^*g_{L^2}\|_{L^{\infty}}
\end{equation} 
is small, where we recall (\ref{composition}) for the definition of the truncated map $\tilde{\Phi}_t^l$. Then we can use blow-up arguments  as in \cite{honda20} for the map $\tilde{\Phi}_t^l$, based on stability results proved in \cite{GMS} by Gigli-Mondino-Savar\'e and in \cite{AmbrosioHonda, AmbrosioHonda2} by Ambrosio and the first named author, to conclude that (3) and (4) hold.

Next assume that (3) of Theorem \ref{themequivalence} holds. Fix a point $x \in X$, a small $t \in (0, 1)$ and take a tangent cone $T_xX$ at $x$ of $X$. Consider a blow-up map $\overline{\Phi}:T_xX \to \ell^2$ of the map $\tilde{\Phi}_t:X \to L^2 \simeq \ell^2$.
A key step is to prove
\begin{itemize}
\item[($\star$)] $T_xX$ is isometric to $\mathbb{R}^m$ and  $\overline{\Phi}$ is a linear map.
\end{itemize}
Then the quantitative version of this observation allows us to prove that (1) holds. 

In order to prove ($\star$), we apply a \textit{blow-down argument} on the tangent cone, which is similar to that in \cite{ChCM} by Cheeger-Colding-Minicozzi. Take a tangent cone \textit{at infinity} $Z$ of $T_xX$ and a blow down map $\underline{\Phi}:Z \to \ell^2$ of $\overline{\Phi}$. Then thanks to \textit{the mean value theorem at infinity} for bounded  subharmonic functions proved in \cite{HKX} by Hua-Kell-Xia (which is a generalization of a result of Li \cite{Li}), we know that $\underline{\Phi}$ is a linear map. Moreover since $\Phi_t$ is a bi-Lipschitz embedding, we see that $\overline{\Phi}$ and $\underline{\Phi}$ are also bi-Lipschitz embeddings. Applying the splitting theorem proved in \cite{Gigli13} by Gigli with the bi-Lipschitz property of $\underline{\Phi}$ shows that $Z$ is isometric to a Euclidean space. The non-collapsed condition yields that the dimension of the Euclidean space is equal to $m$, thus, $Z$ is isometric to $\mathbb{R}^m$.
Then the volume convergence with the Bishop inequality proved in \cite{DG} by DePhilippis-Gigli yields
\begin{equation}
\mathcal{H}^m(B_r(z))=\omega_mr^m,\quad \forall x \in T_xX,\,\,\forall r \in (0, \infty).
\end{equation}
Thus the rigidity of the Bishop inequality given in \cite{DG} proves that $T_xX$ is isometric to $\mathbb{R}^m$. Finally since each $\overline{\Phi}$ is a linear growth harmonic map on $\mathbb{R}^m$ because of the stability results proved in \cite{AmbrosioHonda2} (see also Theorem \ref{spectral2}), it is actually a linear map. Thus we have ($\star$). Therefore as explained above, (1) holds.
Similarly we can prove the implication from (4) to (1).

Finally we assume that (2) holds. Then it follows by a similar blow-up argument with the Reifenberg flatness of $(X, \dist_X)$ that $\mathcal{H}^m(B_r(\cdot))/(\omega_mr^m)$ uniformly converge to $1$ as $r \to 0$.
In particular (1) holds, which completes the proof of Theorem \ref{themequivalence}.

Let us give a worth recording remark. When we will justify the above arguments, in particular ($\star$), we will actually prove a more general rigidity result, which is new even for Riemannian manifolds.
\begin{itemize}
\item[($\star \star$)] If a non-collapsed $\RCD$ space with non-negative Ricci curvature has a bi-Lipschitz embedding into $\ell^2$ by an harmonic map, then the space is isometric to a Euclidean space.
\end{itemize}
This result ($\star \star$) should be compared with a result of Greene-Wu \cite{GW} which states that \textit{any} open (that is, complete and non-compact) Riemannian manifold has a smooth embedding into a Euclidean space by a harmonic map.
See Corollary \ref{lemlemelem} for the proof (see also Theorem \ref{prop:rigiditynonnegative}).

Next let us introduce a sketch of the proof of Theorem \ref{thm:cheeger}.
It follows from the Gaussian estimate for the heat kernel proved in \cite{JiangLiZhang} by Jiang-Li-Zhang with an argument as in \cite{AHPT} that the implication from (2) to (1) is always true without the assumptions (\ref{stronger}), (\ref{absolutecont}).

For the proof of the converse implication, we adopt a blow-up argument for the map $f$. In order to simplify our explanation, let us assume that (1) with (\ref{stronger}) holds. Then fix $x \in X$ and take tangent cones $T_xX, T_{f(x)}Y$ at $x \in X, f(x) \in Y$, respectively. Consider a blow-up map $f^0:T_xX\to T_{f(x)}Y$ of $f$. With no loss of generality we can assume that $x$ is a regular point, that is, $T_xX$ is isometric to $\mathbb{R}^m$, where $m$ denotes the essential dimension of $(X, \dist_X, \meas_X)$. Since $Y$ has no singular points, $T_{f(x)}Y$ is isometric to $\mathbb{R}^n$. Moreover applying Cheeger's differentiability theorem \cite{Cheeger} to the map $\tilde{\Phi}_t\circ f$ shows that the composition 
\begin{equation}
\mathbb{R}^m \simeq T_xX \stackrel{f^0}{\to} T_{f(x)}Y \simeq \mathbb{R}^n \stackrel{\overline{\Phi}}{\to} \ell^2
\end{equation}
is a linear map and that $\overline{\Phi}$ is also a linear map, where $\overline{\Phi}$ is a blow-up map of $\tilde{\Phi}_t$ at $f(x)$ as discussed in (the above sketch of) the proof of Theorem \ref{themequivalence}. Note that $\overline{\Phi}$ is a bi-Lipschitz embedding because of Theorem \ref{themequivalence}. The bi-Lipschitz property of $\overline{\Phi}$ allows us to conclude that $f^0$ is also linear. Applying an approximation by test functions with respect to measured Gromov-Hasudorff convergence proved in \cite{AmbrosioHonda2}, we see that $f^*g_Y$ $L^2_{\mathrm{loc}}$-strongly converge to $(f^0)^*g_{\mathbb{R}^n}$. Since it is not hard to check that $G_f(x)$ is bounded below by the best bound of the bilinear form $f^0g_{\mathbb{R}^n}$ (essentially), we see that (2) holds. We can also prove the remaining statements along similar ways.

Finally let us explain how to achieve Theorem \ref{takahashitheorem}. Assume that (1) holds. Then it follows from the Euler-Lagrange equation that $\Delta_Xf_i+e_Y(f)f_i = 0$ holds. Since $f$ is isometric, we have $e_Y(f)=\langle f^*g_Y, g_X\rangle=|g_X|^2=m$. In particular $f$ is an eigenmap. Then applying a result proved in \cite{honda20} completes the proof of (2).
The converse implication is justified by a direct calculation.

\subsection{Plan of the paper}

The paper is organized as follows: section \ref{prel} collects notations, preliminary results and terminology on $\RCD$ spaces, giving technical new results. 
Section \ref{sec:app} deals with approximate Sobolev maps for fixed $t \in (0, \infty)$, which plays a key role later when we prove Theorem \ref{thm:cheeger}. 
In section \ref{secembed}, the bi-Lipschitz embeddability of the heat kernel embedding into $L^2$ is established. In pariticular we prove Theorem \ref{themequivalence}.
Combining results obtained in sections \ref{sec:app} and \ref{secembed}, we study the behavior of $t$-Sobolev maps as $t \to 0^+$ in section \ref{asympto}. Then Theorem \ref{thm:cheeger} is proved.
Section \ref{takahashitheoremsec} provides a proof of Theorem \ref{takahashitheorem}. In section \ref{secenergies} we discuss the behavior of Sobolev maps with respect to measured Gromov-Hausdorff convergence. In particular compactness results (Theorems \ref{prop:lowersemiapp}, \ref{hybfgaty} and \ref{thmcompacatara}) are proved.
In the final section, section \ref{biliplp}, we give several generalizations of the bi-Lipschitz embeddability results given in section \ref{secembed} to general $\RCD$ space as independet interests.

\smallskip\noindent
\textbf{Acknowledgement.}
The first author acknowledges supports of Grant-in-Aid for Scientific Research (B) of 18H01118, of 20H01799, of 21H00977 and Grant-in-Aid for Transformative Research Areas (A) of 22H05105. The second author is partially supported by the Simons foundation and NSF under DMS grant $2154219$, " Regularity {\sl vs} singularity formation in elliptic and parabolic equations". 
{\color{blue}The authors thank the anonymous referee for the very careful review and several helpful suggestions that improved the presentation.}

\section{Preliminaries}\label{prel}
Throughout the paper we usually use the notation $C(C_1, C_2, \ldots)$ for a (positive) constant depending only on constants $C_1, C_2, \ldots$, which may change from line to line. Moreover we sometimes use a standard notation in convergence theory
\begin{equation}\label{anasuarbaj}
\Psi(\epsilon_1, \epsilon_2,\ldots, \epsilon_l; c_1, c_2, \ldots, c_m)
\end{equation}
denotes a function $\Psi:(\mathbb{R}_{>0})^l\times \mathbb{R}^m \to (0, \infty)$ satisfying
\begin{equation}
\lim_{(\epsilon_1,\ldots, \epsilon_l) \to 0}\Psi(\epsilon_1, \epsilon_2,\ldots, \epsilon_l; c_1, c_2, \ldots, c_m)=0,\quad \forall c_i \in \mathbb{R}.
\end{equation}
\subsection{Metric notion}
Let us fix two metric spaces $(X, \dist_X), (Y, \dist_Y)$.
For $\epsilon \in (0, 1)$, a map $f$ from $X$ to $Y$ is said to be a \textit{$(1\pm \epsilon)$-bi-Lipschitz embedding} if 
\begin{equation}
(1-\epsilon)\dist_X(x_1, x_2) \le \dist_Y(f(x_1), f(x_2)) \le (1+\epsilon)\dist_X(x_1, x_2),\quad \forall x_i \in X.
\end{equation}
Note that 
\begin{align}\label{texteq}
&\text{if $f$ is a $(1\pm \epsilon)$-bi-Lipschitz embedding,} \nonumber \\
&\text{then $f$ gives a $(\epsilon \cdot \mathrm{diam}(X, \dist_X))$-Gromov-Hausdorff approximation of the image $(f(X), \dist_Y)$.}
\end{align}
See for instance \cite{bbi, fukaya90} for the definition of Gromov-Hausdorff distance.
Let $\phi:X \to Y$ be a Lipschitz map. Define
\begin{itemize}
\item the \textit{best Lipschitz constant} of $\phi$ by
\begin{equation}
\mathbf{Lip} \phi:=\sup_{x_1 \neq x_2}\frac{\dist_Y(\phi(x_1), \phi(x_2))}{\dist_X(x_1, x_2)};
\end{equation} 
\item the \textit{asymptotically Lipschitz constant} at $x$ by 
\begin{equation}
\mathrm{Lip}_a\phi (x):=\lim_{r \to 0^+}\mathbf{Lip}\phi|_{B_r(x)};
\end{equation} 
\item the \textit{local slope} of $\phi$ at $x$ by
\begin{equation}\label{localslope}
{\rm Lip}\phi (x):=\lim_{r\to 0^+}\sup_{y \in B_r(x) \setminus \{x\}}\frac{\dist_Y(\phi (x), \phi(y))}{\dist_X (x, y)}
\end{equation}
if $x$ is not isolated, ${\rm Lip}\phi (x):=0$ otherwise. 
\end{itemize}
\begin{definition}[Metric density point]\label{def:den}
For any subset $A$ of $X$, a point $x \in X$ is said to be a \textit{metric density point of $A$} if for any $\epsilon \in (0, 1)$ there exists $r_0 \in (0, 1)$ such that $B_r(x) \cap A$ is $\epsilon r$-dense in $B_r(x)$ {\color{blue}(namely the closed $\epsilon r$-neighborhood of $B_r(x) \cap A$ includes $B_r(x)$)} for any $r \in (0, r_0)$. We denote by $\mathrm{Den} (A)$ the set of all density points of $A$. 
\end{definition}
Since it is easy to check the following proposition, we skip the proof;
\begin{proposition}\label{prop:den}
Let $(X, \dist_X), (Y, \dist_Y)$ be metric spaces, let $A$ be a subset of $X$ and let $f:X \to Y$ be a Lipschitz map. Then for any $x \in \mathrm{Den} (A)$ we have
\begin{equation}
{\rm Lip}\phi (x)={\rm Lip}(\phi |_A) (x),\quad \mathrm{Lip}_a\phi(x)=\mathrm{Lip}_a(\phi |_A)(x).
\end{equation}
\end{proposition}
We denote by $\mathcal{H}^N_{\dist}$, or simply $\mathcal{H}^N$, the $N$-dimensional Hausdorff measure of $(X, \dist)$. Finally we denote by $\mathrm{Lip}(X, \dist_X)$, ($\mathrm{Lip}_c(X, \dist_X)$, respectively), the set of all Lipschitz functions on $X$ (the set of all Lipschitz functions on $X$ with compact support, respectively). 
\subsection{$\RCD$ space}\label{rcddef1}
A triple $(X, \dist_X, \meas_X)$ is said to be a \textit{metric measure space} if $(X, \dist_X)$ is a complete separable metric space and $\meas_X$ is a Borel measure on $X$ with full support. We fix a metric measure space $(X, \dist_X, \meas_X)$ below.

Define the \textit{Cheeger energy} $\Ch:L^2(X,\meas)\to [0,\infty]$ by
\begin{equation}\label{eq:defchee}
\Ch(f):=\inf\left\{\liminf_{i\to\infty}\int_X{\rm Lip}^2f_i\di\meas_X;\ f_i\in\Lip (X,\dist_X)\cap (L^2 \cap L^{\infty})(X, \meas_X),\,\,\,\|f_i-f\|_{L^2}\to 0
\right\}.
\end{equation}
Then the \textit{Sobolev space} $H^{1, 2}=H^{1,2}(X,\dist_X,\meas_X)$ is defined as the finiteness domain of $\Ch$ in $L^2(X, \meas_X)$ and it is a Banach space equipped with the norm $\|f\|_{H^{1, 2}}=\sqrt{\|f\|_{L^2}^2+\Ch(f)}$.
We are now in position to introduce the definition of $\RCD(K, N)$ space. 
\begin{definition}[$\RCD(K, N)$ space]\label{def:rcd}
$(X, \dist_X, \meas_X)$ is said to be an $\RCD(K, N)$ \textit{space} for some $K \in \mathbb{R}$ and some $N \in [1, \infty]$ if the following four conditions hold.
\begin{itemize}
\item{(Volume growth condition)} There exist $C\in (0, \infty)$ and $x \in X$ such that $\meas_X (B_r(x)) \le Ce^{Cr^2}$ holds for any $r \in (0, \infty)$.
\item{(Infinitesimally Hilbertian condition)} $H^{1, 2}$ is a Hilbert space. In particular for all $f_i \in H^{1, 2} (i=1, 2)$,
\begin{equation}
\langle \nabla f_1, \nabla f_2\rangle:=\lim_{t \to 0}\frac{|\nabla (f_1+tf_2)|^2-|\nabla f_1|^2}{2t} \in L^1(X, \meas_X)
\end{equation}
is well-defined, where $|\nabla f_i|$ denotes the minimal relaxed slope of $f_i$ (e.g. \cite[Def. 4.2]{AmbrosioGigliSavare13}).
\item{(Sobolev-to-Lipschitz property)} Any function $f \in H^{1, 2}$ satisfying $|\nabla f|(y) \le 1$ for $\meas_X$-a.e. $y \in X$ has a $1$-Lipschitz representative.  
\item{(Bochner inequality)} For any $f \in D(\Delta)$ with $\Delta f \in H^{1, 2}$, we have
\begin{equation}
\frac{1}{2}\int_X\Delta \phi |\nabla f|^2\di \meas_X \ge \int_X \phi \left(\frac{(\Delta f)^2}{N}+\langle \nabla \Delta f, \nabla f\rangle +K|\nabla f|^2\right)\di \meas_X
\end{equation}
for any $\phi \in D(\Delta) \cap L^{\infty}(X, \meas_X)$ with $\Delta \phi \in L^{\infty}(X, \meas_X)$ and $\phi \ge 0$, where
\begin{equation}
D(\Delta):=\left\{ f \in H^{1, 2}; \exists h=:\Delta f \in L^2, \,\mathrm{s.t.} \, \int_X\langle \nabla f, \nabla \psi \rangle \di \meas_X =-\int_Xh\psi \di \meas_X,\,\forall \psi \in H^{1, 2}\right\}.
\end{equation}
\end{itemize}
\end{definition}
We will sometimes use the notation $\Delta_X$ instead of using the simpler one $\Delta$ as above when we need to clarify the space $(X, \dist_X, \meas_X)$.
For brevity, $(X, \dist_X, \meas_X)$ is said to be a \textit{finite dimensional $\RCD$ space} if it is an $\RCD(K, N)$ space for some $K \in \mathbb{R}$ and some $N \in [1, \infty)$. 

Finally let us recall the \textit{heat flow} associated to the Cheeger energy on an $\RCD(K, \infty)$ space $(X, \dist_X, \meas_X)$
\begin{equation}
h_t:L^2(X, \meas_X ) \to D(\Delta)
\end{equation}
which is determined by satisfying that for any $f \in L^2(X, \meas_X)$, the map $t \mapsto h_tf$ is absolutely continuous and satisfies
\begin{equation}
\frac{\di^+}{\di t}h_tf=\Delta h_tf.
\end{equation}
Note that this map $t \mapsto h_tf$ is actually smooth (see \cite[Prop.5.2.12]{GP2}) and that $h_t$ can be easily extended to a linear continuous map $L^p(X, \meas_X) \to L^p(X, \meas_X)$ with operator norm at most $1$ for any $p \in [1, \infty]$.
Then for any $f \in H^{1, 2}(X, \dist_X, \meas_X) \cap \mathrm{Lip}(X, \dist_X)$, the $1$-Bakry-\'Emery gradient estimate \cite[Cor.4.3]{Savare} holds
\begin{equation}\label{absbayraywsrai}
|\nabla h_tf|(x)\le e^{-Kt}h_t|\nabla f|(x),\quad \text{for $\meas_X$-a.e. $x \in X$}.
\end{equation}
In the sequel, we always denote by $N$ a finite number, and by $K$ a real number.
\subsection{Infinitesimal structure of finite dimensional $\RCD$ space}
Let $(X, \dist_X, \meas_X)$ be a finite dimensional $\RCD$ space with $\mathrm{diam}(X, \dist)>0$. It is known that $(X, \dist)$ is a proper geodesic space. In the paper we omit the notion of convergence of metric measure spaces, for example, \textit{(pointed) Gromov-Hausdorff convergence, (pointed) measured Gromov-Hausdorff convergence} and so on. It is worth pointing out that they are metrizable topologies. Thus ``$\epsilon$-closeness'' makes sense for such convergence.
See for instance \cite{GMS} for our purposes.
\begin{definition}[Tangent cones]
For $x \in X$, we say that a pointed metric measure space $(Y, \dist_Y,\meas_Y,y)$ is said to be a \textit{tangent cone of $(X, \dist_X, \meas_X)$ at $x$} if 
\begin{equation}\label{eq:ulla_ulla}
\left(X, r_i^{-1}\dist_X, \meas_X (B_{r_i}(x))^{-1}\meas_X, x\right) \stackrel{\mathrm{pmGH}}{\to} (Y, \dist_Y,\meas_Y,y) 
\end{equation}
holds for some $r_i\to 0^+$, where pmGH denotes the pointed measured Gromov-Hausdorff convergence (we will use similar notations, mGH, pGH etc. immediately later).
We denote by $\mathrm{Tan}(X, \dist_X,\meas_X,x)$ the set of all tangent cones of $(X,\dist_X,\meas_X)$ at $x$.
\end{definition}
\begin{definition}[Regular set $\mathcal{R}_k$]
For any $k \geq 1$, we denote by $\mathcal{R}_k$ the \textit{$k$-dimensional regular set}  of $(X, \dist_X, \meas_X)$, 
namely, the set of all points $x \in X$ such that
$$
\mathrm{Tan}(X, \dist_X,\meas_X,x) =\left\{ \left(\mathbb{R}^k, \dist_{\mathbb{R}^k},\omega_k^{-1}\mathcal{H}^k,0_k\right) \right\},
$$
where $\omega_k=\mathcal{H}^k(B_1(0_k))=\mathcal{L}^k(B_1(0_k))$. 
\end{definition}
The following result is proved in \cite[Th.0.1]{BrueSemola} after \cite{MondinoNaber} which gives a generalization of \cite[Th.1.12]{ColdingNaber} to finite dimensional $\RCD$ spaces.
\begin{theorem}[Essential dimension of finite dimensional $\RCD$ spaces]\label{thmRCD decomposition}
Let $(X,\dist,\meas)$ be a finite dimensional $\RCD$ space. Then, there exists a unique integer $n\in \mathbb{N}$, called the essential dimension of $(X, \dist, \meas)$, such that
 \begin{equation}\label{eq:regular set is full}
\meas(X\setminus \mathcal{R}_n\bigr)=0
\end{equation}
holds.
\end{theorem}
Combining independent results in \cite[Th.4.11]{DMR}, in \cite[Th.3.5]{GP} and in \cite[Th.1.2]{KM} with \cite[Th.1.1]{MondinoNaber} and Theorem \ref{thmRCD decomposition}, we have the following (see \cite[Th.4.1]{AmbrosioHondaTewodrose}).
\begin{theorem}[Weak Ahlfors regularity and metric measure rectifiability]\label{thm:RN}
Let $(X, \dist_X, \meas_X)$ be a finite dimensional $\RCD$ space whose essential dimension is equal to $n$, put $\meas_X=\theta\mathcal{H}^n\res\mathcal{R}_n$ and set
\begin{equation}\label{eq:defRkstar}
{\mathcal R}_n^*:=\left\{x\in\mathcal{R}_n:\
\exists\lim_{r\to 0^+}\frac{\meas_X(B_r(x))}{\omega_nr^n}\in (0,\infty)\right\}.
\end{equation}
Then $\meas_X(\mathcal{R}_n\setminus\mathcal{R}_n^*)=0$, $\meas\res\mathcal{R}_n^*$ and 
$\mathcal{H}^n\res\mathcal{R}_n^*$ are mutually absolutely continuous and
\begin{equation}\label{eq:gooddensity}
\lim_{r\to 0^+}\frac{\meas_X(B_r(x))}{\omega_nr^n}=\theta(x)
\qquad\text{for $\meas_X$-a.e. $x\in\mathcal{R}_n^*$,}
\end{equation}
\begin{equation}\label{eq:goodlimitRN}
\lim_{r\to 0^+} \frac{\omega_nr^n}{\meas_X(B_{r}(x))}=1_{\mathcal{R}^*_{n}}(x)
\frac{1}{\theta(x)}
\qquad\text{for $\meas_X$-a.e. $x\in X$.}
\end{equation}
Moreover $(X, \dist_X, \meas_X)$ is metric measure rectifiable in the sense that for any $\epsilon \in (0, 1)$ there exist a sequence of Borel subsets $A_i$ of $\mathcal{R}_n^*$ and a sequence of $(1 \pm \epsilon)$-bi-Lipschitz embeddings $\phi_i:A_i \to \mathbb{R}^n$ such that $\meas_X(X \setminus \bigcup_iA_i)=0$ holds. We call such a pair $(A_i, \phi_i)$ a \textit{$(1 \pm \epsilon)$-bi-Lipschitz rectifiable chart of $(X, \dist_X, \meas_X)$}.
\end{theorem}

\subsection{Sobolev spaces and Laplacians on open sets}
Let us introduce the
Sobolev space $H^{1, 2}(U, \dist_X, \meas_X)$ for an open subset $U$ of a finite dimensional $\RCD$ space $(X, \dist_X, \meas_X)$. See also \cite{Cheeger,Shanmugalingam} for the definition of Sobolev space $H^{1, p}(U, \dist_X, \meas_X)$ for any $p \in [1, \infty)$. Our working definition is the following.

\begin{definition}\label{def:reddu}
Let $U\subset X$ be open. 
\begin{enumerate}
\item{($H^{1,2}_0$-Sobolev space)} We denote by $H^{1,2}_0(U, \dist_X, \meas_X )$ the 
$H^{1,2}$-closure of $\Lip_c(U, \dist_X)$.
%the subspace of $\Lip(U,\dist_X)$ of compactly supported functions. 
\item{(Sobolev space on an open set $U$)} We say that $f\in L^2_{\mathrm{loc}}(U,\meas_X)$ belongs to $H^{1,2}_{\rm loc}(U,\dist_X,\meas_X)$ if
$\phi f \in H^{1, 2}(X, \dist_X, \meas_X)$ for any $\phi \in\Lip_c(U, \dist_X)$. If, in addition, $f,  |\nabla f|\in L^2(U,\meas_X)$,
we say that $f\in H^{1,2}(U,\dist_X,\meas_X)$. 
\end{enumerate}
\end{definition}

Notice that $f\in H^{1,2}_{\rm loc}(U,\dist_X,\meas_X)$  if and only if  for any bounded subset $V$ of $U$ with $\overline{V} \subset U$, there exists 
$\tilde f\in H^{1,2}(X,\dist_X,\meas_X)$ with $\tilde f\equiv f$ on $V$.
The global condition $f, |\nabla f|\in L^2(U,\meas_X)$ in the definition of $H^{1,2}(U,\dist_X,\meas_X)$
is meaningful, since the locality properties of the minimal relaxed slope ensure that $|\nabla f|(x)$ makes sense for $\meas_X$-a.e.  $x \in U$  for all functions $f\in H^{1,2}_{\rm loc}(U,\dist_X,\meas_X)$. Indeed, choosing
$\phi_n\in\Lip_c(U, \dist_X)$ with $\{\phi_n=1\}\uparrow U$ and defining
$$
|\nabla f|:=|\nabla(f\phi_n)|\qquad\text{for $\meas_X$-a.e. in $\{\phi_n=1\}$}
$$
we obtain an extension of the minimal relaxed slope to $H^{1,2}(U,\dist_X,\meas_X)$
(for which we keep the same notation, being also $\meas_X$-a.e. independent of the choice of $\phi_n$) 
which retains all bilinearity and locality properties. See also \cite[Th.10.5.3]{HKST} for the compatibility with Korevaar-Schoen type Sobolev spaces (for functions).

%We introduce the Dirichlet Laplacian acting only on $H^{1, 2}_0$-functions as follows:
%\begin{definition}[Dirichlet Laplacian on an open set $U$]\label{def:dlapballs}
%Let $D_0(\Delta, U)$ denote the set of all $f \in H^{1, 2}_0(U, \dist_X, \meas_X)$ such that 
%there exists $h:=\Delta_{U} f\in L^2(U,\meas_X)$ satisfying
%$$\int_U hg\di\meas=-\int_U \langle\nabla f,\nabla g\rangle\di\meas\qquad\forall g\in H^{1,2}_0(U,\dist_X,\meas_X).
%$$
%We also set $\Delta_{x, R}:=\Delta_{B_R(x)}$ when $U=B_R(x)$ for some $x \in X$ and $R >0$.
%\end{definition}
%Strictly speaking, the Dirichlet Laplacian $\Delta_{U}$ should not be confused with the operator $\Delta$, even if the two operators
%agree on functions compactly supported on $U$; for this reason we adopted a distinguished symbol. 
%Notice that $\lambda_1^D(B_R(x))>0$ whenever $\meas(X\setminus B_{R}(x))>0$, as
%a direct consequence of the local Poincar\'e inequality.

\begin{definition}[Laplacian on an open set $U$]
For $f \in H^{1, 2}(U, \dist_X, \meas_X)$, we write $f\in D(\Delta, U)$ if there exists 
$h:=\Delta_{U}f\in L^2(U,\meas_X)$ satisfying
$$
\int_U hg\di\meas_X=-\int_U \langle\nabla f, \nabla g\rangle \di\meas_X \qquad
\forall g\in H^{1,2}_0(U,\dist_X,\meas_X).
$$
We usually use a simplified notation $\Delta f$ instead of using $\Delta_Uf$ if there is no confusion.
\end{definition}

It is easy to check that for any $f \in D(\Delta, U)$ and any 
$\phi \in D(\Delta ) \cap \Lip_c(U, \dist_X)$ with $\Delta \phi \in L^{\infty}(X, \meas_X)$
one has (understanding $\phi\Delta_{U}f$ to be null out of $U$)
$\phi f \in D(\Delta )$ with
\begin{equation}\label{eq:local to global}
\Delta (\phi f)=f\Delta \phi +2\langle\nabla\phi,\nabla f\rangle +\phi \Delta_{U} f
\qquad\text{for $\meas_X$-a.e. in $X$.}
\end{equation}

Such notions allow to define harmonic functions on an open set $U$ as follows.

\begin{definition}
Let $U$ be an open subset of $X$. We say that $f\in H^{1,2}_{\rm loc}(U,\dist_X,\meas_X)$ is harmonic in $U$ if 
$f \in \mathcal{D}(\Delta, V)$ with $\Delta_V f=0$ for any bounded open subset $V$ of $U$ with $\overline{V} \subset U$.
Let us denote by $\mathrm{Harm}(U, \dist_X, \meas_X)$ the set of all harmonic functions on $U$.
\end{definition}

\subsection{Second-order calculus}
We first refer to \cite{Gigli} as a main reference on this section because, in order to keep our presentation short, we assume readers to be familiar with the theory of $L^p$-normed modules, including the \textit{second-order differential calculus on $\RCD$ spaces}, developed in \cite{Gigli}. 
Fix an $\RCD(K, \infty)$ space $(X, \dist_X, \meas_X)$. Recall the set of all \textit{test functions}
\begin{equation}
\mathrm{Test}F(X, \dist_X, \meas_X):=\{f \in D(\Delta) \cap \Lip(X, \dist_X)\cap L^{\infty}(X, \meas_X); \Delta f \in H^{1, 2}(X, \dist_X, \meas_X)\}.
\end{equation}
It is known that $\mathrm{Test}F(X, \dist_X, \meas_X)$ is an algebra with $|\nabla f|^2 \in H^{1, 2}(X, \dist_X, \meas_X)$ for any $f \in \mathrm{Test}F(X, \dist_X, \meas_X)$. 

Fix a Borel subset $A$ of $X$ and denote by $L^p(T(A, \dist_X, \meas_X))$ the set of all $L^p$-vector fields over $A$, where we usually denote by $L^0$ the set of all Borel measurable objects. Note that any element $V \in L^p(T(A, \dist_X, \meas_X))$ can be characterized by a linear map $V:\mathrm{Test}F(X, \dist_X, \meas_X) \to L^0(A, \meas_X)$ with the Leibniz rule for all $f_i \in \mathrm{Test}F(X, \dist_X, \meas_X)$
\begin{equation}
V(f_1f_2)(x)=f_1(x)V(f_2)(x)+f_2(x)V(f_1)(x),\quad \text{for $\meas_X$-a.e. $x \in A$}
\end{equation}
and the inequality for some non-negatively valued $\phi \in L^p(A, \meas_X)$
\begin{equation}
|V(f)(x)| \le \phi (x) |\nabla f|(x),\quad \text{for $\meas_X$-a.e. $x \in A$},
\end{equation}
for any $f \in \mathrm{Test}F(X, \dist_X, \meas_X)$. 
The least $\phi$ is call the \textit{Hilbert-Schmidt norm} of $V$, denoted by $|V|$. Note that the gradient vector field $\nabla f$ of $f \in H^{1, 2}(X, \dist_X, \meas_X)$ is well-defined in $L^2(T(X, \dist_X, \meas_X))$ and that the pointwise norm $| \cdot |$ comes from a pointwise inner product $\langle \cdot, \cdot \rangle$ for $\meas_X$-a.e. sense. Similarly we can define the set of all $L^p$-$1$-forms on $A$, denoted by $L^p(T^*(A, \dist_X, \meas_X))$. It is worth pointing out that there is a canonical isometry $\iota$ from $L^2(T(A, \dist_X, \meas_X))^*$ to $L^2(T^*(A, \dist_X, \meas_X))$, hence $\iota(\nabla f)=\dist f$ for any $f \in H^{1, 2}(X, \dist_X, \meas_X)$, where $\dist f$ is the differential (or the exterior derivative) of $f$ (see \cite[Def.2.2.2]{Gigli}).

Let us recall the \textit{Hilbert-Schmidt norm} and the \textit{(best) bound} of a tensor of type $(0, 2)$. 
\begin{definition}[Norms]\label{def:extralabel}
Let $T:[L^2(T(A,\dist_X,\meas_X))]^2\to L^0(A, \meas_X)$ be a tensor of type $(0, 2)$ over $A$, namely, it is an $L^\infty(A, \meas_X)$-bilinear form. We define the \textit{Hilbert-Schmidt norm} $| \cdot |_{HS}$ and the (best) \textit{bound $|\cdot |_B$} of $T$ as follows.
\begin{enumerate}
\item The smallest Borel measurable function $h:A\to [0,\infty]$, up to $\meas_X$-negligible sets, satisfying
\begin{equation}
\left|\sum_i\chi_iT(\nabla f_i^1, \nabla f_i^2)\right|\leq h\left|\sum_{i, j}\chi_i \chi_j \langle \nabla f_i^1, \nabla f_j^1\rangle \cdot \langle \nabla f_i^2, \nabla f_j^2\rangle \right|^{1/2},
\quad\text{for $\meas_X$-a.e. in $A$}
\end{equation}
for all $\chi_i, f_i^j \in \mathrm{Test}F(X,\dist_X,\meas_X)$, is denoted $|T|_{HS}$ or $|T|$ for short (because we will usually consider the Hilbert-Schmidt norm for given tensor).
\item The smallest Borel measurable function $h:A\to [0,\infty]$, up to $\meas_X$-negligible sets, satisfying
\begin{equation}
\left|\chi T(\nabla f^1, \nabla f^2)\right|\leq h|\chi|\cdot |\nabla f^1| \cdot |\nabla f^2|,
\quad\text{for $\meas_X$-a.e. in $A$}
\end{equation}
for all $\chi, f^j \in \mathrm{Test}F(X,\dist_X,\meas_X)$, is denoted $|T|_{B}$.
\end{enumerate}
\end{definition}
Let us denote by $L^p((T^*)^{\otimes 2}(A, \dist_X, \meas_X))$ the set of all tensors $T$ of type $(0, 2)$ satisfying $|T| \in L^p(A, \meas_X)$. Note that the pointwise Hilbert-Schmidt norm also comes from a pointwise inner product for $\meas_X$-a.e. sense as in the case of vector fields. In particular $L^2((T^*)^{\otimes 2}(A, \dist_X, \meas_X))$ is a Hilbert space.

We need the following important notion, the \textit{Hessian} of a function:
\begin{theorem}[Hessian]
For any $f \in \mathrm{Test}F(X, \dist_X, \meas_X)$ there exists a unique $$T \in L^2((T^*)^{\otimes 2}(X, \dist_X, \meas_X)),$$ called the \textit{Hessian of f}, denoted by $\mathrm{Hess}_f$, such that for all $f_i \in \mathrm{Test}F(X, \dist_X, \meas_X)$,
\begin{equation}\label{eq:hess}
\langle T, \dist f_1\otimes \dist f_2\rangle=\frac{1}{2}\left(\langle \nabla f_1, \nabla \langle \nabla f_2, \nabla f\rangle \rangle + \langle \nabla f_2, \nabla \langle \nabla f_1, \nabla f \rangle \rangle -\langle \nabla f, \nabla \langle \nabla f_1, \nabla f_2\rangle \rangle \right)
\end{equation}
holds for $\meas_X$-a.e. $x \in X$. 
\end{theorem}
See \cite[Th.3.3.8]{Gigli} and \cite[Lem.3.3]{Savare}.
Moreover combining the locality property \cite[Prop.3.3.24]{Gigli} with a good cut-off, it is proved in \cite[Th.3.3.8 and Cor.3.3.9]{Gigli} and \cite[Th.3.4]{Savare} that for any open subset $U$ of $X$, the Hessian is well-defined for any $f \in D(\Delta, U)$ by satisfying (\ref{eq:hess}) for $\meas_X$-a.e. $x \in U$, and that the Bochner inequality involving the Hessian term
\begin{equation}\label{eq:bochner}
\frac{1}{2}\int_X|\nabla f|^2\Delta \phi \di \meas_X \ge \int_X\phi\left( |\mathrm{Hess}_f|^2+\langle \nabla \Delta f, \nabla f\rangle +K|\nabla f|^2\right)\di \meas_X
\end{equation}
holds for any $f \in \mathrm{Test}F(X, \dist_X, \meas_X)$ and $\phi \in D(\Delta)$ with $\phi \ge 0$, $\phi, \Delta \phi \in L^{\infty}(X, \meas_X)$ and $\supp \phi \subset U$.  
Let us define the \textit{Riemannian metric} as follows. See \cite[Prop.3.2]{AHPT} and \cite[Th.5.1]{GP} for the proof.
\begin{proposition}[Riemannian metric]\label{Riemdef}
There exists a unique $g_X \in L^{\infty}((T^*)^{\otimes 2}(X, \dist_X, \meas_X))$ such that for any $f_i \in \mathrm{Test}F(X, \dist_X, \meas_X)$ we have
\begin{equation}
\langle g_X, \dist f_1 \otimes \dist f_2\rangle (x) =\langle \nabla f_1, \nabla f_2\rangle (x),\quad \text{for $\meas_X$-a.e. $x \in X$}. 
\end{equation}
We call $g_X$ the \textit{Riemannian metric of $(X, \dist_X, \meas_X)$}. Moreover it holds that
\begin{equation}
|g_X|(x)=\sqrt{n},\quad \text{for $\meas_X$-a.e. $x \in X$}.
\end{equation}
if $(X, \dist_X, \meas_X)$ is finite dimensional and the essential dimension is equal to $n$.
\end{proposition}
\begin{proposition}\label{prop:ineq}
Assume that $(X, \dist_X, \meas_X)$ is finite dimensional and that the essential dimension is equal to $n$.
Then for any symmetric tensor $T$ of type $(0, 2)$ over $A$, we have
\begin{equation}\label{eq:hsb}
|T|_B(x)\le |T|(x) \le \sqrt{n}|T|_B(x),\quad \text{for $\meas_X$-a.e. $x \in A$}.
\end{equation}
\end{proposition}
\begin{proof}
The conclusion is trivial when $(X, \dist_X, \meas_X)$ is isometric to $(\mathbb{R}^n, \dist_{\mathbb{R}^n}, \mathcal{H}^n)$. Fix $\epsilon \in (0, 1)$ and take a $(1 \pm \epsilon)$-bi-Lipschitz rectifiable chart $(A, \phi)$ of $(X, \dist_X, \meas_X)$. Then since (\ref{eq:hsb}) is satisfied for $\phi(A)$, we conclude because $\epsilon$ is arbitrary (see also \cite[Th.5.1]{GP3}).
\end{proof}
It is worth pointing out that in Proposition \ref{prop:ineq}, thanks to (\ref{eq:hsb}), $|T| \in L^p(A, \meas_X)$ holds if and only if $|T|_B \in L^p(A, \meas_X)$ holds.
\subsection{Non-collapsed $\RCD$ spaces}
Let us recall a nicer subclass of $\RCD$ spaces, so-called \textit{non-collapsed $\RCD$ spaces}, introduced in \cite[Def.1.1]{DG}.
The following definition is motivated by seminal works on non-collapsed Ricci limit spaces in \cite{CheegerColding1, CheegerColding2, CheegerColding3}, in particular in \cite[Th.5.9]{CheegerColding1}. Fix $K \in \mathbb{R}, N \in [1, \infty)$.
\begin{definition}[Non-collapsed $\RCD$ space]
An $\RCD(K, N)$ space $(X, \dist_X, \meas_X)$ is said to be \textit{non-collapsed} if $\meas_X=\mathcal{H}^N$ holds. 
\end{definition}
Non-collapsed $\RCD(K, N)$ spaces have nicer properties over general $\RCD(K, N)$ spaces. Let us introduce some of them: 
\begin{theorem}\label{thm:bishop}
Let $(X, \dist_X, \mathcal{H}^N)$ be a non-collapsed $\RCD(K, N)$ space. Then the following holds.
\begin{enumerate}
\item The essential dimension of $(X, \dist_X, \mathcal{H}^N)$ is equal to $N$.
\item It holds that for any $x \in X$,
\begin{equation}\label{ref}
\lim_{r \to 0^+}\frac{\mathcal{H}^N(B_r(x))}{\omega_Nr^N} \le 1.
\end{equation}
Moreover the equality in (\ref{ref}) is satisfied if and only if $x \in \mathcal{R}_N$ holds.
\end{enumerate}
\end{theorem}
The inequality (\ref{ref}) is sometimes refered as the \textit{Bishop inequality}. See \cite[Th.1.3 and 1.6]{DG}. 
It is worth pointing out that a quantitative version of the rigidity part of the Bishop inequality is also satisfied as follows, where $\dist_{\mathrm{GH}}, \dist_{\mathrm{pmGH}}$ denote the Gromov-Hausdorff, pointed Gromov-Hausdorff distances, respectively.
\begin{theorem}[Almost rigidity of Bishop inequality]\label{prop:almostrigidity}
Let $(X, \dist_X, \mathcal{H}^N)$ be a non-collapsed $\RCD(K, N)$ space and let $x \in X$.
If 
\begin{equation}\label{ttffff}
\left|\frac{\mathcal{H}^N(B_r(x))}{\omega_Nr^N} - 1\right|<\epsilon
\end{equation}
holds for some $\epsilon, r \in (0, 1)$, then  
\begin{equation}
\dist_{\mathrm{GH}}(B_{r/2}(x), B_{r/2}(0_N))< \Psi(\epsilon, r;K, N)r
\end{equation}
and
\begin{equation}
\dist_{\mathrm{pmGH}}\left((X, t^{-1}\dist_X, x, \mathcal{H}^N ), (\mathbb{R}^N, \dist_{\mathbb{R}^N}, 0_N, \mathcal{H}^N)\right)<\Psi (\epsilon, t/r, r; K, N), \quad \forall t \in (0, 1)
\end{equation}
hold. Conversely if
\begin{equation}
\dist_{\mathrm{GH}}(B_{r}(x), B_{r}(0_N))<\epsilon r
\end{equation}
holds for some $\epsilon, r \in (0, 1)$, then
\begin{equation}
\left|\frac{\mathcal{H}^N(B_r(x))}{\omega_Nr^N} - 1\right|<\Psi(\epsilon, r;K, N)
\end{equation}
is satisfied.
\end{theorem}
See \cite[Th.1.3 and 1.6]{DG} for the proof (see also \cite[Prop.6.5]{AHPT}).

Finally let us end this subsection by giving the following convergence result proved in \cite[Th.1.2]{DG} (see \cite[Th.5.9]{CheegerColding1} with \cite[Th.0.1]{Colding} for the corresponding results on Ricci limit spaces).
\begin{theorem}[GH implies mGH]\label{GHmGH}
Let $(X_i, \dist_{X_i}, \mathcal{H}^N, x_i)$ be a sequence of pointed non-collapsed $\RCD(K, N)$ spaces. If $(X_i, \dist_{X_i}, x_i)$ pointed Gromov-Hausdorff converge to a pointed complete metric space $(X, \dist_X, x)$, then
\begin{equation}
\mathcal{H}^N(B_r(z_i)) \to \mathcal{H}^N(B_r(z))
\end{equation}
holds for any $r \in (0, \infty)$ and any $z_i \in X_i \to z \in X$.
\end{theorem}

\subsection{Heat kernel}\label{susectionheat}
Let $(X, \dist_X, \meas_X)$ be a finite dimensional $\RCD$ space. In order to give precise estimates below, we assume that it is an $\RCD(K, N)$ space with $K \in \mathbb{R}$ and $N \in [1, \infty)$.

Then the \textit{heat kernel $p_X(x, y, t)$ of $(X, \dist_X, \meas_X)$} is determined by the continuous function $p_X:X \times X \times (0, \infty) \to \mathbb{R}$ satisfying
\begin{equation}
h_tf(x)=\int_Xf(y)p_X(x, y, t)\di \meas_X (y), \quad \forall f \in L^2(X, \meas_X),\,\,\,\forall x \in X.
\end{equation}
The sharp Gaussian estimates on $p_X$  proved by Jiang-Li-Zhang \cite[Th.1.2]{JiangLiZhang} is  as follows;
for any $\epsilon>0$, there exists $C=C(K, N, \epsilon) \in (1, \infty)$ depending only on $K, N$ and $\epsilon$ such that
\begin{equation}\label{eq:gaussian}
\frac{C^{-1}}{\meas_X (B_{\sqrt{t}}(x))}\exp \left(-\frac{\dist_X (x, y)^2}{(4-\epsilon)t}-Ct \right) \le p_X(x, y, t) \le \frac{C}{\meas_X (B_{\sqrt{t}}(x))}\exp \left( -\frac{\dist_X (x, y)^2}{(4+\epsilon)t}+Ct \right)
\end{equation}
for all $x,\, y \in X$ and any $t \in (0, 1]$. Combined with the Li-Yau inequality \cite[Cor.1.5]{GarofaloMondino}, \cite[Th.1.1, 1.2 and 1.3]{Jiang15}, 
\eqref{eq:gaussian} implies a gradient estimate \cite[Cor.1.2]{JiangLiZhang}
\begin{equation}\label{eq:equi lip}
|\nabla_x p_X(x, y, t)|\le \frac{C}{\sqrt{t}\meas_X (B_{\sqrt{t}}(x))}\exp \left(-\frac{\dist_X(x, y)^2}{(4+\epsilon) t}+Ct\right)
\qquad\text{for $\meas_X$-a.e. $x\in X$}
\end{equation}
for any $t>0$, $y\in X$. Note that the gradient estimate (\ref{eq:equi lip}) is also satisfied if we replace the minimal relaxed slope in the LHS of (\ref{eq:equi lip}) by the asymptotically Lipschitz constant (cf. \cite[Prop.1.10]{GV}) because of the continuity of the RHS of (\ref{eq:equi lip}).

From now on we assume that $(X, \dist_X)$ is compact. Then since the inclusion $H^{1, 2}(X, \dist_X, \meas_X) \hookrightarrow L^2(X, \meas_X)$ is a compact operator (cf.  \cite[Th.8.1]{HajlaszKoskela}), we know that the (minus) Laplacian $-\Delta$ admits a discrete positive spectrum
\begin{equation}\label{eq:spectrum}
0=\lambda_0^X < \lambda_1^X \le \lambda_2^X \le \cdots \to + \infty
\end{equation}
counted with multiplicities. Denote by $\phi_i^X$ corresponding eigenfunctions of $\lambda_i^X$ with $\|\phi_i^X\|_{L^2}=1$ and recall that $\{\phi_i^X\}_i$ is an $L^2$-orthogonal basis of $L^2(X, \meas_X)$ and that each $\phi_i^X$ is Lipschitz.
 
It is well-known that the following expansions for $p_X$ hold
\begin{equation}\label{eq:expansion1}
p_X(x,y,t) = \sum_{i \ge 0} e^{- \lambda_i^X t} \phi_i^X(x) \phi_i^X (y) \qquad \text{in $C(X\times X)$}
\end{equation}
for any $t>0$ and
\begin{equation}\label{eq:expansion2}
p_X(\cdot,y,t) = \sum_{i \ge 0} e^{- \lambda_i^X t} \phi_i^X(y) \phi_i^X \qquad \text{in $H^{1,2}(X,\dist_X,\meas_X)$}
\end{equation}
for any $y\in X$ and $t>0$. 
Combining (\ref{eq:expansion1}) and (\ref{eq:expansion2}) with (\ref{eq:equi lip}), we know that
\begin{equation}\label{eq:eigenfunction}
\|\phi_i^X\|_{L^\infty} \leq \tilde{C} (\lambda_i^X)^{N/4}, \qquad \| \nabla \phi_i^X \|_{L^\infty} \leq \tilde{C} (\lambda_i^X)^{(N+2)/4}, \qquad \lambda_i^X \ge \tilde{C}^{-1}i^{2/N},
\end{equation}
where $\tilde{C}:=\tilde{C}(d, K, N) \in (1, \infty)$ and $d$ denotes an upper bound on the diameter of $(X, \dist_X)$ (cf. the appendix of \cite{AHPT}).

\subsection{Pull-back by Lipschitz map into Hilbert space}
Let $(X, \dist_X, \meas_X)$ be a finite dimensional $\RCD$ space whose essential dimension is equal to $n$ and let $A$ be a Borel subset of $X$. We start this section by recalling \textit{Lebesgue points};
\begin{definition}[Lebesgue point]\label{lebesguepoi}
Let $f\in L_\loc^p(X, \meas_X)$ with $p\in [1,\infty)$. 
We say that $x \in X$ is a \textit{$p$-Lebesgue point of $f$} if there exists $a \in \setR$ such that
\[
\lim\limits_{r \to 0} \frac{1}{\meas_X(B_r(x))}\int_{B_r (x)} |f(y)-a|^p\di \meas_X (y) = 0.
\]
The real number $a$ is uniquely determined by this condition and denoted by $\overline{f}(x)$ (we omit the $p$-dependence). 
The set of all $p$-Lebesgue points of $f$ is Borel and denoted by $\Leb_p(f)$.
\end{definition}

Note that the property of being a $p$-Lebesgue point and $\overline{f}(x)$ do not depend on the choice of the
versions of $f$, and that $x\in\Leb_p(f)$ implies
$\meas_X(B_r(x))^{-1}\int_{B_r (x)} |f(y)|^p\di\meas_X\to |\overline{f}(x)|^p$ as $r\to 0^+$.
It is well-known (e.g. subsection 3.4 of \cite{HKST}) that the doubling property ensures that $\meas_X(X \setminus \Leb_p(f))=0$, 
and that the set $\{ x\in\Leb_p(f) :\  \overline{f}(x) = f(x) \}$ (which does depend on the choice of representative in 
the equivalence class) has full measure in $X$, so-called Lebesgue differentiation theorem. When we apply these properties to a characteristic
function $f=1_A$, we obtain that $\meas_X$-a.e. $x\in A$ is a point of density $1$ for $A$ and
$\meas_X$-a.e. $x\in X\setminus A$ is a point of density $0$ for $A$, namely, the set
\begin{equation}
\mathrm{Leb} (A):=\left\{ x \in A; \lim_{r \to 0^+}\frac{\meas_X (B_r(x) \cap A)}{\meas_X (B_r(x))}=1\right\}
\end{equation}
satisfies 
\begin{equation}\label{lebdesntity}
\meas_X (A \setminus \mathrm{Leb} (A))=0.
\end{equation}
Lebesgue points can be understood as metric ``measure'' density points. In fact  (\ref{lebdesntity}) implies
\begin{equation}\label{densitymetri}
\mathrm{Leb} (A) \subset \mathrm{Den}(A).
\end{equation}
\begin{definition}[Pull-back]
Let $(H, \langle \cdot, \cdot \rangle)$ be a separable real Hilbert space and let $f:A \to H$ be a Lipschitz map. The \textit{pull-back by $f$}, denoted by $f^*g_H$, is defined by
\begin{equation}
f^*g_H:=\sum_{i=1}^{\infty}\dist f_i \otimes \dist f_i, \qquad \text{in $L^{2}((T^*)^{\otimes 2}(A, \dist_X, \meas_X))$},
\end{equation}
where $f=\sum_{i=1}^{\infty}f_ie_i$ for some orthonormal basis $\{e_i\}_i$ of $H$. This does not depend on the choice of $\{e_i\}_i$ with 
\begin{equation}\label{eq:bound}
|f^*g_H|(x) \le \sum_{i=1}^{\infty}|\dist f_i|^2(x) \le nL^2, \quad for\,\,\meas_X-a.e.\,\,x \in A,
\end{equation}
whenever $f$ is $L$-Lipschitz.
\end{definition}
See \cite[Lem.4.8 and Prop.4.9]{AHPT} for the detail. 
\begin{lemma}\label{lem:1}
Let $f:A \to \mathbb{R}^k$ be a Lipschitz map, let $\epsilon \in (0, 1)$ and let $\Phi:f(A) \to \mathbb{R}^l$ be a $(1\pm \epsilon)$-bi-Lipschitz embedding.
Then for $\meas_X$-a.e. $x \in A$,
\begin{equation}\label{eq:keyineq}
(1+\epsilon)^{-2}f^*g_{\mathbb{R}^k} \le (\Phi \circ f)^*g_{\mathbb{R}^l}\le (1+\epsilon)^2f^*g_{\mathbb{R}^k},
\end{equation}
that is, for any $ V \in L^{\infty}(T(A, \dist_X, \meas_X))$,
\begin{equation}\label{asbruaiburaw}
(1+\epsilon)^{-2}\int_Af^*g_{\mathbb{R}^k}(V, V)\di \meas_X \le \int_A (\Phi \circ f)^*g_{\mathbb{R}^l}(V, V)\di \meas_X \le (1+\epsilon)^2\int_Af^*g_{\mathbb{R}^k}(V, V)\di \meas_X.
\end{equation}
In particular for $\meas_X$-a.e. $x \in A$,
\begin{equation}
\left|(\Phi \circ f)^*g_{\mathbb{R}^l}- f^*g_{\mathbb{R}^k}\right|(x)\le C(n) \epsilon \left| f^*g_{\mathbb{R}^k}\right|(x).
\end{equation}
\end{lemma}
\begin{proof}
Let us divide the proof into the following two cases.

\textit{Case $1$; $(X, \dist, \meas)=(\mathbb{R}^n, \dist_{\mathbb{R}^n}, \mathcal{H}^n)$.}

Kirszbraun's theorem states that there exist Lipschitz maps $\tilde{f}:\mathbb{R}^n \to \mathbb{R}^k$ and $\tilde{\Phi}:\mathbb{R}^k \to \mathbb{R}^l$ such that $\tilde{\Phi}$ is $(1+\epsilon)$-Lipschitz, that $\tilde{f}|_A=f$ and that $\tilde{\Phi}|_{f(A)}=\Phi$.
Fix $x \in \mathrm{Leb}(A)$ where both $\tilde{f}$ and $\tilde{\Phi} \circ \tilde{f}$ are differentiable at $x$. Recall that for any $v \in \mathbb{R}^n$ we have
\begin{equation}
|J(\tilde{f})(x)v|_{\mathbb{R}^k}^2=\sum_{i=1}^k(\dist_x \tilde{f}_i(v))^2 =\sum_{i=1}^k(\dist_x f_i(v))^2 
\end{equation}
and
\begin{equation}
|J(\tilde{\Phi} \circ \tilde{f})(x)v|_{\mathbb{R}^l}^2=\sum_{i=1}^l(\dist_x (\tilde{\Phi}_i \circ \tilde{f})(v))^2=\sum_{i=1}^l(\dist_x (\Phi_i \circ f)(v))^2, 
\end{equation}
where $J(\tilde{f}), J(\tilde{\Phi}\circ \tilde{f})$ denote the corresponding Jacobi matrices, $\tilde{f}=(\tilde{f}_i)_i$, $\tilde{\Phi}=(\tilde{\Phi}_i)_i$, and we used the locality of the exterior derivative.
Let us prove
\begin{equation}\label{eq:ineq1}
|J(\tilde{\Phi} \circ \tilde{f})(x)v|_{\mathbb{R}^l} \le (1+\epsilon) |J(\tilde{f})(x)v|_{\mathbb{R}^k}. 
\end{equation}
By the differentiability of $\tilde{f}$ at $x$ for any $\delta \in (0, 1)$ we know
\begin{equation}
\tilde{f}\left(x+\frac{t}{|J(\tilde{f})(x)v|_{\mathbb{R}^k}+\delta }v\right)=\tilde{f}(x) +\frac{t}{|J(\tilde{f})(x)v|_{\mathbb{R}^k}+\delta}J(\tilde{f})(x)v + o(|t|).
\end{equation}
In particular combining this with the $(1+\epsilon)$-Lipschitz continuity of $\Phi$ implies that
\begin{equation}\label{eq:11}
|t|^{-1}\left| \tilde{\Phi}\left( \tilde{f}\left(x+\frac{t}{|J(\tilde{f})(x)v|_{\mathbb{R}^k}+\delta }v\right)\right)- \tilde{\Phi}(\tilde{f}(x))\right| \le (1+\epsilon) \cdot \frac{|J(\tilde{f})(x)v|_{\mathbb{R}^k}}{|J(\tilde{f})(x)v|_{\mathbb{R}^k}+\delta} +o(1).
\end{equation}
Thus letting $t \to 0$ and then letting $\delta \to 0^+$ in (\ref{eq:11}) with the differentiability of $\tilde{\Phi} \circ f$ at $x$ show (\ref{eq:ineq1}).

Then applying the above argument for the maps $\Phi \circ f:A \to \mathbb{R}^l$, $\Phi^{-1}:\Phi(f(A)) \to \mathbb{R}^k$ shows that for $\mathcal{H}^n$-a.e. $x \in \mathbb{R}^n$, and for any $v \in \mathbb{R}^n$
\begin{equation}\label{eq:ineq2}
(1+\epsilon)^{-2} \sum_{i=1}^k(\dist_x f_i(v))^2 \le \sum_{i=1}^l(\dist_x (\Phi_i \circ f)(v))^2 \le (1+\epsilon)^2\sum_{i=1}^k(\dist_x f_i(v))^2,
\end{equation}
which completes the proof of (\ref{asbruaiburaw}).

\textit{Case $2$; general $(X, \dist_X, \meas_X)$.}

For the general case, for any $\delta \in (0, 1)$ we find a $(1\pm \delta)$-bi-Lipschitz rectifiable chart $\hat{\phi}:\hat{A} \to \mathbb{R}^n$ of $(X, \dist_X, \meas_X)$ (see Theorem \ref{thm:RN} for the definition of rectifiable charts). Then applying (\ref{eq:keyineq}) for $\hat{\phi}(A \cap \hat{A})$ completes the proof because of the arbitrariness of $\delta$.
\end{proof}
\begin{corollary}\label{corfromlip}
Let $\epsilon \in (0, 1)$ and let $f:A \to \mathbb{R}^k$ be a $(1\pm \epsilon)$-bi-Lipschitz embedding. Then
\begin{equation}\label{ddccvgrrtm}
\left|f^*g_{\mathbb{R}^k}-g_X\right|(x)\le C(n)\epsilon,\quad \text{for $\meas_X$-a.e. $x \in A$.}
\end{equation}
\end{corollary}
\begin{proof}
It follows from Lemma \ref{lem:1} that (\ref{ddccvgrrtm}) holds if $(X, \dist, \meas)=(\mathbb{R}^n, \dist_{\mathbb{R}^n}, \mathcal{H}^n)$. In general case, the same argument as in the Case $2$ of the proof of the lemma allows us to conlude.
\end{proof}
Similarly we have the following which gives a geometric meaning of the pull-back;
\begin{proposition}\label{prop:geometric}
Let $f:A \to \mathbb{R}^k$ be a Lipschitz map. Then
\begin{equation}\label{eq:geomet}
\mathrm{Lip}f(x)=\left(\left|f^*g_{\mathbb{R}^k}\right|_B(x)\right)^{1/2},\quad \text{for $\meas$-a.e. $x \in A$}.
\end{equation}
In particular
\begin{equation}\label{eq:geomet2}
\mathrm{Lip}f(x)\le \left(\left|f^*g_{\mathbb{R}^k}\right|(x)\right)^{1/2},\quad \text{for $\meas$-a.e. $x \in A$}.
\end{equation}
\end{proposition}
\begin{proof}
As in the proof of Lemma \ref{lem:1}, it is enough to consider the case when $(X, \dist, \meas)=(\mathbb{R}^n, \dist_{\mathbb{R}^n}, \mathcal{H}^n)$ (note that (\ref{eq:geomet2}) is a direct consequence of Proposition \ref{prop:ineq} with (\ref{eq:geomet})).

Applying Kirszbraun's theorem, we can find a Lipschitz map $\tilde{f}:\mathbb{R}^n \to \mathbb{R}^k$ with $\tilde{f}|_A=f$.
Then since
\begin{equation}
\mathrm{Lip}\tilde{f}(x)=\left|\sum_{i=1}^k\dist \tilde{f}_i(x)\otimes \dist \tilde{f}_i(x) \right|_B^{1/2}
\end{equation}
holds for any differentiable point $x$ of $\tilde{f}=(\tilde{f}_i)_i$, we see that (\ref{eq:geomet}) holds for any $x \in \mathrm{Leb}(A)$ which is also a differentiable point of $\tilde{f}$ because of Proposition \ref{prop:den} with (\ref{densitymetri}). Thus we conclude because of (\ref{lebdesntity}).
\end{proof}
\subsection{Embedding into $L^2$ via heat kernel/eigenfunctions}\label{subsec:pullbak}
Let $(X, \dist_X, \meas_X)$ be a finite dimensional compact $\RCD$ space whose essential dimension is equal to $n$.
Then for any $t \in (0, \infty)$ the map $\Phi_t^X:X \to L^2(X, \meas_X)$ (we will use the simplified notation $\Phi_t$ instead of using $\Phi_t^X$ below) defined by
\begin{equation}\label{eq:embedding}
\Phi_t(x):=\left( z \mapsto p_X(x, z, t)\right)
\end{equation}
is Lipschitz and gives a topological embedding to the image $\Phi_t(X)$. Fix an $L^2$-orthonomal basis $\{\phi_i^X\}_i$ associated with (\ref{eq:spectrum}), namely, $\Delta_X\phi_i^X+\lambda_i^X\phi_i^X=0$ and $\|\phi_i^X\|_{L^2}=1$ are satisfied. 
Then for the canonical isometry $\iota: L^2(X, \meas_X) \to \ell^2$ via $\{\phi_i^X\}_i$, we have
\begin{equation}\label{eq:ell23}
\Phi_t^{\ell^2}(x):=\iota \circ \Phi_t(x)=\left(e^{-\lambda_i^Xt}\phi_i^X(x)\right)_{i=1}^{\infty}.
\end{equation}
Moreover it is proved in \cite[Prop.4.7 and Th.5.10]{AHPT} that the pull-back $g_t:=\Phi_t^*g_{L^2}$ can be written by
\begin{align}
g_t&=\int_X\dist_xp_x( \cdot, z, t) \otimes \dist_xp(\cdot, z, t)\di \meas_X(z)\nonumber \\
&=\sum_ie^{-2\lambda_i^Xt}\dist \phi_i^X \otimes \dist \phi_i^X,
\end{align}
that $t\meas_X (B_{\sqrt{t}}(\cdot ))g_t$ $L^p$-strongly converge to $\overline{c}_ng_X$ for any $p \in [1, \infty)$, where $\overline{c}_n$ is a positive constant depending only on $n$, and that 
\begin{equation}
\left| t\meas_X (B_{\sqrt{t}}( \cdot ))g_t\right|(x)\le C(K, N)<\infty, \quad \text{for $\meas_X$-a.e. $x \in X$},
\end{equation}
if $(X, \dist_X, \meas_X)$ is a compact $\RCD(K, N)$ space and $t \in (0, 1)$.

\subsection{Convergence}
In this section we will discuss several convergence results.
\subsubsection{Uniform convergence}
We discuss the convergence of maps into $\ell^2:=\{(a_i)_{i=1}^{\infty}; a_i \in \mathbb{R}, \sum_{i=1}^{\infty}(a_i)^2<\infty\}$.
\begin{proposition}\label{prop:compactness}
Let 
\begin{equation}\label{eq:pmGH}
(X_i, \dist_{X_i}) \stackrel{\mathrm{GH}}{\to} (X, \dist_X)
\end{equation}
be a Gromov-Hausdorff convergent sequence of compact metric spaces, let $L \in (0, \infty)$ and let $\Phi_i=(\phi_{i, j})_j:X_i \to \ell^2$ be a sequence of $L$-Lipschitz maps. Assume that the following two conditions hold;
\begin{enumerate}
\item we have $\sup_{i \in \mathbb{N}, x_i \in X_i}\|\Phi_i(x_i)\|_{\ell^2} <\infty$;
\item for any $\epsilon \in (0, 1)$ there exists $l:=l(\epsilon) \in \mathbb{N}$ such that 
\begin{equation}\label{eq:decay}
\sup_{i, y_i \in X_i}\sum_{j \ge l}\phi_{i, j}(y_i)^2<\epsilon
\end{equation}
holds.
\end{enumerate}
Then after passing to a subsequence there exists an $L$-Lipschitz map $\Phi:X \to \ell^2$ such that $\Phi_i$ converge uniformly to $\Phi$ on $X$, namely, $\{\Phi_i\}_i$ is equi-continuous and $\Phi_i(x_i) \to \Phi(x)$ holds whenever $x_i \to x \in X$.
\end{proposition}
\begin{proof}
Since the sequence $\{\phi_{i, j}\}_i$ has a uniformly convergent subsequence for each $j \in \mathbb{N}$, after passing to a diagonal process, there exists a sequence of Lipschitz functions $\phi_j$ on $X$ such that $\phi_{i, j}$ converge uniformly to $\phi_j$ on $X$ for each $j$. With no loss of generality we can assume that $\Phi_i(x_i)=0$ holds for a convergent sequence $x_i \in X_i$ to $x \in X$. For any $l \in \mathbb{N}$ and any convergent sequence $y_i \in X_i \to y \in X$, we have 
\begin{equation}\label{eq:sum}
\left(\sum_{j=1}^l\phi_j(y)^2\right)^{1/2}=\lim_{i \to \infty}\left(\sum_{j=1}^l\phi_{i, j}(y_i)^2\right)^{1/2}\le \limsup_{i \to \infty}\left(\sum_{j=1}^{\infty}\phi_{i, j}(y_i)^2\right)^{1/2}\le L\dist (x, y),
\end{equation}
where we used the $L$-Lipschitz continuity of $\Phi_i$ in the last inequality of (\ref{eq:sum}) with $\Phi_i(x_i)=0$. Letting $l \to \infty$ in (\ref{eq:sum}) shows that the function $\Phi:=(\phi_i)_i$ from $X$ to $\ell^2$ is well-defined. It is easy to check the pointwise convergence of $\Phi_i$ to $\Phi$ by (\ref{eq:decay}). Thus the $L$-Lipschitz continuity of $\Phi$ comes from that of $\Phi_i$. Moreover it is easy to check the desired uniform convergence.
\end{proof}
\begin{remark}
In the above theorem, the assumption (\ref{eq:decay}) is essential. Actually a sequence of maps $\Phi_i$ from a single point $\{p\}$ to $\ell^2$ defined by 
\begin{equation}
\Phi_i(p):=(\overbrace{0,0,\ldots, 0, 1}^i, 0,\ldots) 
\end{equation}
has no pointwise convergent subsequence. This reason comes from the fact that a subset $A$ of $\ell^2$ including $0$ is relatively compact if and only if for any $\epsilon \in (0, 1)$, there exists $l \in \mathbb{N}$ such that 
\begin{equation}
\sup_{x \in A}\sum_{j \ge l}x_j^2<\epsilon
\end{equation}
holds, where $x=(x_i)_i$. In connection with this observation, we can easily check that for two Gromov-Hausdorff convergent sequences $(X_i, \dist_{X_i}) \stackrel{\mathrm{GH}}{\to} (X, \dist_X), (Y_i, \dist_{Y_i}) \stackrel{\mathrm{GH}}{\to} (Y, \dist_Y)$ of compact metric spaces, any sequence of equi-continuous maps $\Phi_i:X_i \to Y_i$ has a uniform convergent subsequence to a continuous map $\Phi:X \to Y$. 
\end{remark}
\subsubsection{Functional convergence}
Let us fix $R \in (0, \infty], K \in \mathbb{R}, N \in [1, \infty)$ and a pointed measured Gromov-Hausdorff convergent sequence of pointed $\RCD(K, N)$ spaces
\begin{equation}\label{jju}
(X_i, \dist_{X_i}, \meas_{X_i}, x_i) \stackrel{\mathrm{pmGH}}{\to} (X, \dist_X, \meas_X, x).
\end{equation}
In this setting it is well-defined that a sequence $f_i \in L^p(B_R(x_i), \meas_{X_i})$ \textit{$L^p$-strongly/weakly converge to $f \in L^p(B_R(x), \meas_X)$ on $B_R(x)$} for $p \in [1, \infty)$. Note that $B_R(x)=X$ when $R=\infty$.  Since it is enough to discuss only on $L^2$-ones for our purposes, we recall it here (see \cite{AmbrosioHonda, AmbrosioHonda2, AmbrosioStraTrevisan, GMS, Honda15} for the details).
\begin{definition}[$L^2$-convergence of functions]
Let $f_i \in L^2(B_R(x_i), \meas_{X_i})$ be a sequence of $L^2$-functions on $B_R(x_i)$ and let $f \in L^2(B_R(x), \meas_X)$.
\begin{enumerate}
\item We say that $f_i$ \textit{$L^2$-weakly converge to $f$ on $B_R(x)$} if $\sup_i\|f_i\|_{L^2(B_R(x_i))}<\infty$ holds and 
\begin{equation}
\int_{B_R(x_i)}\phi_if_i\di \meas_{X_i} \to \int_{B_R(x)}\phi f\di \meas_X
\end{equation} 
holds for any uniformly convergent sequence $\phi_i \in C_c(X_i)$ to $\phi \in C_c(X)$ with uniformly compact supports {\color{blue}(namely there exists $R>0$ such that $\supp \phi_i \subset B_R(x_i)$ and $\supp \phi \subset B_R(x)$ are satisfied for any $i$).}
\item  We say that $f_i$ \textit{$L^2$-strongly converge to $f$ on $B_R(x)$} if it is an $L^2$-weakly convergent sequence on $B_R(x)$ and $\limsup_{i \to \infty}\|f_i\|_{L^2(B_R(x_i))}\le \|f\|_{L^2(B_R(x))}$ holds.
\end{enumerate}
\end{definition}
A typical example is that $1_{B_r(y_i)}$ $L^2$-strongly converge to $1_{B_r(y)}$ on $B_R(x)$ for all $r \in (0, \infty)$ and $y_i \in X_i \to y \in X$.

Next we give the definition of $L^2$-convergence of tensors as follows.
\begin{definition}[Convergence of tensors]
We say that a sequence $T_i \in L^2((T^*)^{\otimes 2}(B_R(x_i), \dist_{X_i}, \meas_{X_i}))$ \textit{$L^2$-weakly converge to $T \in L^2((T^*)^{\otimes 2}(B_R(x), \dist_X, \meas_X))$ on $B_R(x)$} if the following two conditions are satisfied.
\begin{enumerate}
\item $\sup_i\|T_i\|_{L^2(B_R(x_i))}<\infty$ holds.
\item We see that $\langle T_i, \dist f_{1, i} \otimes \dist f_{2, i}\rangle$  $L^2$-weakly converge to $\langle T, \dist f_1 \otimes \dist f_2\rangle$ on $B_R(x)$ whenever $f_{j, i} \in \mathrm{Test}F(X_i, \dist_{X_i}, \meas_{X_i})$ $L^2$-strongly converge to $f_j \in \mathrm{Test}F(X, \dist_X, \meas_X)$ with 
\begin{equation}
\sup_{i, j}\left(\|f_{j, i}\|_{L^{\infty}(X_i)}+\|\nabla f_{j, i}\|_{L^{\infty}(X_i)} +\|\Delta_{X_i}f_{j, i}\|_{L^2(X_i)}\right)<\infty.
\end{equation}
\end{enumerate}
Moreover we say that \textit{$T_i$ $L^2$-strongly converge to $T$ on $B_R(x)$} if it is an $L^2$-weak convergent sequence on $B_R(x)$ with $\limsup_{i \to \infty}\|T_i\|_{L^2(B_R(x_i))}\le \|T\|_{L^2(B_R(x))}$ holds. 
\end{definition}
Compare with \cite[Def.5.18 and Lem.6.4]{AHPT}. 
Note that we can easily check the following by an argument similar to the proof of \cite[Th.10.3]{AmbrosioHonda} (cf. \cite[Prop.2.24]{honda20}).
\begin{proposition}[Lower semicontinuity of $L^2$-norms]\label{prop:low}
A sequence $T_i \in L^2((T^*)^{\otimes 2}(B_R(x_i), \dist_{X_i}, \meas_{X_i}))$ $L^2$-weakly converge to $T \in L^2((T^*)^{\otimes 2}(B_R(x), \dist_X, \meas_X))$ on $B_R(x)$, then it holds that
\begin{equation}
\liminf_{i \to \infty}\|T_i\|_{L^2(B_R(x_i))}\ge \|T\|_{L^2(B_R(x))}.
\end{equation}
\end{proposition}

Let us give a typical example of $L^2$-weak convergence of tensors. The following lower semicontinuity of the essential dimensions is already proved in \cite[Th.1.5]{kita} by a different way (see also \cite[Rem.5.20]{AHPT} and \cite[Prop. 2.27]{honda20}).
\begin{proposition}[$L^2_{\mathrm{loc}}$-weak convergence of Riemannian metrics]\label{weakriem}
Assume $R<\infty$. Then $g_{X_i}$ $L^2$-weakly converge to $g_X$ on $B_R(x)$. In particular the essential dimensions are lower semicontinuous with respect to the pointed measured Gromov-Hausdorff convergence.
\end{proposition}
Note that similarly we can define $L^2$-strong/weak convergence of vector fields with respect to (\ref{jju}) (see also \cite{AmbrosioStraTrevisan, Honda15}).

Next let us recall the definition of $H^{1, 2}$-strong convergence:
\begin{definition}[$H^{1, 2}$-strong convergence]\label{dettrrrss}
We say that a sequence of $f_i \in H^{1, 2}(B_R(x_i), \dist_{X_i}, \meas_{X_i})$ \textit{$H^{1, 2}$-strongly converge to $f \in H^{1, 2}(B_R(x), \dist_X, \meas_X)$ on $B_R(x)$} if $f_i$ $L^2$-strongly converge to $f$ on $B_R(x)$ with $\lim_{i \to \infty}\|\nabla f_i\|_{L^2(B_R(x_i))}=\|\nabla f\|_{L^2(B_R(x))}$.
\end{definition}
In connection with Definition \ref{dettrrrss}, we introduce a Rellich type compactness result with respect to measured Gromov-Hausdorff convergence (see \cite[Th.6.3]{GMS}, \cite[Th.7.4]{AmbrosioHonda} and \cite[Th.4.2]{AmbrosioHonda2}).
\begin{theorem}[Convergence of gradient operators]\label{bbbg}
If a sequence $f_i \in H^{1, 2}(B_R(x_i), \dist_{X_i}, \meas_{X_i})$ satisfies $\sup_i\|f_i\|_{H^{1, 2}}<\infty$, then after passing to a subsequence, there exists $f \in H^{1, 2}(B_R(x), \dist_X, \meas_X)$ such that $f_i$ $L^2$-strongly converge to $f$ on $B_r(x)$ for any $r \in (0, R)$ and that $\nabla f_i$ $L^2$-weakly converge to $\nabla f$ on $B_R(x)$. In particular 
\begin{equation}
\liminf_{i \to \infty}\|\nabla f_{i}\|_{L^2(B_R(x_i))} \ge \|\nabla f\|_{L^2(B_R(x))}
\end{equation}
holds.  Moreover if in addition $f_i$ $H^{1, 2}$-strongly converge to $f$ on $B_r(x)$ for some $r \in (0, R]$, then $|\nabla f_i|^2$ $L^1$-strongly converge to $|\nabla f|^2$ on $B_r(x)$.
\end{theorem}
Note that in Theorem \ref{bbbg} if $R<\infty$, then $L^2$-strong convergence of $f_i$ to $f$ is satisfied on $B_R(x)$, which is justified by using the Sobolev embedding theorem $H^{1, 2}(B_R(x), \dist_X, \meas_X) \hookrightarrow L^{2N/(N-2)}(B_R(x), \meas_X)$. See \cite[Th.4.2]{AmbrosioHonda2}.

The convergence of the heat flows with respect to (\ref{jju}) is obtained in \cite[Th.5.7]{GMS}  (more precisely they discussed it in more general setting, for $\CD(K, \infty)$ spaces under pmG-convergence). As a corollary, it is proved in \cite[Th.7.8]{GMS} that the following spectral convergence result holds, which will play a key role later (see \cite[Th.7.3 and 7.9]{CheegerColding3} for Ricci limit spaces. Compare with \cite[Prop.3.3]{AmbrosioHonda2}).
\begin{theorem}[Spectral convergence]\label{thm:spectral}
If $(X, \dist_X)$ is compact, 
then 
\begin{equation}
\lambda_j^{X_i} \to \lambda_j^X,\quad \forall j.
\end{equation}
Moreover for any $\phi_j \in D(\Delta_X)$ with $\Delta_X \phi_j^X+\lambda_j^X\phi_j^X=0$, there exists a sequence of $\phi_{j, i}^{X_i} \in D(\Delta_{X_i})$ such that $\Delta_{X_i}\phi_{j, i}^{X_i}+\lambda_j^{X_i}\phi_{j, i}^{X_i}=0$ holds and that $\phi_{j, i}^{X_i}$ $H^{1, 2}$-strongly converge to $\phi_j^X$ on $X$.
\end{theorem}
Let us recall the following stability results proved in \cite[Th.4.4]{AmbrosioHonda2}.
\begin{theorem}[Stability of Laplacian on balls]\label{spectral2}
Let $f_i \in D(\Delta_{X_i}, B_R(x_i))$ satisfy 
$$
\sup_i(\|f_i\|_{H^{1, 2}(B_R(x_i))}+\|\Delta_{X_i} f_i\|_{L^2(B_R(x_i))})<\infty,
$$
and let us assume that $f_i$ $L^2$-strongly converge to $f \in L^2(B_R(x), \meas_X)$ on $B_R(x)$ (so that, by Theorem~\ref{bbbg}, 
$f \in H^{1, 2}(B_R(x), \dist_X, \meas_X)$). Then we have the following. 
\begin{enumerate}
\item[(1)] $f \in D(\Delta_X, B_R(x))$.
\item[(2)] $\Delta_{X_i}f_i$ $L^2$-weakly converge to $\Delta_X f$ on $B_R(x)$. 
\item[(3)] $f_i$ $H^{1, 2}$-strongly converge to $f$ on $B_r(x)$ 
for any $r<R$. 
\end{enumerate}
\end{theorem}
Note that in Theorem \ref{spectral2} if $R=\infty$, then the $H^{1, 2}$-strong convergence of $f_i$ to $f$ is satisfied on $B_R(x)=X$. See \cite[Cor.10.4]{AmbrosioHonda}.

Finally let us mention that $L^p_{\mathrm{loc}}$-strong (or $H^{1, 2}_{\mathrm{loc}}$-strong, or $L^p_{\mathrm{loc}}$-weak, respectively) convergence means the $L^p$-strong (or $H^{1, 2}$-strong, or $L^p$-weak, respectively) convergence on $B_r(x)$ for any $r \in (0, \infty)$.

\subsection{Splitting theorem}
We say that a map $\gamma$ from $\mathbb{R}$ to a metric space $(Z, \dist_Z)$ is a \textit{line} if it is an isometric embedding as metric spaces, that is, $\dist_Z(\gamma(s), \gamma(t))=|s-t|$ holds for all $s,t \in \mathbb{R}$. Then the \textit{Busemann function $b_{\gamma}:Z \to \mathbb{R}$ of $\gamma$} is defined by
\begin{equation}
b_{\gamma}(x):=\lim_{t\to \infty}\left(t-\dist_Z(\gamma(t), x)\right).
\end{equation} 
We introduce now an important result in $\RCD$ theory, the so-called  \textit{splitting theorem}, proved in \cite[Th.1.4]{Gigli13}.
\begin{theorem}[Splitting theorem]\label{splitting}
Let $(X, \dist_X, \meas_X)$ be an $\RCD(0, N)$ space with $N \in [1, \infty)$ and let $x \in X$. Assume that the following (1) or (2) holds.
\begin{enumerate}
\item There exist lines $\gamma_i:\mathbb{R}\to X (i=1, 2, \ldots, k)$ such that $\gamma_i(0)=x$ and
\begin{equation}
\int_{B_1(x)}b_{\gamma_i}b_{\gamma_j}\di \meas_X =0,\quad \forall i \neq j
\end{equation}
are satisfied. 
\item There exist harmonic functions $f_i:X \to \mathbb{R} (i=1, 2, \ldots, k)$ such that $f_i(x)=0$ and $\langle \dist f_i, \dist f_j\rangle \equiv \delta_{ij}$ are satisfied.
\end{enumerate}
Let us put $\phi_i:=b_{\gamma_i}$ if (1) holds, $\phi_i:=f_i$ if (2) holds. 
Then there exist a pointed $\RCD(0, N-k)$ space $(Y, \dist_Y, \meas_Y, y)$ and an isometry
\begin{equation}
\Phi:(X, \dist_X, \meas_X, x) \to \left(\mathbb{R}^k\times Y, \sqrt{\dist_{\mathbb{R}^k}^2+\dist_Y^2}, \mathcal{H}^k\otimes \meas_Y, (0_k, y) \right)
\end{equation}
such that $\phi_i \equiv \pi_i\circ \Phi$ holds, where $\pi_i:\mathbb{R}^k \times Y \to \mathbb{R}$ is the projection to the $i$-th $\mathbb{R}$ of the Euclidean factor $\mathbb{R}^k$.
\end{theorem}
Based on this theorem, we define linear functions as follows.
\begin{definition}[Linear function]
Let $(X, \dist_X, \meas_X)$ be an $\RCD(0, N)$ space. We say that a function $f:X \to \mathbb{R}$ is \textit{linear} if it is harmonic and $|\nabla f|$ is constant.
\end{definition}
Theorem \ref{splitting} tells us that any linear function is a constant or the projection of a Euclidean factor $\mathbb{R}$ of $(X, \dist_X, \meas_X)$, up to multiplying by a constant.
The following well-known proposition will play a key role later. See \cite[Def.3.8]{Ket2} for the definition of warped product spaces of metric measure spaces, in particular, metric measure cones.
\begin{proposition}\label{prop:linearfunction}
Let $(X, \dist_X, \meas_X)$ be an $\RCD(N-1, N)$ space and let $C_0^N(X, \dist_X, \meas_X)$ denote the $(0, N)$-metric measure cone of $(X, \dist_X, \meas_X)$ (then \cite[Th.1.1]{Ket2} proves that $C_0^N(X, \dist_X, \meas_X)$ is an $\RCD(0, N+1)$ space). Then any Lipschitz harmonic function $f$ on $C_0^N(X, \dist_X, \meas_X)$ is linear.
\end{proposition}
\begin{proof}
A well-known proof of this result is to use the spectral theory with the separation of variables (see \cite[Prop.2.1]{DZ}). Let us provide another proof here by contradiction. Compare with the proof of Theorem \ref{prop:rigiditynonnegative}.

If not, there exists a Lipschitz harmonic function $f$ on $C_0^N(X, \dist_X, \meas_X)$ such that $f$ is not linear. Let us denote by $\mathbb{R}^k$ the maximal Euclidean factor of $C_0^N(X, \dist_X, \meas_X)$. Since $C_0^N(X, \dist_X, \meas_X)$ is a scale invariant space, thanks to Theorem \ref{spectral2}, there exists a sequence of $R_i \to \infty$ such that $R_i^{-1}(f-f(p))$ converge to a Lipschitz harmonic function $F$ on $C_0^N(X, \dist_X, \meas_X)$, where $p$ denotes the pole of $C_0^N(X, \dist_X, \meas_X)$. Applying the mean value theorem at infinity proved in \cite[Th.5.4]{HKX} for a bounded subharmonic function $|\nabla f|^2$ with a blow-down argument in \cite{ChCM} shows that $F$ is linear and that $\mathbf{Lip}F=\mathbf{Lip}f$ holds. Since $f$ is not linear, we have $\mathbf{Lip}f>0$, in particular $F$ is not a constant. If in addition $F=\sum_{i=1}^ka_i\pi_i+a_{k+1}$ holds for some $a_i \in \mathbb{R}$,  where $\pi_i$ denotes the $i$-th projection to $\mathbb{R}$, then applying \cite[Th.5.4]{HKX} again for a bounded subharmonic function $|\nabla (f- \sum_{i=1}^ka_i\pi_i)|^2$ shows that $f- \sum_{i=1}^ka_i\pi_i$ is constant. In particular $f$ is linear which is a contradiction. Thus we know that $F$ cannot be written as a linear combination of $\{\pi_i\}_{i=1}^k$.
Since $F$ is not a constant, Theorem \ref{splitting} yields that $C_0^N(X, \dist_X, \meas_X)$ has a Euclidean factor $\mathbb{R}^{k+1}$ which contradicts the maximality of $k$.
\end{proof}

\section{Approximate Sobolev map}\label{sec:app}
Throughout the section, let us fix
\begin{itemize}
\item a finite dimensional (not necessary compact) $\RCD$ space $(X, \dist_X, \meas_X)$,
\item a finite dimensional compact $\RCD$ space $(Y, \dist_Y, \meas_Y)$, 
\item an open subset $U$ of $X$. 
\end{itemize}
For any $\lambda \in \mathbb{R}_{\ge 0}$ let 
\begin{equation}
E_{Y, \lambda}:=\left\{ \phi \in D(\Delta_Y); \Delta_Y\phi+\lambda \phi=0\right\},
\end{equation}
where $(E_{Y, \lambda}, \| \cdot\|_{L^2})$ is a finite dimensional Hilbert space because the canonical inclusion from $H^{1, 2}(Y, \dist_Y, \meas_Y)$ to $L^2(Y, \meas_Y)$ is a compact operator.
\begin{definition}[Weakly smooth map]
A Borel map $f:U \to Y$ is said to be \textit{weakly smooth} if $\phi \circ f \in H^{1, 2}(U, \dist_X, \meas_X)$ holds for any eigenfunction $\phi$ of $(Y, \dist_Y, \meas_Y)$.  
\end{definition}
Note that any Lipschitz map from $U$ to $Y$ is weakly smooth if $U$ is bounded. 
It is easy to check that the following is well-defined because $f^{\lambda, *}g_Y$ vanishes if $\lambda$ is not an eigenvalue of $-\Delta_Y$.
\begin{definition}[$t$-Sobolev map]
Let $f:U \to Y$ be a weakly smooth map.
\begin{enumerate} 
\item For any $\lambda \in \mathbb{R}_{\ge 0}$, put
\begin{equation}
f^{\lambda, *}g_Y:=\sum_{i=1}^k\dist\left( \phi_i \circ f\right) \otimes \dist \left( \phi_i \circ f\right) \in L^1\left( (T^*)^{\otimes 2}(U, \dist_X, \meas_X)\right)
\end{equation} 
and
\begin{equation}
e_{Y}^{\lambda}(f):=\left\langle f^{\lambda, *}g_Y, g_X\right\rangle \in L^1(U, \meas_X),
\end{equation}
where $\{\phi_i\}_{i=1}^k$ is an orthonormal basis of $(E_{Y, \lambda}, \| \cdot\|_{L^2})$. 
\item For any $t \in (0, \infty)$, $f$ is said to be a \textit{$t$-Sobolev map} if 
\begin{equation}\label{eq:enegryfunctional}
\frac{1}{2}\int_U\left(\sum_{\lambda \in \mathbb{R}_{\ge 0}}e^{-2\lambda t}e_{Y}^{\lambda}(f)\right)\di \meas_X<\infty.
\end{equation}
Then the LHS of (\ref{eq:enegryfunctional}) is denoted by $\mathcal{E}_{U, Y, t}(f)$ and called the \textit{$t$-energy of $f$}. Moreover the integrand, $\sum_{\lambda \in \mathbb{R}_{\ge 0}}e^{-2\lambda t}e_{Y}^{\lambda}(f)$, is denoted by $e_{Y, t}(f)$ and called the \textit{$t$-energy density of $f$}.
\end{enumerate}
\end{definition}

\begin{proposition}\label{prop:densityfunction}
Let $t \in (0, \infty)$ and let $f:U \to Y$ be a $t$-Sobolev map. Then the \textit{$t$-pull-back of $f$}, denoted by $f^*g_{Y, t}$
\begin{equation}\label{eq:pullbackl1}
f^*g_{Y, t}:=\sum_{\lambda \in \mathbb{R}_{\ge 0}}e^{-2\lambda t}f^{\lambda, *}g_Y \in L^1\left((T^*)^{\otimes 2}(U, \dist_X, \meas_X)\right)
\end{equation}
is well-defined. Moreover it holds that
\begin{equation}\label{eq:pullbackl12}
e_{Y, t}(f)=\left\langle f^*g_{Y, t}, g_X\right\rangle.
\end{equation}
\end{proposition}
\begin{proof}
Fix $\epsilon \in (0, 1)$ and take an eigenvalue $\lambda$ of $-\Delta_Y$ with
\begin{equation}
\sum_{\mu \ge \lambda}\int_Ue^{-2\mu t}e_Y^{\mu}(f)\di \meas_X<\epsilon.
\end{equation}
Then for any $\mu_i \in [\lambda, \infty) (i=1, 2)$,
\begin{equation}
\left\| \sum_{\mu \le \mu_1} e^{-2\mu t}f^{\mu, *}g_Y- \sum_{\mu \le \mu_2} e^{-2\mu t}f^{\mu, *}g_Y\right\|_{L^1(U)} \le \sum_{\mu \ge \lambda}\int_Ue^{-2\mu t}e_Y^{\mu}(f)\di \meas_X<\epsilon
\end{equation}
which implies that the sequence $\{\sum_{\mu \le \alpha} e^{-2\mu t}f^{\mu, *}g_Y\}_{\alpha \in \mathbb{R}_{\ge 0}}$ is a convergent sequence in $L^1((T^*)^{\otimes 2}(U, \dist_X, \meas_X))$. Thus (\ref{eq:pullbackl1}) is well-defined and (\ref{eq:pullbackl12}) holds.
\end{proof}
Note that for a $t$-Sobolev map $f:U \to Y$, we see that 
\begin{equation}\label{eq:densitybound}
\left|f^*g_{Y, t}\right|(x) \le \sum_{\lambda \in \mathbb{R}_{\ge 0}}e^{-2\lambda t}\left| f^{\lambda, *}g_Y\right|(x) \le  \sum_{\lambda \in \mathbb{R}_{\ge 0}}e^{-2\lambda t}e_Y^{\lambda}(f)(x)=e_{Y, t}(f)(x),\quad \text{for $\meas_X$-a.e. $x \in U$.} 
\end{equation}
and that if $f|_A$ is Lipschitz on a Borel subset $A$ of $U$, then
\begin{equation}\label{absiriiwiiwiwi}
f^*g_{Y, t}=(\Phi_t^Y\circ f)^*g_{L^2}
\end{equation}
in $L^{\infty}((T^*)^{\otimes 2}(A, \dist_X, \meas_X))$.
\begin{theorem}[Compactness]\label{propcompact}
Let $t_i \to t$ be a convergent sequence in $(0, \infty)$, let $R \in (0, \infty]$, let $x \in X$, let $f_i:B_R(x) \to Y$ be a sequence of $t_i$-Sobolev maps with 
\begin{equation}\label{eq:boun}
\liminf_{i \to \infty}\mathcal{E}_{B_R(x), Y, t_i}(f_i)<\infty.
\end{equation}
Then after passing to a subsequence there exists a $t$-Sobolev map $f:B_R(x) \to Y$  such that $f_i(z)$ converge to $f(z)$ for $\meas_X$-a.e. $z \in B_R(x)$ and that
\begin{equation}\label{eq:lowesemico11}
\liminf_{i \to \infty}\int_{B_R(x)}\phi_i e_{Y, t_i}(f_i)\di \meas_X \ge \int_{B_R(x)}\phi e_{Y, t}(f)\di \meas_X
\end{equation}
for any $L^2_{\mathrm{loc}}$-strongly convergent sequence $\phi_i \to \phi$ with $\phi_i \ge 0$ and $\sup_i\|\phi_i\|_{L^{\infty}}<\infty$.
In particular
\begin{equation}\label{eq:lowesemico1}
\liminf_{i \to \infty}\mathcal{E}_{B_R(x), Y, t_i}(f_i)\ge \mathcal{E}_{B_R(x), Y, t}(f).
\end{equation}
\end{theorem}
\begin{proof}
Let us follow the notation of (\ref{eq:spectrum}).
Thanks to (\ref{eq:boun}), for each $i \in \mathbb{N}$ the sequence $\{e^{-\lambda_i^Yt_j}\phi_i^Y \circ f_j\}_j$ is bounded  in $H^{1, 2}(B_R(x), \dist_X, \meas_X)$. Thus after passing to a subsequence with a diagonal process, there exists $F_i \in H^{1, 2}(B_R(x), \dist_X, \meas_X)$ such that $e^{-\lambda_i^Yt_j}\phi_i^Y \circ f_j$ $L^2_{\mathrm{loc}}$-strongly converge to $F_i$ on $B_R(x)$ and that 
$\phi_j\dist (e^{-\lambda_i^Yt_j}\phi_i^Y \circ f_j)$ $L^2$-weakly converge to $\phi \dist F_i$ on $B_R(x)$. In particular after passing to a subsequence again, there exists a Borel subset $A$ of $B_R(x)$ such that $\meas_X(B_R(x) \setminus A)=0$ and that $e^{-\lambda_i^Yt_j}\phi_i^Y \circ f_j(z) \to F_i(z)$ for any $z \in A$. Moreover by (\ref{eq:eigenfunction}) we know that $\Phi_{t_j}^{\ell^2}\circ f_j$ pointwisely converge to a map $\Phi:=(F_i)_i:A \to \ell^2$ on $A$ (recall (\ref{eq:ell23}) for the definition of a topological embedding $\Phi_t^{\ell^2}$ from $Y$ to $\ell^2$). Since it is trivial that $\Phi_{t_j}^{\ell^2}(Y)$ Hausdorff converges to $\Phi_t^{\ell^2}(Y)$ in $\ell^2$ (see \cite[Th.5.19]{AHPT} for a more general result), we have $\Phi (A) \subset \Phi_t^{\ell^2}(Y)$. Thus the map $f:=(\Phi_t^{\ell^2})^{-1}\circ \Phi :A \to Y$ is well-defined. Let $f(z):=x$ for any $z \in B_R(x) \setminus A$.
Then it is trivial that $f$ is weakly smooth. Moreover since for any $l \in \mathbb{N}$
\begin{equation}\label{eq:122}
\liminf_{j \to \infty}\mathcal{E}_{B_R(x), Y, t_j}(f_j)\ge \liminf_{j \to \infty}\frac{1}{2}\sum_{i=1}^l\int_{B_R(x)}|\dist (e^{-\lambda_i^Yt_j}\phi_i^Y \circ f_j)|^2\di \meas_X \ge \frac{1}{2}\sum_{i=1}^l\int_{B_R(x)}|\dist F_i|^2\di \meas_X,
\end{equation}
letting $l \to \infty$ in (\ref{eq:122}) shows 
\begin{equation}\label{eq:1222}
\infty > \liminf_{j \to \infty}\mathcal{E}_{B_R(x), Y, t_j}(f_j)\ge  \frac{1}{2}\sum_{i=1}^{\infty}\int_{B_R(x)}|\dist (e^{-\lambda_i^Yt}\phi_i^Y \circ f)|^2\di \meas_X,
\end{equation}
which prove that $f$ is a $t$-Sobolev map. Finally the $L^2$-weak convergence of $\phi_j\dist (e^{-\lambda_i^Yt_j}\phi_i^Y \circ f_j)$ $L^2$-weakly converge to $\phi \dist F_i=\phi \dist (e^{-\lambda_i^Yt}\phi_i^Y \circ f)$ on $B_R(x)$ implies (\ref{eq:lowesemico11}).
\end{proof}

Next let us recall the definition of Sobolev maps from metric measure spaces to metric spaces in this setting (see \cite[Sec.10]{HKST} and also \cite[Def.2.9]{GT2}).
\begin{definition}[Sobolev map]\label{defsoble}
We say that a map $f:U \to Y$ is a \textit{Sobolev map} if the following two conditions hold.
\begin{enumerate}
\item For any Lipschitz function $\phi$ on $Y$ we have $\phi \circ f \in H^{1, 2}(U, \dist_X, \meas_X)$.
\item There exists $G \in L^2(U, \meas_X)$ such that for any Lipschitz function $\phi$ on $Y$ we have
\begin{equation}\label{uut}
|\nabla (\phi \circ f)|(x) \le \mathbf{Lip} \phi \cdot G(x),\quad \text{for $\meas_X$-a.e. $x \in U$}.
\end{equation}
\end{enumerate}
Then the smallest Borel function $G$, up to $\meas_X$-negligible sets, is denoted by $G_f$.
\end{definition}
It is proved in \cite[Lem.3.2]{GPS} that in Definition \ref{defsoble}, (\ref{uut}) can be improved to
\begin{equation}\label{eq:improve}
|\nabla (\phi \circ f)|(x) \le \mathrm{Lip}_a \phi (f(x)) \cdot G(x),\quad \text{for $\meas_X$-a.e. $x \in U$}.
\end{equation}
The following property of $G_f$ is an easy consequence of (\ref{eq:eigenfunction}).
\begin{proposition}[Sobolev-to-Lipschitz property for Sobolev map]\label{propsobtolip}
Let $f:U \to Y$ be a Sobolev map and let $L \in [0, \infty)$. Then the following two conditions are equivalent.
\begin{enumerate}
\item The map $f$ has locally Lipschitz representative with
\begin{equation}
\dist_Y(f(x), f(\tilde{x}))\le L \dist_X(x, \tilde{x})
\end{equation}
for all $x, \tilde{x} \in U$ with $\dist_X(x, \tilde{x})\le \dist_X(x, \partial U)$.
\item We have $G_f(x)\le L$ for $\meas_X$-a.e. $x \in X$.
\end{enumerate}
\end{proposition}
\begin{proof}
Since one implication is trivial, it is enough to check the implication from (2) to (1). Assume that (2) holds and that $(Y, \dist_Y, \meas_Y)$ is an $\RCD(K, N)$ space. Let us first check that $f$ admits a continuous representative.
Fix $t \in (0,1)$, $l \in \mathbb{N}$ and consider the truncated map $\Phi_t^l:Y \to \mathbb{R}^l$ of $\Phi_t^Y:Y \to L^2(Y, \meas_Y)$ as in the introduction, namely;
\begin{equation}
\Phi_t^l(y):=\left( e^{-\lambda_i^Yt}\phi_i^Y(y)\right)_{i=1}^l.
\end{equation}
Then the local Sobolev-to-Lipschitz property for functions \cite[Prop.1.10]{GV} (or a telescopic argument with the Poincar\'e inequality) ensures that $\Phi_t^l\circ f$ has a locally Lipschitz representative $F^l_t$. Moreover by (\ref{eq:eigenfunction}), we have for $\meas_X$-a.e. $x \in U$
\begin{align}
\mathrm{Lip} F_t^l(x) &=\left| (F^l_t)^*g_{\mathbb{R}^l}\right|_B(x) \nonumber \\
&\le \left| (F^l_t)^*g_{\mathbb{R}^l}\right|(x) \nonumber \\
&\le \left| (\Phi_t^l\circ f)^*g_{L^2}\right| (x) \nonumber \\
&\le \sum_{i=1}^l e^{-2\lambda_i^Yt}\left| \dist (\phi_i^Y \circ f)\right|^2(x) \le L^2 \sum_{i=1}^{\infty}e^{-2\lambda_i^Yt}\mathbf{Lip}\phi_i^Y \le C(L, K, N, t).
\end{align}
Then it follows from a telescopic argument with \cite[Th.8.1.42]{HKST} that $F^l_t$ is a locally $C(L, K, N, t)$-Lipschitz.
In particular, thanks to Arzel\'a-Ascoli theorem with (\ref{eq:eigenfunction}), after passing to a subsequence $\{l_i\}_i$, there exists a locally Lipschitz map $F_t:U\to \ell^2$ such that $F_t^{l_i}$ converge uniformly to $F_t$ as $i \to \infty$ on any compact subset of $U$, where we used immediately the canonical inclusion $\mathbb{R}^l \hookrightarrow \ell^2$ by $v \mapsto (v, 0)$. By the construction of $F_t$, the image is included in $\Phi_t^{\ell^2}(Y)$. Thus the continuous map $f_t:U \to Y$ defined by $f_t:=(\Phi_t^{\ell^2})^{-1} \circ F_t$ provides the desired continuous representative of $f$.

Let us use the same notation $f$ as the continuos representative for brevity.
Fix $x \in U$ and take a $1$-Lipschitz map $\dist_{f(x)}:Y \to \mathbb{R}$ defined by $\dist_{f(x)}(y):=\dist_Y(f(x), y)$.
Applying (\ref{uut}) for this $1$-Lipschitz map shows
\begin{equation}\label{absarubawuaria}
|\nabla (\dist_{f(x)} \circ f)|(z) \le G_f(z) \le L,\quad \text{for $\meas_X$-a.e. $z \in U$}.
\end{equation}
Thus the local Sobolev-to-Lipschitz property \cite[Prop.1.10]{GV} for the function $\dist_{f(x)}\circ f$ with (\ref{absarubawuaria}) and the continuity of $\dist_{f(x)} \circ f$ yields that 
\begin{equation}
\left|(\dist_{f(x)} \circ f)(z)-(\dist_{f(x)} \circ f)(w)\right|\le L\dist_X(z, w)
\end{equation}
for any $z, w \in U$ with $\dist_X(z, w) \le \dist_X(z, \partial U)$.
In particular for any $\tilde{x} \in U$ with $\dist_X(x, \tilde{x}) \le \dist_X(x, \partial U)$,
\begin{equation*}
\dist_Y(f(x), f(\tilde{x}))=(\dist_{f(x)} \circ f)(\tilde{x})=\left|(\dist_{f(x)} \circ f)(x)-(\dist_{f(x)} \circ f)(\tilde{x})\right|\le L\dist_X(x, \tilde{x})
\end{equation*}
which completes the proof.
\end{proof}
We are now in a position to give another definition of Sobolev maps via the heat kernel.
\begin{definition}[$0$-Sobolev map]
A Borel map $f:U \to Y$ is said to be a \textit{$0$-Sobolev map} if it is a $t$-Sobolev map for any sufficiently small $t \in (0, 1)$ with
\begin{equation}\label{eq:suplim}
\limsup_{t \to 0^+}\int_Ut\meas_Y(B_{\sqrt{t}}(f(x)))e_{Y, t}(f)\di \meas_X(x)<\infty.
\end{equation}
\end{definition}
Let us provide a relationship between Sobolev maps and $0$-Sobolev maps.
\begin{proposition}[Compatibility, I]\label{propcompat}
Any Sobolev map from $U$ to $Y$ is a $0$-Sobolev map.  In particular any Lipschitz map from $U$ to $Y$ is a $0$-Sobolev map if $U$ is bounded.
\end{proposition}\label{rr}
\begin{proof}
Let $f:U \to Y$ be a Sobolev map.
Since for all $x \in U, z \in Y, t \in (0, \infty)$
\begin{equation}\label{ww}
p_Y(f(x), z, t)=\sum_ie^{-\lambda_i^Yt}\phi_i^Y(f(x))\phi_i^Y(z),
\end{equation}
we have
\begin{equation}
\dist_xp_Y(f(x), z, t)=\sum_ie^{-\lambda_i^Yt}\phi_i^Y(z)\dist_x(\phi_i^Y(f(x))) \quad \mathrm{in}\,\,L^2(T^*(U, \dist_X, \meas_X))
\end{equation}
because the inequalities (\ref{eq:eigenfunction}) and (\ref{uut}) imply that the equality (\ref{ww}) is satisfied in $H^{1, 2}(U, \dist_X, \meas_X)$ for fixed $z \in Y, t \in (0, \infty)$. Thus
\begin{align}
&\dist_xp_Y(f(x), z, t) \otimes \dist_xp_Y(f(x), z, t) \nonumber \\
&=\sum_{i, j}e^{-(\lambda_i^Y+\lambda_j^Y)t}\phi_i^Y(z)\phi_j^Y(z)\dist_x(\phi_i^Y(f(x)))\otimes \dist_x(\phi_j^Y(f(x))),\quad \mathrm{in}\,\, L^1((T^*)^{\otimes 2}(U, \dist_X, \meas_X)).
\end{align}
Integrating this (as the Bochner integral) over $Y$ with respect to $z$ yields
\begin{equation}
\int_Y\dist_xp_Y(f(x), z, t) \otimes \dist_xp_Y(f(x), z, t)\di \meas_Y(z)=\sum_ie^{-2\lambda_i^Yt}\dist_x(\phi_i^Y(f(x)))\otimes \dist_x(\phi_i^Y(f(x))).
\end{equation}
In particular by the Gaussian gradient estimate (\ref{eq:equi lip}) we have for $\meas_X$-a.e. $x \in U$
\begin{align}
\sum_ie^{-2\lambda_i^Yt}|\dist_x (\phi_i^Y(f(x)))|^2&=\left\langle \int_Y\dist_xp_Y(f(x), z, t) \otimes \dist_xp_Y(f(x), z, t)\di \meas_Y(z), g_X\right\rangle \nonumber \\
&=\int_Y\left|\dist_xp_Y(f(x), z, t)\right|^2\di \meas_Y(z) \nonumber \\
&\le \frac{CG_f(x)^2}{t\meas_Y(B_{\sqrt{t}}(f(x)))^2}\int_Y \exp \left( \frac{-2\dist_Y(f(x), z)^2}{5t}\right) \di \meas_Y(z) \nonumber \\
&\le \frac{CG_f(x)^2}{t\meas_Y(B_{\sqrt{t}}(f(x)))},
\end{align}
where we used (\ref{eq:improve}) and Cavalieri's formula (e.g \cite[Lem.2.3]{AHPT}). In particular
\begin{equation}
\limsup_{t \to 0^+}\int_Ut\meas_Y(B_{\sqrt{t}}(f(x)))\sum_ie^{-2\lambda_i^Yt}|\dist_x (\phi_i^Y(f(x)))|^2\di \meas_X(x)<\infty
\end{equation}
which completes the proof.
\end{proof}
It is worth pointing out that in the proof of Proposition \ref{propcompat} we immediately proved the following result. 
\begin{proposition}
Let $f:U \to Y$ be a Sobolev map. Then we see that $p_Y(f(\cdot), z, t) \in H^{1, 2}(U, \dist_X, \meas_X)$ holds for all $z \in Y, t \in (0, 1)$ and that the map $L^{\infty}(T(U, \dist_X, \meas_X)) \times L^{\infty}(T(U, \dist_X, \meas_X)) \to [0, \infty)$ defined by
\begin{equation}
(V_1, V_2) \mapsto \int_Y\int_U\dist_xp_Y(f(x), z, t)(V_1)\cdot \dist_xp_Y(f(x), z, t)(V_2)\di \meas_X(x)\di \meas_Y(z),
\end{equation}
defines an element of $L^1((T^*)^{\otimes 2}(U, \dist_X, \meas_X))$. This element is denoted by 
\begin{equation}
\int_Y\dist_xp_Y(f(x), z, t)\otimes \dist_xp_Y(f(x), z, t)\di \meas_Y(z).
\end{equation}
Then we have
\begin{equation}
\int_Y\dist_xp_Y(f(x), z, t) \otimes \dist_xp_Y(f(x), z, t)\di \meas_Y(z)=\sum_ie^{-2\lambda_i^Yt}\dist_x(\phi_i^Y(f(x)))\otimes \dist_x(\phi_i^Y(f(x)))
\end{equation}
with a pointwise estimate of the density
\begin{equation}\label{abasurbauawurb}
t \meas_Y(B_{\sqrt{t}}(f(x)))e_{Y, t}(f)\le C(K, N)G_f(x)^2,\quad \text{for $\meas_X$-a.e. $x \in U$}
\end{equation}
if $(Y, \dist_Y, \meas_Y)$ is an $\RCD(K, N)$ space.
\end{proposition}
By Proposition \ref{propcompat}, it is natural to ask whether any $0$-Sobolev map is a Sobolev map or not. This will be justified under assuming that the target space is non-collapsed (i.e. $\meas_Y=\mathcal{H}^N$) and that the image of $f$ is included in a ``weakly smooth subset of $Y$'' up to a $\meas_X$-negligible set (Theorem \ref{propcompatr}). In order to give the precise statement, we need to establish a bi-Lipschitz embeddability of $\Phi_t$ on a most part of $Y$ as in the next section.
\section{$(1\pm \epsilon)$-bi-Lipschitz embedding via heat kernel}\label{secembed}
%In order to give the precise behavior of the LHS of (\ref{eq:suplim}) on a \textit{smooth part}, let us try to improve the %regularity of the embedding $\Phi_t$ defined by (\ref{eq:embedding}).

Throughout the section, let us fix
\begin{itemize}
\item $K \in \mathbb{R}$, $N \in [1, \infty)$ and $d, v \in (0, \infty)$,
\item a non-collapsed compact $\RCD(K, N)$ space $(Y, \dist_Y, \mathcal{H}^N)$ with $\mathcal{H}^N(Y) \ge v$ and $\mathrm{diam}(Y, \dist_Y) \le d$.
\end{itemize}
In this setting the convergence results for $g_t:=\Phi_t^*g_{L^2}$ stated in subsection \ref{subsec:pullbak} can be stated as follows:
\begin{equation}\label{eq:pullbackconv}
\int_{Y}\left|g_Y-c_Nt^{(N+2)/2}g_t\right|^p\di \mathcal{H}^N \to 0
\end{equation}
holds as $t \to 0^+$ for any $p \in [1, \infty)$ with 
\begin{equation}\label{eq:gradeert}
\left|t^{(N+2)/2}g_t \right|(y) \le C(K, N, v)<\infty,\quad \text{for $\mathcal{H}^N$-a.e. $y \in Y$}
\end{equation}
for any $t \in (0, 1)$,
where $g_t:=(\Phi_t)^*g_{L^2}$ and $c_N$ is a positive constant depending only on $N$. Recall our notation in subsection \ref{susectionheat}, 
denote by $\{\lambda_i^Y\}_i$ the spectrum of $-\Delta_Y$ counted with multiplicities, and fix corresponding eigenfunctions $\{\phi_i^Y\}_i$ with $\|\phi_i^Y\|_{L^2}=1$.
Then letting 
\begin{equation}\label{eq:normalized}
\tilde{\Phi}_t:=c_N^{1/2}t^{(N+2)/4}\Phi_t,\quad \tilde{\Phi}_t^{\ell^2}:=c_N^{1/2}t^{(N+2)/4}\Phi_t^{\ell^2}
\end{equation}
shows for any $p \in [1, \infty)$, as $t \to 0^+$
\begin{equation}
\tilde{\Phi}_t^*g_{L^2} = (\tilde{\Phi}_t^{\ell^2})^*g_{\ell^2} \to g_X,\quad \mathrm{in}\,\, L^p((T^*)^{\otimes 2}(Y, \dist_Y, \mathcal{H}^N)).
\end{equation}
Finally for any $l \in \mathbb{N}$ we will also discuss the truncated map $\tilde{\Phi}_t^l:Y \to \mathbb{R}^l$ defined by
\begin{equation}\label{truncatedmap}
\tilde{\Phi}^l_t(y):=(c_N^{1/2}t^{(N+2)/4}e^{-\lambda_i^Yt}\phi_i^Y(y))_{i=1}^l.
\end{equation}
\subsection{Smoothable point}
Our goals in this section are to define a \textit{smoothable point of $(Y, \dist_Y, \mathcal{H}^N)$ via the heat kernel} and to prove
\begin{enumerate}
\item almost all points are smoothable (Proposition \ref{propdensereg}),
\item any smoothable point is regular (Proposition \ref{prop:regular}).
\end{enumerate}
\begin{definition}[Smoothable point via heat kernel]
For all $\epsilon, t, \tau \in (0, \infty)$, a point $y \in Y$ is an \textit{$(\epsilon, t, \tau)$-smoothable point} if 
\begin{equation}
\sup_{r \in (0, \tau]}\frac{1}{\mathcal{H}^N(B_r(y))}\int_{B_r(y)}|g_Y-c_Nt^{(N+2)/2}g_t|\di \mathcal{H}^N\le \epsilon.
\end{equation}  
We denote by $\mathcal{R}_Y(\epsilon, t, \tau)$ the set of all $(\epsilon, t, \tau)$-smoothable points. Let $\mathcal{R}_Y(\epsilon, t):=\mathcal{R}_Y(\epsilon, t, d)$ (recall that $d$ is an upper bound of $\mathrm{diam}(Y, \dist_Y)$). Moreover for any convergent sequence $t_i \to 0^+$, let us denote by $\mathcal{R}_Y(\{t_i\}_i)$ the set of all points $y \in Y$ satisfying 
\begin{equation}
\lim_{i \to \infty}\left(\sup_{r \in (0, \infty)}\frac{1}{\mathcal{H}^N(B_r(y))}\int_{B_r(y)}|g_Y-c_Nt_i^{(N+2)/2}g_{t_i}|\di \mathcal{H}^N\right)=0,
\end{equation}
in other words, 
\begin{equation}
\mathcal{R}_Y(\{t_i\}_i)=\bigcap_{\epsilon \in (0, 1)} \bigcup_i \bigcap_{j \ge i}\mathcal{R}_Y(\epsilon, t_j, d).
\end{equation}
The set $\mathcal{R}_Y(\{t_i\}_i)$ is called the \textit{smooth part} of $(Y, \dist_Y, \mathcal{H}^N)$ with respect to $\{t_i\}_i$.
\end{definition}
Let us prove that almost all points are smoothable.
\begin{proposition}\label{propdensereg}
For any convergent sequence $t_i \to 0^+$ there exists a subsequence $\{t_{i(j)}\}_j$ such that $\mathcal{H}^N(Y \setminus \mathcal{R}_Y(\{t_{i(j)}\}_j))=0$ holds.
\end{proposition}
\begin{proof}
If
\begin{equation}
\int_{Y}|g_Y-c_Nt^{(N+2)/2}g_t|\di \mathcal{H}^N\le \epsilon
\end{equation}
holds for some $\epsilon \in (0, 1)$ and some $t \in (0, 1)$, then the maximal function theorem (cf. \cite{Heinonen}) yields 
\begin{equation}
\mathcal{H}^N\left(Y \setminus \mathcal{R}_Y(\epsilon^{1/2}, t)\right) \le \frac{C(K, N)}{\epsilon^{1/2}}\int_Y|g_Y-c_Nt^{(N+2)/2}g_t|\di \mathcal{H}^N \le C(K, N)\epsilon^{1/2}.
\end{equation}
Thus thanks to (\ref{eq:pullbackconv}) with this observation, there exists a subsequence $\{t_{i(j)}\}_j$ such that 
\begin{equation}
\int_Y|g_Y-c_Nt_{i(j)}^{(N+2)/2}g_{t_{i(j)}}|\di \mathcal{H}^N\le 4^{-j},\quad \mathcal{H}^N\left(Y\setminus \mathcal{R}_Y(2^{-j}, t_{i(j)})\right)\le C(K, N)2^{-j}.
\end{equation}
Thus letting $\tilde{\mathcal{R}}:=\bigcup_k \bigcap_{j \ge k}\mathcal{R}_Y(2^{-j}, t_{i(j)})$ shows $\tilde{\mathcal{R}}\subset \mathcal{R}_Y(\{t_{i(j)}\}_j)$ with 
\begin{equation}
\mathcal{H}^N\left(Y \setminus \tilde{\mathcal{R}}\right)=\lim_{k \to \infty}\mathcal{H}^N\left( Y \setminus \bigcap_{j \ge k} \mathcal{R}_Y(2^{-j}, t_{i(j)})\right)\le \lim_{k \to \infty}\sum_{j \ge k}\mathcal{H}^N\left( Y \setminus \mathcal{R}_Y(2^{-j}, t_{i(j)})\right)=0,
\end{equation}
which completes the proof.
\end{proof}
\begin{proposition}\label{prop:regular}
There exists a constant $\delta_N \in (0, 1)$ depending only on $N$ such that if either
\begin{equation}\label{eq:liminf}
\liminf_{r \to 0^+}\frac{1}{\mathcal{H}^N(B_r(y))}\int_{B_r(y)}|g_Y-c_Nt^{(N+2)/2}g_t|\di \mathcal{H}^N\le \delta_N
\end{equation}
or 
\begin{equation}\label{eq:liminf2}
\liminf_{r \to 0^+}\frac{1}{\mathcal{H}^N(B_r(y))}\int_{B_r(y)}|g_Y-\tilde{c}_Nt\mathcal{H}^N(B_{\sqrt{t}}(\cdot))g_t|\di \mathcal{H}^N\le \delta_N
\end{equation}
is satisfied for some $t \in (0, \infty)$, then $y$ is an $N$-dimensional regular point. In particular we have
\begin{equation}
\mathcal{R}_Y(\delta_N, t, \tau )\subset \mathcal{R}_N, \quad \forall t,\,\,\forall \tau \in (0, \infty).
\end{equation}
Thus $\mathcal{R}_Y(\{t_i\}_i)\subset \mathcal{R}_N$ for any convergent sequence $t_i \to 0^+$.
\end{proposition}
\begin{proof}
We give only a proof in the case when (\ref{eq:liminf}) is satisfied because the proof in the other case is similar.
Fix $\delta \in (0, 1)$ which will be determined later. Let $\epsilon_0$ denote the LHS of (\ref{eq:liminf}) and assume $\epsilon_0<\delta$. Take a minimizing sequence $r_i \to 0^+$ as in the LHS of (\ref{eq:liminf}), and find $l \in \mathbb{N}$ with 
\begin{equation}
c_Nt^{(N+2)/2}\sum_{i=l+1}^{\infty}e^{-2\lambda^Y_it}\|\dist \phi_i^Y\|_{L^{\infty}}^2<\frac{\delta-\epsilon_0}{2},
\end{equation} 
where we used the inequality (\ref{eq:eigenfunction}) in order to get the existence of such an $l$. Thus we have
\begin{equation}\label{eq:est riem}
\limsup_{i \to \infty}\frac{1}{\mathcal{H}^N(B_{r_i}(y))}\int_{B_{r_i}(y)}\left|g_Y-c_Nt^{(N+2)/2}\sum_{i=1}^le^{-2\lambda_i^Yt}\dist \phi^Y_i \otimes \dist \phi^Y_i\right|\di \mathcal{H}^N<\frac{\delta+\epsilon_0}{2}.
\end{equation}
After passing to a subsequence we have
\begin{equation}
\left(Y, r_j^{-1}\dist_Y, \mathcal{H}^N_{r_j^{-1}\dist_Y}, y\right) \stackrel{\mathrm{pmGH}}{\to} \left(T_yY, \dist_{T_yY}, \mathcal{H}^N, 0_y\right)
\end{equation}
for some tangent cone $(T_yY, \dist_{T_yY}, \mathcal{H}^N, 0_y)$ of $(Y, \dist_Y, \mathcal{H}^N)$ at $y$. Let us define functions on $(Y, r_j^{-1}\dist_Y)$ by
\begin{equation}
\overline{\phi}_{i, j}:=\frac{c_N^{1/2}t^{(N+2)/4}e^{-\lambda_i^Yt}}{r_j}\left( \phi_i^Y-\frac{1}{\mathcal{H}^N(B_{r_j}(y))}\int_{B_{r_j}(x)}\phi_i^Y\di \mathcal{H}^N\right).
\end{equation}
Thanks to Theorem \ref{spectral2} {\color{blue}with local $(2,2)$-Poincar\'e inequality}, after passing to a subsequence again, there exists a family of Lipschitz harmonic functions $\{\overline{\phi}_i\}_{i=1}^l$ on $T_yY$ such that $\overline{\phi}_{i, j}$ $H^{1, 2}_{\mathrm{loc}}$-strongly converge to $\overline{\phi}_i$ on $T_yY$. Since $(T_yY, \dist_{T_yY}, \mathcal{H}^N, 0_y)$ is the metric measure cone over a non-collapsed $\RCD(N-2, N-1)$ space \cite[Prop.2.8]{DG} (see also \cite[Th.1.1]{DG2}), Proposition \ref{prop:linearfunction} shows that any  Lipschitz harmonic function $f$ on $T_yY$ is linear. Note that (\ref{eq:est riem}) implies
\begin{equation}\label{eq:splitting}
\frac{1}{\mathcal{H}^N(B_1(0_y))}\int_{B_1(0_y)}\left|g_{T_yY}-\sum_{i=1}^l\dist \overline{\phi}_i\otimes \dist \overline{\phi}_i\right|\di \mathcal{H}^N<\frac{\delta+\epsilon_0}{2}.
\end{equation}
Let us denote by $m$ the maximal dimension of the Euclidean factor $\mathbb{R}^m$ coming from $\{\overline{\phi}_i\}_{i=1}^l$. Then $T_yY$ is isometric to $\mathbb{R}^m\times Z$ for some non-collapsed $\RCD(K, N-m)$ space $(Z, \dist_Z, \mathcal{H}^{N-m})$. If $Z$ is not a single point (that is, $m<N$), then (\ref{eq:splitting}) with Fubini's theorem yields
\begin{equation}\label{eq:di}
\frac{1}{\mathcal{H}^{N-m}(B_1(z))}\int_{B_1(z)}|g_Z|\di \mathcal{H}^{N-m}<\frac{(\delta+\epsilon_0)C_N}{2}\le \delta C_N
\end{equation}
where $C_N$ is a positive constant depending only on $N$. Since the LHS of (\ref{eq:di}) is equal to $(N-m)^{1/2} \ge 1$ and the RHS is smaller than $1$ if $\delta$ is smaller than $1/C_N$, which is a contradiction. Thus $Z$ must be a single point.  In particular we know
\begin{equation}
\lim_{r \to 0^+}\frac{\mathcal{H}^N(B_{r}(y))}{\mathcal{H}^N(B_{r}(0_N))}=\lim_{i \to \infty}\frac{\mathcal{H}^N(B_{r_i}(y))}{\mathcal{H}^N(B_{r_i}(0_N))}=1,
\end{equation}
which completes the proof because of Theorem \ref{thm:bishop}.
\end{proof}

\subsection{Locally bi-Lipschitz embedding}
In order to establish a bi-Lipschitz embeddability of $\Phi_t$ on a large part of $Y$, we need a quantitative estimate for a Gromov-Hausdorff distance as follows (see the beginning of this section \ref{secembed} for the setting).
\begin{proposition}\label{prop:quantapp}
For all $\epsilon \in (0,1)$ and $t \in (0, \infty)$ there exists $r_0:=r_0(\epsilon, K, N, d, t) \in (0, 1)$ such that if for some $r \in (0, r_0)$ the following
\begin{equation}
\frac{1}{\mathcal{H}^N(B_r(y))}\int_{B_r(y)}|g_Y-c_Nt^{(N+2)/2}g_t|\di \mathcal{H}^N\le \epsilon
\end{equation}
holds and that $(Y, r^{-1}\dist_Y, y)$ is $r_0$-pointed Gromov-Hausdorff close to $(\mathbb{R}^N, \dist_{\mathbb{R}^N}, 0_N)$,
then the map $\tilde{\Phi}_t|_{B_r(y)}$ gives a $3\epsilon r$-Gromov-Hausdorff approximation to the image $\tilde{\Phi}_t(B_r(y))$ which is also $3\epsilon r$-Gromov-Hausdorff close to $B_r(0_N)$ ( recall (\ref{eq:normalized}) for the definition of $\tilde{\Phi}_t$).
\end{proposition}
\begin{proof}
The proof is done by contradiction. If not, then there exist a sequence of pointed non-collapsed compact $\RCD(K, N)$ spaces $(Y_i, \dist_{Y_i}, \mathcal{H}^N, y_i)$ with $\mathrm{diam}(Y_i, \dist_{Y_i})\le d$, and a sequence of $r_i \to 0^+$ with 
\begin{equation}
\quad \frac{1}{\mathcal{H}^N(B_{r_i}(y_i))}\int_{B_{r_i}(y_i)}|g_{Y_i}-c_Nt^{(N+2)/2}g_t^{Y_i}|\di \mathcal{H}^N\le \epsilon
\end{equation}
and
\begin{equation}\label{ddbnrbbaus}
\left(Y_i, r_i^{-1}\dist_{Y_i}, \mathcal{H}^N_{r_i^{-1}\dist_{Y_i}}, x_i\right) \stackrel{\mathrm{pmGH}}{\to} \left(\mathbb{R}^N, \dist_{\mathbb{R}^N}, \mathcal{H}^N, 0_N\right)
\end{equation}
such that one of the following holds.
\begin{enumerate}
\item[($\star$)] $\tilde{\Phi}_t^{Y_i}|_{B_{r_i}(y_i)}$ does not give a $3\epsilon r_i$-Gromov-Hausdorff approximation to the image  $\tilde{\Phi}_t^{Y_i}(B_{r_i}(y_i))$.
\item[($\star \star$)] $\tilde{\Phi}_t^{Y_i}(B_{r_i}(y_i))$ is not $3\epsilon r_i$-Gromov-Hausdorff close to $B_{r_i}(0_N)$. 
\end{enumerate}
Note that we used Theorem \ref{GHmGH} in order to get (\ref{ddbnrbbaus}).

Let us define functions on $(\overline{Y}_i, \dist_{\overline{Y}_i}):=(Y_i, r_i^{-1}\dist_{Y_i})$ by
\begin{equation}
\overline{\phi}_{i, j}:=\frac{c_N^{1/2}t^{(N+2)/4}e^{-\lambda_{j}^{Y_i}t}}{r_i}\left(\phi_{j}^{Y_i}-\frac{1}{\mathcal{H}^N(B_{r_i}(y_i))}\int_{B_{r_i}(y_i)}\phi_{j}^{Y_i}\di \mathcal{H}^N\right).
\end{equation}
Then (\ref{eq:eigenfunction}) allows us to define the map $\overline{\Phi}_i:\overline{Y}_i \to \ell^2$ by 
$\overline{\Phi}_i:=\left( \overline{\phi}_{i, j}\right)_{j=1}^{\infty}$.
Moreover thanks to Theorem \ref{spectral2} and Proposition \ref{prop:compactness} with (\ref{eq:eigenfunction}), after passing to a subsequence there exists a Lipschitz map $\overline{\Phi}:\mathbb{R}^N \to \ell^2$ such that the following hold.
\begin{itemize}
\item $\overline{\Phi}_i$ uniformly converge to $\overline{\Phi}=(\overline{\phi}_j)_{j=1}^{\infty}$ on any bounded subset of $\mathbb{R}^N$.
\item $\overline{\phi}_{i, j}$ $H^{1, 2}_{\mathrm{loc}}$-strongly converge to $\overline{\phi}_i$ on $\mathbb{R}^N$.
\item Each $\overline{\phi}_i$ is linear.
\item The $L^{\infty}$-tensors on $\overline{Y}_i$
\begin{equation}
\sum_{j=1}^{\infty}\dist \overline{\phi}_{i, j}\otimes \dist \overline{\phi}_{i, j}
\end{equation}
$L^2_{\mathrm{loc}}$-strongly converge to the $L^{\infty}$-tensor 
\begin{equation}
\overline{\Phi}^*g_{\ell^2}=\sum_{j=1}^{\infty}\dist \overline{\phi}_j \otimes \dist \overline{\phi}_j
\end{equation}
on $\mathbb{R}^N$. 
\end{itemize}
Note
\begin{align}\label{absaurauubairub}
&\frac{1}{\mathcal{H}^N(B_1(0_N))}\int_{B_1(0_N)}\left| g_{\mathbb{R}^N}-\sum_{j=1}^{\infty}\dist \overline{\phi}_j \otimes \dist \overline{\phi}_j \right| \di \mathcal{H}^N \nonumber \\
&=\lim_{i \to \infty}\frac{1}{\mathcal{H}^N(B_1^{\dist_{\overline{Y}_i}}(y_i))}\int_{B_1^{\dist_{\overline{Y}_i}}(y_i)}\left| g_{\overline{Y}_i}-\sum_{j=1}^{\infty}\dist \overline{\phi}_{i, j}\otimes \dist \overline{\phi}_{i, j}\right| \di \mathcal{H}^N_{\dist_{\overline{Y}_i}} \nonumber \\
&=\lim_{i \to \infty}\frac{1}{\mathcal{H}^N(B_{r_i}(y_i))}\int_{B_{r_i}(y_i)}\left| g_{Y_i}-c_Nt^{(N+2)/2}g_t^{Y_i}\right| \di \mathcal{H}^N\le \epsilon.
\end{align}
Since $| g_{\mathbb{R}^N}-\sum_{j=1}^{\infty}\dist \overline{\phi}_j \otimes \dist \overline{\phi}_j|$ is constant because of the linearity of $\overline{\phi}_j$, by (\ref{absaurauubairub}), we have $|g_{\mathbb{R}^N}-\overline{\Phi}^*g_{\ell^2}|\le \epsilon$ on $\mathbb{R}^N$. Thus $\overline{\Phi}$ is a $(1\pm \epsilon)$-bi-Lipschitz embedding from $\mathbb{R}^N$ to $\ell^2$.
Then the local uniform convergence of  $\overline{\Phi}_i$ to $\overline{\Phi}$ with (\ref{texteq})  {\color{blue} for $\overline{\Phi}$} shows that for any sufficiently large $i$, it holds that  $\tilde{\Phi}_t^{Y_i}|_{B_{r_i}(x_i)}$ gives a $3\epsilon r_i$-Gromov-Hausdorff approximation to the image which is also $3\epsilon r_i$-Gromov-Hausdorff close to $B_{r_i}(0_N)$, {\color{blue}because $\overline{\Phi}_i$ is obtained by a rescaling and a translation of $\tilde{\Phi}_t^{Y_i}$'' after ``to $B_{r_i}(0_N)$}. This contradicts ($\star$) and ($\star \star$).
\end{proof}

\begin{theorem}[Bi-Lipschitz embeddability of $\Phi_t$]\label{prop:bilip}
For all $\epsilon \in (0, 1/3), t, \tau \in (0, \infty)$ and $y \in \mathcal{R}_Y(\epsilon, t, \tau)$ there exists $r_1 \in (0, 1)$ such that $\tilde{\Phi}_t|_{B_{r_1}(y) \cap \mathcal{R}_Y(\epsilon, t, \tau)}$ is a $(1\pm 3\epsilon)$-bi-Lipschitz embedding.
\end{theorem}
\begin{proof}
Let $\delta \in (0, 1)$ be a small number which will be determined later. We can find $s_0 \in (0, \delta )$ such that $(Y, s_0^{-1}\dist_Y, y)$ $\delta$-pointed Gromov-Hausdorff close to $(\mathbb{R}^N, \dist_{\mathbb{R}^N}, 0_N)$. In particular by Theorem \ref{prop:almostrigidity}  we know that 
\begin{equation}
1-\Psi(\delta; K, N)\le \frac{\mathcal{H}^N(B_{s_0}(z))}{\omega_Ns_0^N} \le 1+\Psi(\delta; K, N), \quad \forall z \in B_{s_0}(y).
\end{equation}
Applying the Bishop and Bishop-Gromov inequalities yields 
\begin{equation}
1-\Psi(\delta; K, N)\le \frac{\mathcal{H}^N(B_{s}(z))}{\omega_Ns^N} \le 1+\Psi(\delta; K, N), \quad \forall z \in B_{s_0}(y),\,\,\forall s \in (0, s_0].
\end{equation}
Let $r_1:=s_0^2$. Then for any $s \in (0, r_1]$, applying Theorem \ref{prop:almostrigidity} again for the rescaled space $(Y, s^{-1}\dist, \mathcal{H}^N_{s^{-1}\dist}, x)$ shows that the rescaled space is $\Psi(\delta; K, N)$-pointed measured Gromov-Hausdorff close to $(\mathbb{R}^N, \dist_{\mathbb{R}^N}, \mathcal{H}^N, 0_N)$. With no loss of generality we can assume 
$\Psi(\delta;K,N)<r_0$, where $r_0$ is as in Proposition \ref{prop:quantapp}. Thus Proposition \ref{prop:quantapp} yields that   $\tilde{\Phi}_t|_{B_s(y)}$ gives a $3\epsilon s$-Gromov-Hausdorff approximation to the image. In particular for any $z, w \in B_{r_1}(y)$ with $z \neq w$, letting $s:=\dist (z, w) \in (0, 1)$ shows 
\begin{equation}
\left|\|\tilde{\Phi}_t(z)-\tilde{\Phi}_t(w)\|_{L^2} -\dist_Y (z, w) \right|\le 3\epsilon s,
\end{equation}
namely
\begin{equation}
(1-3\epsilon)\dist_Y(z, w) \le \|\tilde{\Phi}_t(z)-\tilde{\Phi}_t(w)\|_{L^2}\le (1+3\epsilon)\dist_Y (z, w)
\end{equation}
which completes the proof.
\end{proof}
By an argument similar to the proof of Theorem \ref{prop:bilip}, we have the following.
\begin{theorem}\label{prop:finitedimen}
For all $\epsilon \in (0, 1/6), \delta \in (0, \epsilon)$, $t, \tau \in (0, \infty)$, let $y \in \mathcal{R}_Y(\epsilon, t, \tau)$. 
Then there exists $r_2 \in (0, 1)$ such that the restriction of the truncated map $\tilde{\Phi}_t^l:Y \to \mathbb{R}^l$ defined by (\ref{truncatedmap})
to $B_{r_2}(y) \cap \mathcal{R}_Y(\epsilon, t, \tau)$ is a $(1 \pm 3(\epsilon+\delta))$-bi-Lipschitz embedding for any $l \in \mathbb{N}$ with 
\begin{equation}
c_Nt^{(N+2)/2}\sum_{i=l+1}^{\infty}e^{-2\lambda_it}\|\dist \phi_i^Y\|_{L^{\infty}}^2<\delta.
\end{equation} 
\end{theorem}

\begin{remark}\label{remarkporte}
Portegies proved in \cite{P} that for all $\epsilon, \tau, d \in (0, \infty)$ and any $K \in \mathbb{R}$ there exists $t_0:=t_0(n, K, \epsilon, \tau, d) \in (0, 1)$ such that for any $t \in (0, t_0)$ there exists $N_0:=N_0(n, K, \epsilon, \tau, d, t) \in \mathbb{N}$ such that if an $n$-dimensional closed Riemannian manifold $(M^n, g)$ satisfies $\mathrm{diam}(M^n, \dist_g) \le d, \mathrm{Ric}_{M^n}^g \ge K$ and $\mathrm{inj}_{M^n}^g\ge \tau$, where $\mathrm{inj}_{M^n}^g$ denote the injectivity radius, then for any $l \in \mathbb{N}_{\ge N_0}$, the map $\tilde{\Phi}_{t}^l:M^n \to \mathbb{R}^{l}$ is a smooth embedding with
\begin{equation}\label{78}
\|(\tilde{\Phi}_{t}^l)^*g_{\mathbb{R}^{l}}-g\|_{L^{\infty}}<\epsilon.
\end{equation}
In particular we have $M^n=\mathcal{R}_{M^n}(\epsilon, t)$ for any $t \in (0, t_0)$. Therefore Proposition \ref{prop:bilip} and Theorem \ref{prop:finitedimen} can be regarded as a generalization of his result to the $\RCD$ setting. Moreover Propositions \ref{prop:regular} and \ref{prop:dense} below tell us that this observation cannot be extended over the singular set. See also Remark \ref{remnonsmoo}. A non-smooth example along this direction can be found in \cite[Exam.5.1]{Peters}. The smooth embeddability part of $\tilde{\Phi}_t^l$ in his result will be generalized to the $\RCD$ setting in the next section too after replacing ``smooth'' by ``bi-Lipschitz''. See Proposition \ref{prop:portegies}.
\end{remark}

\subsection{Characterization of weakly smooth $\RCD$ spaces}
Our goal in this section is to give a proof of Theorem \ref{themequivalence}. Let us begin with giving the following rigidity result, where recall that a pointed metric measure space $(W, \dist_W, \meas_W, w)$ is said to be a \textit{tangent cone at infinity} of an $\RCD(0, N)$ space $(Z, \dist_Z, \meas_Z)$ if there exists a sequence $R_i \to \infty$ such that 
\begin{equation}
\left( Z, R_i^{-1}\dist_Z, \meas_Z(B_{R_i}(z))^{-1}\meas_Z, z\right) \stackrel{\mathrm{pmGH}}{\to} (W, \dist_W, \meas_W, w)
\end{equation}
holds for some (or equivalently all) $z \in Z$.
\begin{theorem}\label{prop:rigiditynonnegative}
Let $(Z, \dist_Z, \meas_Z)$ be an $\RCD(0, N)$ space whose essential dimension is equal to $n$, and let $\Phi=(\phi_i)_i:Z \to \ell^2$ be a bi-Lipschitz embedding. Assume that each $\phi_i$ is a harmonic function on $Z$. Then we have the following.
\begin{enumerate}
\item Any tangent cone at infinity of $(Z, \dist_Z, \meas_Z)$ is isometric to $(\mathbb{R}^n, \dist_{\mathbb{R}^n}, \omega_n^{-1}\mathcal{H}^n, 0_n)$.
\item After relabeling, the map $\Phi^n:=(\phi_i)_{i=1}^n:Z \to \mathbb{R}^n$ gives a bi-Lipschitz homeomorphism.
\end{enumerate}
\end{theorem}
\begin{proof}
We follow a blow-down argument in \cite{ChCM} as follows. Fix a sequence $R_i \to \infty$.
After passing to a subsequence, 
\begin{equation}\label{eq;convmet}
(Z_i, \dist_{Z_i}, \meas_{Z_i}, z):=\left(Z, R_i^{-1}\dist_Z, \meas_Z(B_{R_i}(z))^{-1}\meas_Z, z\right) \stackrel{\mathrm{pmGH}}{\to} (W, \dist_W, \meas_W, w)
\end{equation}
holds for some pointed  $\RCD(0, N)$ space $(W, \dist_W, \meas_W, w)$.
Let us define functions on $(Z_i, \dist_{Z_i})$ by
\begin{equation}
\overline{\phi}_{j, i}:=\frac{1}{R_i}\left( \phi_j-\frac{1}{\meas_Z(B_{R_i}(z))}\int_{B_{R_i}(z)}\phi_j\di \meas_Z\right).
\end{equation}
Note that it follows from the Lipschitz continuity of $\Phi$ that each $\phi_i$ is a Lipschitz harmonic function (thus $\overline{\phi}_{j, i}$ is too).
As already discussed in the proofs of Propositions \ref{prop:regular} and \ref{prop:quantapp}, after passing to a subsequence again, Theorem \ref{spectral2} yields that there exist Lipschitz harmonic functions $\overline{\phi}_j$ on $W$ such that $\overline{\phi}_{j, i}$ $H^{1, 2}_{\mathrm{loc}}$-strongly converge to $\overline{\phi}_j$ on $W$. Applying \cite[Lem.3.1]{MondinoNaber} for the rescaled space $(Z_i, \dist_{Z_i}, \meas_{Z_i})$ there exists $\psi_i \in D(\Delta_Z) \cap \mathrm{Lip}(Z, \dist_Z)$ such that $0\le \psi_i \le 1$, that $\supp \psi_i \subset B_{2R_i}(z)$, that $\psi_i|_{B_{R_i}(z)}\equiv 1$, and that $R_i|\nabla \psi_i| +R_i^2|\Delta_Z\psi_i|\le C(N)$.
Recall that the Bochner inequality implies that $|\nabla \phi_i|^2$ is subharmonic. Thus we can apply the mean value theorem at infinity \cite[Th.5.4]{HKX} to get 
\begin{equation}\label{eq:li}
\lim_{R\to \infty}\frac{1}{\meas_Z(B_R(z))}\int_{B_R(z)}|\nabla \phi_i|^2\di \meas_Z=\|\nabla \phi_i\|_{L^{\infty}}^2=(\mathbf{Lip}\phi_i)^2,
\end{equation}
where we used \cite[Prop.1.10]{GV} in the last equality.
The Bochner formula yields
\begin{align}
\frac{1}{\mathcal{H}^N(B_{R_i}(z))}\int_{B_{R_i}(z)}|\mathrm{Hess}_{\phi_j}|^2\di \meas_Z &\le  \frac{2^N}{\mathcal{H}^N(B_{2R_i}(z))}\int_{B_{2R_i}(z)}\psi_i|\mathrm{Hess}_{\phi_j}|^2\di \meas_Z \nonumber \\
&\le \frac{2^{N-1}}{\mathcal{H}^N(B_{2R_i}(z))}\int_{B_{2R_i}(z)}\Delta_{\color{blue}Z}\psi_i\cdot (|\nabla \phi_j|^2-(\mathbf{Lip}\phi_i)^2)\di \meas_Z \nonumber \\
&\le 2^{N-1}R_i^{-2}o(1)
\end{align}
because of (\ref{eq:li}), {\color{blue}where we used the Bochner inequality in the second inequality above and the last inequality just comes from (\ref{eq:li}).} Thus as $i \to \infty$
\begin{equation}
\frac{R_i^2}{\mathcal{H}^N(B_{R_i}(z))}\int_{B_{R_i}(z)}|\mathrm{Hess}_{\phi_j}|^2\di \meas_Z \to 0.
\end{equation}
Therefore applying \cite[Th.10.3]{AmbrosioHonda} with a good cut-off by \cite[Lem.3.1]{MondinoNaber} we have $\mathrm{Hess}_{\overline{\phi}_j}=0$. Thus since this implies that $|\nabla \overline{\phi}_j|$ is constant, Theorem \ref{splitting} yields that each $\overline{\phi}_j$ is linear.

From now on let us prove
\begin{equation}\label{eee}
\sum_j\|\dist \phi_j\|_{L^{\infty}}^2<\infty.
\end{equation}
Take $L \in [1, \infty)$ satisfying that $\Phi$ is $L$-Lipschitz, and fix $l \in \mathbb{N}$. Since $\overline{\Phi}^l_i:=(\overline{\phi}_{j, i})_{j=1}^l :Z_i \to \mathbb{R}^l$ uniformly converge to $\overline{\Phi}^l:=(\overline{\phi}_j)_{j=1}^l :W \to \mathbb{R}^l$ on any bounded subset of $W$, we know that $\overline{\Phi}^l$ is $L$-Lipschitz. In particular combining (\ref{eq:bound}) with the linearity of $\overline{\phi}_j$ shows
\begin{equation}\label{abrabruairiua}
\sum_{j=1}^l\|\dist \overline{\phi}_j\|_{L^{\infty}}^2\le nL^2.
\end{equation}
Then the above arguments using the mean value theorem at infinity allows us to conclude
\begin{equation}\label{abraubaibrao}
\sum_{j=1}^l\|\dist \phi_j\|_{L^{\infty}}^2<\infty.
\end{equation}
Thus letting $l \to \infty$ in (\ref{abraubaibrao}) proves (\ref{eee}).

Then %the Poincar\'e inequality with 
(\ref{eee}) easily implies that for any $R \in (0, \infty)$ and any $\epsilon \in (0, 1)$, there exists $i_0 \in \mathbb{N}$ such that 
\begin{equation}\label{eq:referee req}
\sum_{j \ge i_0}\|\overline{\phi}_{j, i}\|_{L^{\infty}(B_R^{\dist_{Z_i}}(z))}<\epsilon,\quad \forall i \in \mathbb{N}
\end{equation}
holds. {\color{blue} For reader's convenience let us provide a proof of (\ref{eq:referee req}) as follows. Since the average of $\overline{\phi}_{j, i}$ over the unit ball is zero, we can find $z_{j, i} \in B_1^{\dist_{Z_i}}(z)$ with $\overline{\phi}_{j, i}(z_{j, i})=0$. Then for any $w \in B_R^{\dist_{Z_i}}(z)$, recalling $| \dist \phi_j| = |\dist \overline{\phi}_{j, i}|$, we have
\begin{equation}
\left| \overline{\phi}_{j, i}(w) \right| =\left| \overline{\phi}_{j, i}(w) - \overline{\phi}_{j, i}(z_{j, i})\right| \le \|\dist \phi_{j, i}\|_{L^{\infty}} \cdot  \dist_{Z_i}(w, z_{j, i}) \le (2R+2) \|\dist \phi_{j, i}\|_{L^{\infty}}. 
\end{equation}
Taking the sum with respect to $j \ge i_0$, we conclude because of (\ref{eee}).
}
In particular, thanks to Proposition \ref{prop:compactness}, we see that $\overline{\Phi}_i:=(\overline{\phi}_{j, i})_j$ uniformly converge to $\overline{\Phi}:=(\overline{\phi}_j)_j$ on any bounded subset of $W$ with respect to the convergence (\ref{eq;convmet}). In particular, $\overline{\Phi}$ is also a bi-Lipschitz embedding into $\ell^2$.
Then we denote by $\mathbb{R}^m$ the Euclidean factor coming from $\{\overline{\phi}_i\}_i$. Theorem \ref{splitting} shows that there exists an isometry from $\mathbb{R}^m \times \tilde{W}$ to $W$ for some non-collapsed $\RCD(0, N-m)$ space $(\tilde{W}, \dist_{\tilde{W}}, \meas_{\tilde{W}})$ such that $\overline{\Phi} \circ \iota (v, w_1)=\overline{\Phi} \circ \iota (v, w_2)$ holds for all $w_i \in \tilde{W} (i=1, 2)$ and $v \in \mathbb{R}^m$. The bi-Lipschitz embeddability of $\overline{\Phi}$ shows that $\tilde{W}$ is a single point. Thus $W$ is isometric to $\mathbb{R}^m$.

Let us prove $m=n$. The lower semicontinuity of essential dimensions in Proposition \ref{weakriem} shows $n \ge m$.
On the other hand, after relabeling, with no loss of generality we can assume that $\{\overline{\phi}_i\}_{i=1}^m$ is a family of linearly independent linear functions on $\mathbb{R}^m$ because of the bi-Lipschitz embeddability of $\overline{\Phi}$. Thus for any $i \in \mathbb{N}_{\ge m+1}$ there exist $a_{i, j} \in \mathbb{R} (j=0,1,\ldots, m)$ such that $\overline{\phi}_i=a_{i, 0}+\sum_{j=1}^ma_{i, j}\overline{\phi}_j$ holds. Applying the mean value theorem at infinity \cite[Th.5.4]{HKX} again for $|\nabla (\phi_i - \sum_{j=1}^ma_{i, j}\phi_j)|^2$ shows that $\phi_i-\sum_{j=1}^ma_{i, j}\phi_j$ is constant. From this obervation we know that the truncated map $\overline{\Phi}^m:Z \to \mathbb{R}^m$ is also bi-Lipschitz embedding because $\Phi$ is. Then Theorem \ref{thm:RN} proves $n \le m$. Thus $n=m$.  Therefore we have (1).

Finally let us prove (2). By an argument similar to the proof of Proposition \ref{prop:quantapp} (cf. the proof of \cite[Th.1.1]{honda20}), we can check that $(Z, \dist_Z)$ is Reifenberg flat, that is, the following holds.
\begin{itemize}
\item For any $\epsilon \in (0, 1)$, there exists $r_0 \in (0, 1)$ such that 
\begin{equation}
\dist_{GH}(B_r(\tilde{z}), B_r(0_n))\le \epsilon r,\quad \forall \tilde{z} \in Z,\,\,\forall r\in (0, r_0]
\end{equation}
holds.
\end{itemize}
In particular, thanks to \cite[Th.A.1.2 and A.1.3]{CheegerColding1}, we know that $Z$ is homeomorphic to an $n$-dimensional manifold. Thus by invariance of domain, $\overline{\Phi}^n(Z)$ is an open subset of $\mathbb{R}^n$. On the other hand, the bi-Lipschitz embeddability of $\overline{\Phi}^n$ into $\mathbb{R}^n$ yields that $\overline{\Phi}^n(Z)$ is a closed subset of $\mathbb{R}^n$. Thus $\overline{\Phi}^n(Z)=\mathbb{R}^n$.
\end{proof}
\begin{remark}
It is conjectured that in Theorem \ref{prop:rigiditynonnegative}, $(Z, \dist_Z, \meas_Z)$ is isometric to $(\mathbb{R}^n, \dist_{\mathbb{R}^n}, c\mathcal{H}^n)$ for some $c \in (0, \infty)$. Compare with the next corollary.
\end{remark}
\begin{corollary}\label{lemlemelem}
Under the same assumptions of Theorem \ref{prop:rigiditynonnegative}, if in addition $(Z, \dist_Z, \meas_Z)$ is a non-collapsed $\RCD(0, N)$ space, then $(Z, \dist_Z, \meas_Z)$ is isometric to $(\mathbb{R}^N, \dist_{\mathbb{R}^N}, \mathcal{H}^N)$.
\end{corollary}
\begin{proof}
Theorem \ref{prop:rigiditynonnegative} with Theorem \ref{GHmGH} implies
\begin{equation}
\lim_{R \to \infty}\frac{\mathcal{H}^N(B_{R}(z))}{\omega_NR^N}=1.
\end{equation}
By the Bishop and the Bishop-Gromov inequalities we have 
\begin{equation}
\mathcal{H}^N(B_R(z))=\omega_NR^N,\quad \forall R \in (0, \infty).
\end{equation}
Then the rigidity for the Bishop inequality \cite[Thm.1.6]{DG} (e.g. Theorem \ref{thm:bishop}) completes the proof.
\end{proof}
Similarly we can prove the following.
\begin{corollary}\label{lem:lem}
Let $(Z, \dist_Z, \mathcal{H}^N)$ be a non-collapsed $\RCD(0, N)$ space with Euclidean volume growth, namely
\begin{equation}\label{euclidean}
\lim_{R \to \infty}\frac{\mathcal{H}^N(B_R(z))}{\omega_NR^N}>0
\end{equation}
holds for some (or equivalently all) $z \in Z$.
Assume that there exists a Lipschitz map $\Phi=(\phi_i)_i:Z \to \ell^2$ and a subset $A$ of $Z$ such that the following hold.
\begin{enumerate}
\item The set $A$ is asymptotically dense in the sense:
\begin{itemize}
\item There exist sequences of $\epsilon_i \to 0, R_i \to \infty, L_i \to \infty$ such that $B_{L_iR_i}(z) \cap A$ is $\epsilon_iR_i$-dense in $B_{L_iR_i}(z)$ for any $i$,
\end{itemize}
\item each $\phi_i$ is a harmonic function on $Z$, 
\item the map $\Phi|_A$ is a bi-Lipschitz embedding.
\end{enumerate}
Then $(Z, \dist_Z, \mathcal{H}^N)$ is isometric to $(\mathbb{R}^N, \dist_{\mathbb{R}^N}, \mathcal{H}^N)$.
\end{corollary}
Note that in Corollary \ref{lem:lem}, the assumption (\ref{euclidean}) is necessary, for example, consider $Z=\mathbb{S}^1(1) \times \mathbb{R}$, $A=\{p\}\times \mathbb{R}$ and $\Phi(q, t)=t \in \mathbb{R} \subset \ell^2$.

Let us apply the above results to the embedding map $\Phi_t$. 
\begin{corollary}\label{prop:dense}
Let $A$ be a subset of $Y$. Assume that $\Phi_t|_A$ is a bi-Lipschitz embedding for some $t \in (0, \infty)$. Then
\begin{equation}
\mathrm{Den} (A) \subset \mathcal{R}_N
\end{equation}
In particular, if in addition $A$ is an open subset of $Y$, then $A\subset \mathcal{R}_N$.
\end{corollary}
\begin{proof}
Fix $y \in \mathrm{Den} (A)$. By an argument similar to the proof of Proposition \ref{prop:regular}, for any tangent cone $(T_yY, \dist_{T_yY}, \mathcal{H}^N, 0_y)$ of $(Y, \dist_Y, \mathcal{H}^N)$ at $y$, there exists a bi-Lipschitz embedding $\overline{\Phi}=(\overline{\phi}_i)_{i=1}^{\infty}:T_yY \to \ell^2$ such that each $\overline{\phi}_i$ is linear.
Then Corollary \ref{lemlemelem} shows that $(T_yY, \dist_{T_yY})$ is isometric to $(\mathbb{R}^N, \dist_{\mathbb{R}^N})$, which completes the proof.
\end{proof}
\begin{remark}\label{remnonsmoo}
It is known that there is a non-collapsed sequence of Riemannian metrics $g_i$ on $\mathbb{S}^2$ with non-negative sectional curvature such that the Gromov-Hausdorff limit space $(X, \dist_X)$ is compact and that the singular set $\mathcal{S}$ of $X$ is dense. See the example (2) of page 632 in \cite{OS}. In particular $(X, \dist_X, \mathcal{H}^2)$ is a non-collapsed compact $\RCD(0, 2)$ space. Then Corollary \ref{prop:dense} tells us that for any $t \in (0, \infty)$ the restriction of $\Phi_t^X$ to $\mathcal{S}$ is not locally bi-Lipschitz. 
\end{remark}
From now on let us discuss  the implication of a local bi-Lipschitz embeddability of $\Phi_t$ on an estimate on $|g_Y-c_Nt^{(N+2)/2}g_t|$;
\begin{proposition}\label{prop:quantitativelinfty}
For all $c, t \in (0, \infty)$ and $\epsilon \in (0, 1)$, there exist $r_2:=r_2(c, \epsilon, K, N, t, d ,v) \in (0, 1)$, $\delta_0:=\delta_0(c, \epsilon, K, N, t, d, v) \in (0, 1)$ and $L_0:=L_0(c, \epsilon, K, N, t, d, v) \in (1, \infty)$ such that for some $y \in Y$, some $r \in (0, r_2]$, some $L \in [L_0, \infty)$, some $cr$-dense subset $A$ of $B_{Lr}(y)$ satisfying that $\tilde{\Phi}_t|_A$ gives a $(1\pm \epsilon)$-bi-Lipschitz embedding, we have
\begin{equation}
\frac{1}{\mathcal{H}^N(B_r(y))}\int_{B_r(y)}|g_X-c_Nt^{(N+2)/2}g_t|\di \mathcal{H}^N<2\sqrt{N}\epsilon.
\end{equation}
\end{proposition}
\begin{proof}
The proof is done by a contradiction. If not, there exist a sequence of pointed non-collapsed compact $\RCD(K, N)$ spaces $(Y_i, \dist_{Y_i}, \mathcal{H}^N, y_i)$ with $\mathrm{diam}(Y_i, \dist_{Y_i}) \le d$ and $\mathcal{H}^N(Y_i)\ge v$, sequences of $r_i \to 0^+, \delta_i \to 0^+, L_i \to \infty$ and a sequence of $cr_i$-dense subset $A_i$ of $B_{L_ir_i}(y_i)$ such that $\tilde{\Phi}_t^{Y_i}|_{A_i}$ gives a $(1 \pm \epsilon)$-bi-Lipschitz embedding and that 
\begin{equation}\label{ddd}
\frac{1}{\mathcal{H}^N(B_{r_i}(y_i))}\int_{B_{r_i}(y_i)}|g_{Y_i}-c_Nt^{(N+2)/2}g_t^{Y_i}|\di \mathcal{H}^N\ge 2\sqrt{N}\epsilon.
\end{equation}
Note that the sequence of $\{(Y_i, \dist_{Y_i}, \mathcal{H}^N)\}_i$ is uniformly Ahlfors regular, that is,
\begin{equation}\label{rrtt}
C_1(K, N, d, v)r^N\le \mathcal{H}^N(B_r(z_i)) \le C_2(K, N, d, v)r^N, \quad \forall i,\,\,\forall z_i \in Y_i,\,\,\forall r \in (0, d].
\end{equation}
After passing to a subsequence, we have
\begin{equation}
\left(Y_i, r_i^{-1}\dist_{Y_i}, \mathcal{H}^N_{r_i^{-1}\dist_{Y_i}}, y_i\right) \stackrel{\mathrm{pmGH}}{\to} (Z, \dist_Z, \mathcal{H}^N, z)  
\end{equation}
for some non-collapsed $\RCD(0, N)$ space $(Z, \dist_Z, \mathcal{H}^N)$ which has Euclidean volume growth because of (\ref{rrtt}).
By an argument similar to the proof of Proposition \ref{prop:quantapp} there exists a Lipschitz map $\overline{\Phi}:Z \to \ell^2$ such that rescaled maps of $\tilde{\Phi}_t^{Y_i}$ uniformly converge to $\overline{\Phi}=(\overline{\phi}_i)_i$ on any bounded subset of $Z$ and that each $\overline{\phi}_i$ is harmonic.
Moreover we can find a $2c$-dense subset $A$ of $Z$ satisfying that for any $a \in A$ there exists a sequence of $a_i \in A_i$ such that $a_i \to a$. In particular $\overline{\Phi}|_{A}$ is a $(1\pm \epsilon)$-bi-Lipschitz embedding.  Thus Corollary \ref{lem:lem} shows that $(Z, \dist_Z, \mathcal{H}^N)$ is isometric to $(\mathbb{R}^N, \dist_{\mathbb{R}^N}, \mathcal{H}^N)$.  In particular each $\overline{\phi}_i$ is linear. Thus it follows from the $(1 \pm \epsilon)$-bi-Lipschitz embeddability of $\overline{\Phi}$ and the linearity of $\overline{\phi}_i$ that 
\begin{equation}
|g_{\mathbb{R}^N}-\overline{\Phi}^*g_{\ell^2}|_B\le \epsilon
\end{equation}
holds on $\mathbb{R}^N$.
Thus
\begin{equation}\label{eeqqqw}
\lim_{i \to \infty}\frac{1}{\mathcal{H}^N(B_{r_i}(x_i))}\int_{B_{r_i}(x_i)}\left|g_{Y_i}-c_Nt^{(N+2)/2}g_t^{X_i}\right| \di \mathcal{H}^N=\frac{1}{\omega_N}\int_{B_1(0_N)}|g_{\mathbb{R}^N}-\overline{\Phi}^*g_{\ell^2}|\di \mathcal{H}^N \le \sqrt{N}\epsilon,
\end{equation}
where we used Proposition \ref{prop:ineq}. Then (\ref{eeqqqw}) contradicts (\ref{ddd}).
\end{proof}
\begin{theorem}\label{prop:bilipch}
Let $A$ be a Borel subset of $Y$ and let $t_j \to 0^+$ be a convergent sequence. If for any $\epsilon \in (0, 1)$ there exists $i_0 \in \mathbb{N}$ such that $\tilde{\Phi}_{t_j}|_A$ is a locally $(1 \pm \epsilon)$-bi-Lipschitz embedding for any $j \in \mathbb{N}_{\ge i_0}$, 
then
\begin{equation}
\|g_Y-c_Nt_i^{(N+2)/2}g_{t_i}\|_{L^{\infty}(A)} \to 0.
\end{equation}
\end{theorem}
\begin{proof}
Let $\epsilon \in (0, 1)$ and take $i_0 \in \mathbb{N}$ as in the assumption. Fix $j \in \mathbb{N}_{\ge i_0}$ and take $y \in A$. Then there exists $r_3:=r_3(y) \in (0, 1)$ such that $\tilde{\Phi}_{t_j}|_{B_{r_3}(y) \cap A}$ is a $(1\pm \epsilon)$-bi-Lipschitz embedding. Then applying Proposition \ref{prop:quantitativelinfty} for all $z \in B_{r_3/2}(y)\cap \mathrm{Leb}(A)$ and sufficiently small $r \in (0, r_3/2)$, there exists $r_4 \in (0, 1)$ such that 
\begin{equation}
\frac{1}{\mathcal{H}^N(B_r(z))}\int_{B_r(z)}|g_Y-c_Nt_j^{(N+2)/2}g_{t_j}|\di \mathcal{H}^N< 2\sqrt{N}\epsilon,\quad \forall r \in (0, r_4),\,\,\forall z \in B_{r_3/2}(y)\cap A.
\end{equation}
In particular Lebesgue differentiation theorem yields
\begin{equation}\label{eq:unifo}
\|g_Y-c_Nt_j^{(N+2)/2}g_{t_j}\|_{L^{\infty}(B_{r_3/2}(y)\cap A)}\le 2\sqrt{N}\epsilon.
\end{equation}
Finding a countable subset $\{y_i\}_i$ of $A$ with $A \subset \bigcup_iB_{r_3(y_i)/4}(y_i)$, (\ref{eq:unifo}) shows 
\begin{equation}
\|g_Y-c_Nt_j^{(N+2)/2}g_{t_j}\|_{L^{\infty}(A)}\le 2\sqrt{N}\epsilon
\end{equation}
which completes the proof because $\epsilon$ is arbitrary.
\end{proof}
Let us prove the converse implication of the above result, under assuming a kind of uniformity of $A$.
\begin{theorem}\label{prop:bilipch2}
Let $A$ be a Borel subset of $Y$ and let $t_i \to 0^+$ be a convergent sequence. Assume that 
\begin{equation}
\|g_X-c_Nt_i^{(N+2)/2}g_{t_i}\|_{L^{\infty}(A)} \to 0,\quad (i \to \infty)
\end{equation}
holds and that for all $\epsilon \in (0, 1)$ and $y \in A$ there exists $r_3 \in (0, 1)$ such that 
\begin{equation}\label{absrabuirairabiwi}
\left|\frac{\mathcal{H}^N (B_r(z)\cap A)}{\mathcal{H}^N (B_r(z))}-1\right|<\epsilon,\quad \forall z \in B_{r_3}(y) \cap A,\,\,\forall r \in (0, r_3].
\end{equation}
holds. Then
 for any $\epsilon \in (0, 1)$ there exists $i_0 \in \mathbb{N}$ such that $\tilde{\Phi}_{t_j}|_A$ is a locally $(1 \pm \epsilon)$-bi-Lipschitz embedding for any $j \in \mathbb{N}_{\ge i_0}$. 
\end{theorem}
\begin{proof}
Let $\epsilon \in (0, 1)$ be a sufficiently small and take $j \in \mathbb{N}$ with $\|g_Y-c_Nt_j^{(N+2)/2}g_{t_j}\|_{L^{\infty}(A)}\le \epsilon$. Moreover fix $y \in A$ and take $r_3$ as in the assumption for $\epsilon, y$. Then by (\ref{absrabuirairabiwi}), for all $z\in B_{r_3/2}(y) \cap A$ and $r \in (0, r_3]$ we have
\begin{align}
&\frac{1}{\mathcal{H}^N(B_r(z))}\int_{B_r(z)}\left| g_Y-c_Nt_j^{(N+2)/2}g_{t_j}\right| \di \mathcal{H}^N \nonumber \\
&=\frac{1}{\mathcal{H}^N(B_r(z))}\int_{B_r(z)\cap A}\left| g_Y-c_Nt_j^{(N+2)/2}g_{t_j}\right| \di \mathcal{H}^N+\frac{1}{\mathcal{H}^N(B_r(z))}\int_{B_r(z) \setminus A}\left| g_Y-c_Nt_j^{(N+2)/2}g_{t_j}\right| \di \mathcal{H}^N \nonumber \\
&\le \epsilon \cdot \frac{\mathcal{H}^N(B_r(z) \cap A)}{\mathcal{H}^N(B_r(z))} +C(K, N, v)\frac{\mathcal{H}^N(B_r(z) \setminus A)}{\mathcal{H}^N(B_r(z))} \le C(K, N, v)\epsilon
\end{align}
which proves that $B_{r_3/2}(y) \cap A \subset \mathcal{R}_Y(C(K, N, v)\epsilon, t_j, r_3/2)$. Thus Theorem \ref{prop:bilip} completes the proof because $\epsilon$ is arbitrary.
\end{proof}

The following is a direct consequence of Theorems \ref{prop:bilipch} and \ref{prop:bilipch2}.
\begin{corollary}\label{cor:equiv}
Let $U$ be an open subset of $Y$. Then the following two conditions are equivalent;
\begin{enumerate}
\item For any $\epsilon \in (0, 1)$ there exists $t_0 \in (0, 1)$ such that $\tilde{\Phi}_t|_U$ is a locally $(1 \pm \epsilon)$-bi-Lipschitz embedding for any $t \in (0, t_0)$.
\item We have
\begin{equation}
\|g_Y-c_Nt^{(N+2)/2}g_t\|_{L^{\infty}(U)}\to 0, \quad (t \to 0^+).
\end{equation}
\end{enumerate}
\end{corollary}
\begin{definition}[Weakly smooth $\RCD$]\label{weaklysmoothdef}
A non-collapsed compact $\RCD(K, N)$ space $(Y, \dist_Y, \mathcal{H}^N)$ is said to be \textit{weakly smooth} if as $t \to 0^+$
\begin{equation}
\|g_Y-c_Nt^{(N+2)/2}g_t\|_{L^{\infty}(Y)} \to 0.
\end{equation} 
\end{definition}
It is worth pointing out that if $(Y, \dist_Y, \mathcal{H}^N)$ is weakly smooth, then thanks to Proposition \ref{prop:regular}, $Y=\mathcal{R}_N$. In particular by the intrinsic Reifenberg theorem proved in \cite[Th.A.1.2 and Th.A.1.3]{CheegerColding1}, $Y$ is bi-H\"older homeomorphic to an $N$-dimensional closed Riemannian manifold, where the H\"older exponent can be chosen as an arbitrary $\alpha \in (0, 1)$.
Let us now restate Theorem \ref{themequivalence}.
\begin{theorem}[Characterization of weakly smooth $\RCD$ space]\label{prop:portegies}
The following four conditions are equivalent.
\begin{enumerate}
\item The space $(Y, \dist_Y, \mathcal{H}^N)$ is weakly smooth.
\item We have
\begin{equation}
\|g_Y-\tilde{c}_Nt\mathcal{H}^N(B_{\sqrt{t}}( \cdot))g_t\|_{L^{\infty}(Y)} \to 0,\quad (t \to 0^+).
\end{equation}
\item For any sufficiently small $t \in (0, 1)$, $\Phi_t$ is a bi-Lipschitz embedding. More strongly, for any $\epsilon \in (0, 1)$ there exists $t_0 \in (0, 1)$ such that $\tilde{\Phi}_t$ is a locally $(1 \pm \epsilon)$-bi-Lipschitz embedding for any $t \in (0, t_0]$.
\item For any sufficiently small $t \in (0, 1)$, $\Phi_t^l$ is a bi-Lipschitz embedding for any sufficiently large $l$. More strongly, for any $\epsilon \in (0, 1)$ there exists $t_0 \in (0, 1)$ such that for any $t \in (0, t_0]$ there exists $l_0 \in \mathbb{N}$ such that $\tilde{\Phi}_t^l$ is a locally $(1 \pm \epsilon)$-bi-Lipschitz embedding for any $l \in \mathbb{N}_{\ge l_0}$.
\end{enumerate}
\end{theorem}
\begin{proof}
We first prove the implication from (1) to (4). Assume that (1) holds. Take a sufficiently small $\epsilon \in (0, 1)$ and find $t_0$ with
\begin{equation}
\|g_Y-c_Nt^{(N+2)/2}g_t\|_{L^{\infty}}<\frac{\epsilon}{4},\quad \forall t \in (0, t_0].
\end{equation}
Fix $t \in (0, t_0]$ and find $l_0$ with
\begin{equation}
c_Nt^{(N+2)/2}\sum_{i=l_0+1}^{\infty}e^{-2\lambda_i^Yt}\|\dist \phi_i^Y\|_{L^{\infty}}^2<\frac{\epsilon}{4}.
\end{equation}
Then Theorem \ref{prop:finitedimen} allows us to prove that for any $y \in Y$, there exists $r_4 \in (0, 1)$ such that $\tilde{\Phi}_t^l|_{B_{r_4}(y)}$ is a $(1 \pm \epsilon)$-bi-Lipschitz embedding for any $l \in \mathbb{N}_{\ge l_0}$. Thus in order to get (4), it is enough to prove that $\Phi_t^l$ is injective for any sufficiently large $l$. If not, there exist a sequence of $l_i \to \infty$ and sequences of  $y_i, z_i \in Y$ such that $y_i \neq z_i$ and 
\begin{equation}\label{eq:ineq}
\tilde{\Phi}_t^{l_i}(y_i)=\tilde{\Phi}_t^{l_i}(z_i)
\end{equation}
are satisfied for any $i$.
Since $Y$ is compact, after passing to a subsequence, we have $y_i \to y$ and $z_i \to z$ in $Y$ for some $y, z \in Y$. Letting $i \to \infty$ in (\ref{eq:ineq}) shows that $\Phi_t(y)=\Phi_t(z)$. Thus it follows from the injectivity of $\Phi_t$ that $y=z$ holds. On the other hand applying Theorem \ref{prop:finitedimen} for $y (=z)$ shows that there exists $r_2 \in (0, 1)$ such that $\tilde{\Phi}_t^{l_i}|_{B_{r_2}(y)}$ is injective for any sufficiently large $i$. Thus $y_i=z_i$ holds for any sufficiently large $i$, which is a contradiction. Therefore $\tilde{\Phi}_t^l$ is injective for any sufficiently large $l$, thus we have (4). 

Next we prove the implication from (4) to (1). Assume that (4) holds. 
Fix a sufficiently small $\epsilon \in (0, 1)$ and take $t_0, t, l_0, l$ as in the assumption.
Corollary \ref{corfromlip} yields
\begin{equation}\label{dacdrf}
|\tilde{\Phi}^l_tg_{\mathbb{R}^l}-g_Y|(y)\le C(N)\epsilon, \quad \text{for $\mathcal{H}^N$-a.e. $y \in Y$.}
\end{equation} 
Letting $l \to \infty$ in the weak form of (\ref{dacdrf}) shows that (1) holds.

Since the equivalence between (1) and (3) is justified in Corollary \ref{cor:equiv} by letting $U=Y$, it is enough to check the equivalence between (1) and (2).
Assume that (1) or (2) holds. Then Proposition \ref{prop:regular} shows $Y=\mathcal{R}_N$. By an argument similar to the proof of Theorem \ref{prop:bilip}, we see that $\mathcal{H}^N(B_r(\cdot))/(\omega_Nr^N)$ converge uniformly to $1$ as $r \to 0^+$. In particular combining this with (\ref{eq:gradeert}) yields 
\begin{equation}
\|\tilde{c}_Nt\mathcal{H}^N(B_{\sqrt{t}}( \cdot))g_t-c_Nt^{(N+2)/2}g_t\|_{L^{\infty}(Y)} \to 0,\quad (t \to 0^+),
\end{equation}
which completes the proof of the desired equivalence.
\end{proof}

\section{Asymptotic behavior of $t^{(N+2)/2}\mathcal{E}_{Y, t}(f)$ as $t \to 0^+$.}\label{asympto}
Throughout the section let us fix
\begin{itemize}
\item a finite dimensional (not necessary compact) $\RCD$ space $(X, \dist_X, \meas_X)$ whose essential dimension is equal to $n$,
\item a non-collapsed compact $\RCD(K, N)$ space $(Y, \dist_Y, \mathcal{H}^N)$ with $\mathrm{diam}(Y, \dist_Y)\le d<\infty$ and $\mathcal{H}^N(Y) \ge v>0$,
\item a bounded Borel (not necessary open) subset $A$ of $X$,
\item an open subset $U$ of $X$.
\end{itemize}
We first discuss Lipschitz maps from $A$ to $Y$ and then discuss $0$-Sobolev maps from $U$ to $Y$.
In this section the following notion will play a key role.
\begin{definition}\label{adabdaoydasb}
Let $t_i \to 0^+$ be a convergent sequence and let $\{\tau_i\}_i$ be a sequence in $(0, \infty)$.  Define
\begin{equation}
\mathcal{R}_Y(\{(t_i, \tau_i)\}_i):=\bigcap_{\epsilon \in (0, 1)}\bigcup_i \bigcap_{j \ge i}\mathcal{R}_Y(\epsilon, t_j, \tau_j).
\end{equation}
\end{definition}
Note that by definition we have $\mathcal{R}_Y(\{t_i\}_i) \subset \mathcal{R}_Y(\{(t_i, \tau_i)\}_i)$.
\subsection{Pull-back of Lipschitz map into smooth part}\label{lippull}
Fix a Lipschitz map $f:A \to Y$.
\begin{proposition}\label{prop:cauchy}
Assume that $f(A) \subset \mathcal{R}_Y(\epsilon, t, \tau) \cap \mathcal{R}_Y(\epsilon, s, \tau)$ for some $\epsilon \in (0, 1/6), \tau, t, s \in (0, \infty)$.
Then  
\begin{equation}
\left| (\tilde{\Phi}_t^Y \circ f)^*g_{L^2}- (\tilde{\Phi}_s^Y \circ f)^*g_{L^2}\right| \le C(n)\epsilon \left|(\tilde{\Phi}_t^Y \circ f)^*g_{L^2}\right|,\quad \text{for $\meas_X$-a.e. $x \in A$.}
\end{equation}
\end{proposition}
\begin{proof}
Lemma \ref{lem:1} and Theorem \ref{prop:finitedimen} show that for any $x \in A$ there exists $r_1=r_1(f(x)) \in (0, 1)$ such that 
for any sufficiently large $l$ we have
\begin{equation}\label{eq:finite}
\left| (\tilde{\Phi}_t^l \circ f)^*g_{\mathbb{R}^l}- (\tilde{\Phi}_s^l \circ f)^*g_{\mathbb{R}^l}\right| \le C(n)\epsilon \left|(\tilde{\Phi}_t^l \circ f)^*g_{\mathbb{R}^l}\right|, \quad \text{for $\meas_X$-a.e. $z \in f^{-1}(B_{r_1}(f(x)))$.}
\end{equation}
Letting $l \to \infty$ in (\ref{eq:finite}) completes the proof because we can find a sequence $x_i \in A$ with $A=\bigcup_i f^{-1}(B_{r_1(f(x_i))}(f(x_i)))$.
\end{proof}
Recall Definition \ref{adabdaoydasb} for the definition of $\mathcal{R}_Y(\{(t_i, \tau_i)\}_i)$.
\begin{proposition}\label{prop:def}
If $f(A) \subset \mathcal{R}_Y(\{(t_i, \tau_i)\}_i)$ holds for some $\{(t_i, \tau_i)\}_i$,
then the sequence $(\tilde{\Phi}_{t_i}^Y \circ f)^*g_{L^2}$ is a Cauchy sequence in $L^{p}((T^*)^{\otimes 2}(A, \dist_X, \meas_X))$ for any $p \in [1, \infty)$. The $L^p$-limit does not depend on the choice of $\{(t_i, \tau_i)\}_i$ in the sense
\begin{itemize}
\item if $f(A) \subset \mathcal{R}_Y(\{(s_i, \delta_i)\}_i)$ for some $\{(s_i, \delta_i)\}_i$, then 
\begin{equation}\label{eq:equallimit}
\lim_{i \to \infty}(\tilde{\Phi}_{t_i}^Y \circ f)^*g_{L^2}=\lim_{i \to \infty}(\tilde{\Phi}_{s_i}^Y \circ f)^*g_{L^2},\quad \mathrm{in}\,\,L^{p}((T^*)^{\otimes 2}(A, \dist_X, \meas_X)).
\end{equation}
\end{itemize}
We denote by $f^*g_Y$ the limit tensor. Moreover we see that $f^*g_Y \in L^{\infty}((T^*)^{\otimes 2}(A, \dist_X, \meas_X))$ holds, that
\begin{equation}\label{poasiiairbbbsk}
\|f^*g_Y\|_{L^{\infty}(A)}\le n(\mathbf{Lip}f)^2
\end{equation}
holds and that for any $i$ and for $\meas_X$-a.e. $x \in \bigcap_{j \ge i}\mathcal{R}_Y(\epsilon, t_j, \tau_j)$, we have
\begin{equation}\label{eqeqqe}
|(\tilde{\Phi}_{t_i}^Y \circ f)^*g_{L^2}-f^*g_Y|(x)\le C(n)\epsilon \min \left\{ |f^*g_Y|(x), |(\tilde{\Phi}_{t_i}^Y \circ f)^*g_{L^2}|(x)\right\}.
\end{equation}
\end{proposition}
\begin{proof}
Fix $\epsilon \in (0, 1)$. Let 
\begin{equation}\label{eq:www33}
A_i:=\bigcap_{j \ge i}f^{-1}(\mathcal{R}_Y(\epsilon, t_j, \tau_j)),\quad B_i:=\bigcap_{j \ge i}f^{-1}(\mathcal{R}_Y(\epsilon, s_j, \delta_j)).
\end{equation}
Since $A_i \subset A_{i+1}, B_i \subset B_{i+1}$ with
\begin{equation}
\bigcup_iA_i=\bigcup_iB_i=A,
\end{equation}
for any $\delta \in (0, 1)$ we can find $i \in \mathbb{N}$ with $\meas_X(A \setminus (A_i\cap B_i))<\delta$.
Proposition \ref{prop:cauchy} with (\ref{eq:gradeert}) (recall also (\ref{eq:bound}) and (\ref{eq:finite})) shows
\begin{equation}\label{eq:2}
\|(\tilde{\Phi}_{t_j}^Y\circ f)^*g_{L^2}-(\tilde{\Phi}_{s_l}^Y\circ f)^*g_{L^2}\|_{L^{\infty}(A_i \cap B_i)} \le C(\mathbf{Lip}f, K,  N, v)\epsilon, \quad \forall j, l \in \mathbb{N}_{\ge i}.
\end{equation}
In particular for any $p \in [1, \infty)$
\begin{align}\label{eq:90}
&\int_A \left|(\tilde{\Phi}_{t_j}^Y\circ f)^*g_{L^2}-(\tilde{\Phi}_{s_l}^Y\circ f)^*g_{L^2}\right|^p\di \meas_X \nonumber \\
&=\int_{A\setminus (A_i \cap B_i)} \left|(\tilde{\Phi}_{t_j}^Y\circ f)^*g_{L^2}-(\tilde{\Phi}_{s_l}^Y\circ f)^*g_{L^2}\right|^p\di \meas_X+\int_{A_i \cap B_i} \left|(\tilde{\Phi}_{t_j}^Y\circ f)^*g_{L^2}-(\tilde{\Phi}_{s_l}^Y\circ f)^*g_{L^2}\right|^p\di \meas_X \nonumber \\
&\le C(\mathbf{Lip}f, K, N, v, p)\meas_X(A \setminus (A_i \cap B_i)) +C(\mathbf{Lip}f, K, N, v)\epsilon^p \meas_X(A_i \cap B_i) \nonumber \\
&\le C(\mathbf{Lip}f, K, N, v, p)\delta +C(\mathbf{Lip}f, K, N, v)\epsilon^p\meas_X(A),
\end{align}
which proves that the sequence $(\tilde{\Phi}_{t_i}^Y \circ f)^*g_{L^2}$ is a Cauchy sequence in $L^{p}((T^*)^{\otimes 2}(A, \dist_X, \meas_X))$ for any $p \in [1, \infty)$.
Moreover letting $j \to \infty$ in (\ref{eq:90}) and then letting $\epsilon, \delta \to 0^+$ complete the proof of (\ref{eq:equallimit}). Since for all $p \in [1, \infty)$ and $l \in \mathbb{N}$
\begin{align}\label{eq:linfty}
\|f^*g_Y\|_{L^p(A_l)}=\lim_{i\to \infty}\|(\tilde{\Phi}_{t_i}^Y \circ f)^*g_{L^2}\|_{L^p(A_l)}&\le (\meas_X (A_l))^{1/p}\cdot \lim_{i \to \infty}\|(\tilde{\Phi}_{t_i}^Y \circ f)^*g_{L^2}\|_{L^{\infty}(A_l)} \nonumber \\
&\le (\meas_X (A))^{1/p}(1+3\epsilon) \cdot n(\mathbf{Lip}f)^2,
\end{align}
letting $p \to \infty$ in (\ref{eq:linfty}) proves (\ref{poasiiairbbbsk}), where we used Theorem \ref{prop:bilip} with  (\ref{eq:bound}) in the last inequality of (\ref{eq:linfty}).

In order to prove the remaining statement (\ref{eqeqqe}), fix $i \in \mathbb{N}$. Proposition \ref{prop:cauchy} shows that for all $j, k \in \mathbb{N}_{\ge i}$
\begin{equation}\label{eq33}
\left| (\tilde{\Phi}_{t_j}^Y\circ f)^*g_{L^2}-(\tilde{\Phi}_{t_k}^Y\circ f)^*g_{L^2}\right|\le C(n)\epsilon \left| (\tilde{\Phi}_{t_j}^Y\circ f)^*g_{L^2}\right|,\quad \text{for $\meas_X$-a.e. $x \in \bigcap_{j \ge i}\mathcal{R}_Y(\epsilon, t_j, \tau_j)$.}
\end{equation}
In particular letting $j \to \infty$ and $k \to \infty$ in (\ref{eq33}) with (\ref{eq:equallimit}), respectively completes the proof of (\ref{eqeqqe}).
\end{proof}
Next let us discuss on the behavior of pull-backs under compositions of maps.
\begin{proposition}\label{prop:iso}
Let $(Z, \dist_Z, \mathcal{H}^{\tilde{N}})$ be a non-collapsed compact $\RCD(\tilde{K}, \tilde{N})$ space, let $\epsilon \in (0, 1)$ and let $h:f(A) \to Z$ be a Lipschitz map. Assume that the following hold.
\begin{enumerate}
\item $f(A) \subset \mathcal{R}_Y(\{(t_i, \tau_i)\}_i)$ and $h\circ f(A) \subset \mathcal{R}_Z(\{(s_i, \delta_i)\}_i)$ are satisfied for some $\{(t_i, \tau_i)\}_i$ and some $\{(s_i, \delta_i)\}_i$.
\item For all $y \in f(A)$, there exists $r \in (0, 1)$ such that $h|_{f(A) \cap B_r(y)}$ is a $(1\pm \epsilon)$-bi-Lipschitz embedding.
\end{enumerate}
Then
\begin{equation}\label{absurabiraiubrabi}
\|(h \circ f)^*g_Z-f^*g_X\|_{L^{\infty}(A)}\le C(n)\epsilon.
\end{equation}
In particular if (1) and the following condition (3) are satisfied:
\begin{enumerate}
\item[3.]for all $y \in f(A)$ and $\delta \in (0, 1)$, there exists $r=r(y) \in (0, 1)$ such that $h|_{f(A) \cap B_r(y)}$ is a $(1\pm \delta)$-bi-Lipschitz embedding;
\end{enumerate}
then
\begin{equation}\label{asraorjasorao}
(h \circ f)^*g_Z=f^*g_X.
\end{equation}
\end{proposition}
\begin{proof}
For any fixed sufficiently large $i \in \mathbb{N}$ we have
\begin{equation}\label{data456}
\left| (h \circ f)^*g_{Z}-(\tilde{\Phi}_{t_i}^l\circ f)^*g_{\mathbb{R}^l}\right|(y) \le C(n) \epsilon, \quad \text{for $\meas_X$-a.e. $x \in  \bigcap_{j \ge i}f^{-1}(\mathcal{R}_Y(\epsilon, t_j, \tau_j) \cap B_{r(y)}(f(y)))$}
\end{equation}
for any sufficiently large $l \in \mathbb{N}$ because of Corollary \ref{corfromlip} and Theorem \ref{prop:finitedimen}. Then letting $l \to \infty$ in (\ref{data456}) completes the proof of (\ref{absurabiraiubrabi}). The equality (\ref{asraorjasorao}) is a direct consequence of (\ref{absurabiraiubrabi}).
\end{proof}
Recall that a map $\Phi=(\phi_i)_i$ from an open subset $U$ of $Y$ to $\mathbb{R}^k$ is said to be \textit{regular} if each $\phi_i$ is in $D(\Delta_Y, U)$ with $\Delta_Y \phi_i \in L^{\infty}(U, \mathcal{H}^N)$. It follows from regularity results proved in \cite[Th.3.1]{AMS} and in \cite[Th.3.1]{Jiang} that any regular map is locally Lipschitz. Note that this observation works for general finite dimensional (not necessary non-collapsed) $\RCD$ spaces (see also the beginning of subsection 7.1 of \cite{honda20}).
\begin{corollary}\label{prop:iso2}
Assume that $f(A) \subset \mathcal{R}_Y(\{(t_i, \tau_i)\}_i)$ holds for for some $\{(t_i, \tau_i)\}_i$.
Then for any regular map $h$ from an open neighborhood $U$ of $f(A)$ to $\mathbb{R}^m$ with 
\begin{equation}
\|h^*g_{\mathbb{R}^m}-g_Y\|_{L^{\infty}(U)}\le \epsilon
\end{equation}
for some $\epsilon \in (0, 1)$, 
we have
\begin{equation}
\|(h \circ f)^*g_{\mathbb{R}^m}-f^*g_Y\|_{L^{\infty}(A)}\le \Phi(\epsilon;K, N, m, n).
\end{equation}
In particular if $h^*g_{\mathbb{R}^m}=g_Y$ holds, then
\begin{equation}
(h \circ f)^*g_{\mathbb{R}^m}=f^*g_Y.
\end{equation}
\end{corollary}
\begin{proof}
It is enough to check the assertion under the assumption that $f(A)$ is bounded with $\overline{f(A)} \subset U$.
Choose $R \in [1, \infty)$ with $h \circ f(A) \subset B_{R/2}(0_m)$.
Thanks to \cite[Th.1.1]{honda20}, for all $y \in U$ there exists $r:=r(y) \in (0, 1)$ such that $h|_{B_r(y)}$ is a $(1 \pm \Phi(\epsilon; K, N, m))$-bi-Lipschitz embedding. Then applying Proposition \ref{prop:iso} with $(Z, \dist_Z, \mathcal{H}^N)=(\overline{B}_R(0_m), \dist_{\mathbb{R}^m}, \mathcal{H}^N)$ completes the proof.
\end{proof}
Let us provide a bi-Lipschitz embeddability of a Sobolev map $f$, under assuming some regularity of $f$.
\begin{corollary}\label{asataras}
Let $\epsilon \in (0, 1)$, let $\tilde{f}:U \to Y$ be a Sobolev map with $\tilde{f}(U \setminus D) \subset \mathcal{R}_Y(\{(t_i, \tau_i)\}_i)$ for some $\{(t_i, \tau_i)\}_i$ and some $\meas_X$-negligible subset $D$ of $U$, and let $h:\tilde{f}(U) \to \mathbb{R}^m$ be a map.  Assume that $\|\tilde{f}^*g_Y-g_X\|_{L^{\infty}(U)}\le \epsilon$ holds, that $h \circ \tilde{f}$ is regular and that for any $y \in f(U)$ there exists $r \in (0, 1)$ such that $h|_{B_r(y)}$ is a $(1 \pm \epsilon)$-bi-Lipschitz embedding. Then for any $x \in U$ there exists $\overline{r} \in (0, 1)$ such that $\tilde{f}|_{B_{\overline{r}}(x)}$ is a $(1 \pm \Phi(\epsilon; \tilde{K}, \tilde{N}, m))$-bi-Lipschitz embedding, whenever $(X, \dist_X, \meas_X)$ is an $\RCD(\tilde{K}, \tilde{N})$ space.
\end{corollary}
\begin{proof}
Since $\|\tilde{f}^*g_Y\|_{L^{\infty}(U)} \le \sqrt{n}+\epsilon$ holds, applying Proposition \ref{propsobtolip} and Theorem \ref{propcompatr} we will give later independently yields that $\tilde{f}$ has a locally $(\sqrt{n}+\epsilon)$-Lipschitz representative. Thus it is enough to check the assertion under assuming that both $\tilde{f}(U)$ and $h \circ \tilde{f}(U)$ are bounded.
Let $x \in U$. Since the proof of Proposition \ref{prop:iso} shows $\|(h \circ \tilde{f})^*g_{\mathbb{R}^m}-\tilde{f}^*g_Y\|_{L^{\infty}(U)}\le C(n)\epsilon$, we have
\begin{equation}\label{aoiraibwirab}
\|(h \circ \tilde{f})^*g_{\mathbb{R}^m}-g_X\|_{L^{\infty}(U)}\le C(n)\epsilon.
\end{equation}
Applying \cite[Th.3.4]{honda20} for $h \circ \tilde{f}$ with (\ref{aoiraibwirab}) yields that for any $x \in U$ there exists $r \in (0, 1)$ such that $(h \circ \tilde{f})|_{B_r(x)}$ is a $(1 \pm \Phi(\epsilon; \tilde{K}, \tilde{N}, m))$-bi-Lipschitz embedding. Thus we conclude. 
\end{proof}
\begin{remark}\label{asahoarbasj}
In general, the isometric equation, $f^*g_Y=g_X$, does not imply the local bi-Lipschitz embeddability of $f$ without a regularity assumption on $f$.
In fact let us consider a compact non-collapsed $\RCD(0, 1)$ space $([-1, 1], \dist_{\mathbb{R}}, \mathcal{H}^1)$ and a map $f:[-1, 1] \to [-1, 1]$ defined by $f(x):=|x|$. Then although it is easy to see $f^*g_{[-1, 1]}=g_{[-1, 1]}$, $f$ is not a bi-Lipschitz embedding around the origin. Thus the regularity assumption in Corollary \ref{asataras} is essential.
\end{remark}

We are now in a position to give a geometric meaning of the pull-back.
\begin{proposition}\label{proppullb}
Assume that $f(A) \subset \mathcal{R}_Y(\{(t_i, \tau_i)\}_i)$ for some $\{(t_i, \tau_i)\}_i$. Then
\begin{equation}
\mathrm{Lip}f(x)=\left(\left| f^*g_Y\right|_B(x)\right)^{1/2},\quad \text{for $\meas_X$-a.e. $x \in A$.}
\end{equation}
\end{proposition}
\begin{proof}
Fix a sufficiently small $\epsilon \in (0, 1)$ and $i \in \mathbb{N}$.  Find $l \in \mathbb{N}$ with
\begin{equation}
c_Nt_i^{(N+2)/2}\sum_{j=l+1}^{\infty}e^{-2\lambda_j^Yt_i}\|\dist \phi_j^Y\|_{L^{\infty}}^2<\frac{\epsilon}{4}.
\end{equation}
Then  Theorem \ref{prop:finitedimen} with Proposition \ref{prop:den} show
\begin{equation}
\left| \mathrm{Lip}f(x)-\mathrm{Lip} (\tilde{\Phi}_{t_i}^l \circ f)(x)\right|\le \mathbf{Lip}f \cdot\epsilon, \quad \text{for $\meas_X$-a.e. $x \in f^{-1}(\mathcal{R}_Y(\epsilon, t_i, \tau_i))$},
\end{equation}
applying Proposition \ref{prop:geometric} for $\tilde{\Phi}_{t_i}^l \circ f$ on $f^{-1}(\mathcal{R}_Y(\epsilon, t_i, \tau_i))$ yields
\begin{equation}\label{asnabryasbary}
\left| \mathrm{Lip}f(x)-\left(\left| (\tilde{\Phi}_{t_i}^l \circ f)^*g_{\mathbb{R}^l}\right|_B(x)\right)^{1/2}\right|\le C(n)\cdot \mathbf{Lip}f \cdot\epsilon, \quad \text{for $\meas_X$-a.e. $x \in f^{-1}(\mathcal{R}_Y(\epsilon, t_i, \tau_i))$.}
\end{equation}
Letting $l \to \infty$ in (\ref{asnabryasbary}) gives
\begin{equation}\label{asnabryasbary2}
\left| \mathrm{Lip}f(x)-\left(\left| (\tilde{\Phi}_{t_i}^Y \circ f)^*g_{\ell^2}\right|_B(x)\right)^{1/2}\right|\le C(n) \cdot \mathbf{Lip}f \cdot\epsilon, \quad \text{for $\meas_X$-a.e. $x \in f^{-1}(\mathcal{R}_Y(\epsilon, t_i, \tau_i))$.}
\end{equation}
In particular for $\meas_X$-a.e. $x \in \bigcap_{j \ge i}f^{-1}(\mathcal{R}_Y(\epsilon, t_j, \tau_j))$, we have
\begin{equation}\label{asarubauiweahr}
\left| \mathrm{Lip}f(x)-\left(\left| (\tilde{\Phi}_{t_j}^Y \circ f)^*g_{\ell^2}\right|_B(x)\right)^{1/2}\right|\le C(n)\cdot \mathbf{Lip}f\cdot\epsilon,
\end{equation}
for any $j \ge i$. Thus letting $j \to \infty$ in (\ref{asarubauiweahr}) with Proposition \ref{prop:def} implies
\begin{equation}\label{asarubauiweahr2}
\left| \mathrm{Lip}f(x)-\left(\left| f^*g_{Y}\right|_B(x)\right)^{1/2}\right|\le C(n)\cdot \mathbf{Lip}f\cdot\epsilon, \quad \text{for $\meas_X$-a.e. $x \in \bigcap_{j \ge i}f^{-1}(\mathcal{R}_Y(\epsilon, t_j, \tau_j))$}
\end{equation}
which completes the proof because $\epsilon$ and $i$ are arbitrary.
\end{proof}
Let us introduce the following notion in order to generalize the above observation to more general maps. 
\begin{definition}[Lipschitz-Lusin map]\label{def:lipschitzlusin}
Let $B$ be a Borel subset of $X$.
We say that a map $F:B \to Y$ is a  \textit{Lipschitz-Lusin map} if there exists a sequence of Borel subsets $D_i$ of $B$ such that $\meas_X(B \setminus \bigcup_iD_i)=0$ and that $F|_{D_i}$ is Lipschitz for any $i$. 
\end{definition} 
Applying Proposition \ref{prop:def} to $f=F|_{D_i}, B=D_i$ shows that the following is well-defined.
\begin{definition}[Pull-back of Lipschitz-Lusin map]
Let $B$ be a Borel subset of $X$ and let $F:B \to Y$ be a Lipschitz-Lusin map. Assume that $F(B) \subset \mathcal{R}_Y(\{(t_i, \tau_i)\}_i)$ holds for some $\{(t_i, \tau_i)\}_i$. Then there exists a unique $T \in L^0((T^*)^{\otimes 2}(B, \dist_X, \meas_X))$, denoted by $F^*g_Y$, such that 
\begin{equation}
(F|_D)^*g_Y=T
\end{equation}
holds on $D$ whenever the restriction of $F$ to a Borel subset $D$ of $B$ is Lipschitz. Then define the \textit{energy density}, denoted by $e_Y(F) \in L^0(B, \meas_X)$, by 
\begin{equation}
e_Y(F):=\langle F^*g_Y, g_X\rangle.
\end{equation} 
Moreover define the \textit{energy}, denoted by $\mathcal{E}_{B, Y}(F)$, by
\begin{equation}
\mathcal{E}_{B, Y}(F):=\frac{1}{2}\int_Be_Y(F)\di \meas_X \in [0, \infty].
\end{equation}
Finally we say that $F$ is \textit{isometric} if $F^*g_Y=g_X$.
\end{definition}
\begin{proposition}\label{prop:lipschitzlusin}
Let $F:U \to Y$ be a weakly smooth map. For all $\epsilon \in (0, 1/6)$ and $t, \tau \in (0, \infty)$, the restriction of $F$ to $F^{-1}(\mathcal{R}_Y(\epsilon, t, \tau)))$ is a Lipschitz-Lusin map. In particular if $F(U \setminus D)\subset \mathcal{R}_Y(\{(t_i, \tau_i)\}_i)$ holds for some $\{(t_i, \tau_i)\}_i$ and some $\meas_X$-negligible subset $D$ of $U$, then $F$ is Lipschitz-Lusin.
\end{proposition}
\begin{proof}
Theorem \ref{prop:finitedimen} yields that there exist $l \in \mathbb{N}$ and sequences of $y_i \in F(U)\cap \mathcal{R}_Y(\epsilon, t, \tau)$ and of $r_2(y_i) \in (0, 1)$ satisfying that $\tilde{\Phi}_t^l|_{B_{r_2(y_i)}(y_i)}$ is a bi-Lipschitz embedding and that $F(U) \cap \mathcal{R}_Y(\epsilon, t, \tau)\subset \bigcup_iB_{r_2(y_i)}(y_i)$ holds.
Then it follows from the weak smoothness of $F$ with the Poincar\'e inequality that there exists a sequence of Borel subsets $\tilde{D}_i$ of $X$ such that $\meas_X(U \setminus \bigcup_i\tilde{D}_i)=0$ and that $(\phi_j^Y\circ F)|_{\tilde{D}_i}$ is Lipschitz for all $j \in \mathbb{N}_{\le l}$ and $i \in \mathbb{N}$. In particular $\tilde{\Phi}_t^l \circ F$ is Lipschitz on $\tilde{D}_i$. Then the family $\{\tilde{D}_i \cap F^{-1}(B_{r_2(y_j)}(y_j))\}_{i, j}$ proves the assertion.
\end{proof}
\begin{proposition}\label{prop:l17773}
Let  $F:U \to Y$ be a weakly smooth map. 
Assume that $F(U \setminus D) \subset \mathcal{R}_Y(\{(t_i, \tau_i)\}_i)$ holds for some $\{(t_i, \tau_i)\}_i$ and some $\meas_X$-negligible set $D$ and that 
\begin{equation}
\liminf_{i \to \infty}t_i^{(N+2)/2}\mathcal{E}_{U, Y, t_i}(F)<\infty
\end{equation}
holds.
Then $F^*g_Y \in L^1((T^*)^{\otimes 2}(U, \dist_X, \meas_X))$.
\end{proposition}
\begin{proof}
Propositions \ref{prop:def} and \ref{prop:lipschitzlusin} shows that after passing to a subsequence, we have 
\begin{equation}\label{ahhhhuuuuanpaonranopnaro}
|(\tilde{\Phi}_{t_i}^Y\circ F)^*g_{L^2}-F^*g_Y|(x) \to 0, \quad \text{for $\meas_X$-a.e. $x \in  U$.}
\end{equation}
Thus it follows from Fatou's lemma that $|F^*g_Y|$ is  $L^1$ on $U$ because of (\ref{eq:densitybound}).
\end{proof}
We will discuss the behavior of $t_i^{(N+2)/2}\mathcal{E}_{U, Y, t_i}(F)$ as $i \to \infty$ later.

\subsection{Rademacher type result via blow-up}
Let us first recall the definition of \textit{harmonic points} of a Sobolev function.
\begin{definition}[Harmonic point of a function]\label{def:buhfunction}
Let  $x \in X$, $R \in (0, \infty]$, $z \in B_R(x)$ and let $f \in H^{1, 2}(B_R(x), \dist_X, \meas_X)$. We say that \textit{$z$ is a harmonic point of $f$} if
$z\in\Leb_2(|\nabla f|)$ and for any $(T_zX, \dist_{T_zX}, \meas_{T_zX}, 0_z) \in \mathrm{Tan}(X, \dist_X, \meas_X, z)$ which is the measured Gromov-Hausforff limit space of 
$(X,t_{i}^{-1} \dist_X,\meas_X(B_{t_{i}}(z))^{-1} \meas_X, z)$ for some $t_i \to 0^+$,
there exist a subsequence $(t_{i(j)})_j$ of $(t_i)_i$ and $\hat{f} \in \Lip(T_zX, \dist_{T_zX}) \cap \mathrm{Harm}(T_zX, \dist_{T_zX}, \meas_{T_zX})$ such that the rescaled functions $f_{z, t_{i(j)}}$ $H^{1, 2}_{\mathrm{loc}}$-strongly converge to $\hat{f}$ as $j\to\infty$, where $f_{z,t}$ is defined by
$$f_{z, t}:=\frac{1}{t}\left( f-\frac{1}{\meas_X(B_t(z))}\int_{B_{t}(z)}f\di\meas_X \right)$$
on $(X,t^{-1} \dist_X,\meas_X(B_{t}(z))^{-1} \meas_X)$.
We denote by $H(f)$ the set of all harmonic points of $f$.
\end{definition}
Next we introduce a similar notion for a Lipschitz function defined on a Borel (not necessary open) subset $A$ of $X$. Compare with \cite[Def.5.3]{AHPT}.
\begin{definition}[Harmonic point for Lipschitz function defined on Borel subset]\label{def:har}
Let $\phi$ be a Lipschitz function on $A$ and let $x \in \mathrm{Leb}(A)$. Then $x$ is said to be a \textit{harmonic point of $f$} if there exists a Lipschitz function $\Phi$ on $X$ such that $\Phi|_A\equiv \phi$ and that $x$ is a harmonic point of $\Phi$. It is easy to see that this definition does not depend on the choice of $\Phi$. Thus we denote by $H(\phi )$ the set of all harmonic points of $\phi$.
\end{definition}
Applying \cite[Th.5.4]{AHPT} for $\Phi$ as in the above definition we have the following.
\begin{proposition}
For any Lipschitz function $\phi:A \to \mathbb{R}$, we have $\meas_X(A \setminus H(\phi ))=0$.
\end{proposition}
The following result gives a  non-linear analogue of Cheeger's Rademacher type theorem \cite[Th.3.7]{Cheeger}.
\begin{theorem}[Rademacher type theorem]\label{prop:rademacher}
Let $f:A \to Y$ be a Lipschitz map. Assume that $f(A) \subset \mathcal{R}_Y(\{(t_i, \tau_i)\}_i)$ holds for some $\{(t_i, \tau_i)\}_i$. Then for $\meas_X$-a.e. $x \in A$ we have the following: for any convergent sequence $r_i \to 0^+$, after passing to a subsequence, 
\begin{enumerate}
\item we have
\begin{equation}\label{eq:convtangent}
\left( X, r_i^{-1}\dist_X, (\meas_X(B_{r_i}(x)))^{-1}\meas_X, x\right) \stackrel{\mathrm{pmGH}}{\to} \left( \mathbb{R}^n, \dist_{\mathbb{R}^n}, \omega_n^{-1}\mathcal{H}^n, 0_n\right),
\end{equation}
\item we have
\begin{equation}\label{eq:convtangent2}
\left(Y, r_i^{-1}\dist_Y, \mathcal{H}^N_{r_i^{-1}\dist_Y}, f(x) \right) \stackrel{\mathrm{pmGH}}{\to} \left( \mathbb{R}^N, \dist_{\mathbb{R}^N}, \mathcal{H}^N, 0_N\right),
\end{equation}
\item the maps
\begin{equation}\label{eq:maps}
f:(A, r_i^{-1}\dist_X) \to (Y, r_i^{-1}\dist_Y)
\end{equation}
uniformly converge to a linear map $f^0:\mathbb{R}^n \to \mathbb{R}^N$ on any bounded subset of $\mathbb{R}^n$ with respect to (\ref{eq:convtangent}) and (\ref{eq:convtangent2}), 
\item $f^*g_Y$ $L^2_{\mathrm{loc}}$-strongly converge to $(f^0)^*g_{\mathbb{R}^N}$ with respect to (\ref{eq:convtangent}).
\end{enumerate}
\end{theorem}
\begin{proof}
Let us fix a sufficiently small $\epsilon \in (0, 1)$, $x \in \mathrm{Leb}(A) \cap \mathcal{R}_n$ and $j, l \in \mathbb{N}$ satisfying the following: 
\begin{itemize}
\item $x$ is a harmonic point of $\phi_i^Y \circ f$ for any $i$.
\item We have $f(x) \in \mathcal{R}_Y(\epsilon/6,  t_j, \tau_j)$.
\item We have for any sufficiently small $r \in (0, 1)$
\begin{equation}\label{e33e}
\frac{1}{\meas_X(B_r(x))}\int_{B_r(x)}\left| (\tilde{\Phi}_{t_j}^l\circ f)^*g_{\mathbb{R}^l}-f^*g_Y\right| \di \meas_X \le C(n) (\mathbf{Lip}f)^2 \cdot \epsilon.
\end{equation}
\item We have
\begin{equation}
c_Nt_j^{(N+2)/2}\sum_{i=l+1}^{\infty}e^{-2\lambda_i^Yt_j}\|\dist \phi_i^Y\|_{L^{\infty}}^2<\frac{\epsilon}{6}.
\end{equation}
\end{itemize}
Thanks to Proposition \ref{prop:regular}, after passing to a subsequence, we see that (\ref{eq:convtangent2}) is satisfied, and that the maps (\ref{eq:maps}) uniformly converge to a Lipschitz map $f^0:\mathbb{R}^n \to \mathbb{R}^N$ on any bounded subset of $\mathbb{R}^n$. 

We first prove that $f^0$ is linear. 
By an argument similar to the proof of Proposition \ref{prop:quantapp} (see also the proof of Proposition \ref{prop:quantitativelinfty}), $r_i^{-1}(\tilde{\Phi}_{t_j}^l-\tilde{\Phi}_{t_j}^l(f(y)))$, defined on $(Y, r_i^{-1}\dist_Y)$, uniformly converge to a linear $(1 \pm \epsilon)$-bi-Lipschitz embedding map $\tilde{\Phi}:\mathbb{R}^N \to \mathbb{R}^l$ on any bounded subset of $\mathbb{R}^N$ with respect to (\ref{eq:convtangent2}).

On the other hand it follows from the definition of harmonic points that
$r_i^{-1}(\tilde{\Phi}_{t_j}^l \circ f-\tilde{\Phi}_{t_j}^l(f(y)))$ uniformly converge to a linear map from $\mathbb{R}^n$ to $\mathbb{R}^l$ on any bounded subset of $\mathbb{R}^n$ with respect to (\ref{eq:convtangent}). Since the limit map of $r_i^{-1}(\tilde{\Phi}_{t_j}^l \circ f- \tilde{\Phi}_{t_j}^l(f(y)))$, defined on $(A, r_i^{-1}\dist_X)$, with respect to (\ref{eq:convtangent}) coincides with $\tilde{\Phi}\circ f^0$, we know that $\tilde{\Phi}\circ f^0$ is linear. Thus $f^0$ is also linear because $\tilde{\Phi}$ is linear and injective.

Finally since Theorem \ref{spectral2} shows that $(\tilde{\Phi}_{t_j}^l\circ f)^*g_{\mathbb{R}^l}$ $L^2_{\mathrm{loc}}$-strongly converge to $\tilde{\Phi}^*g_{\mathbb{R}^l}$, we have (4) because of 
\begin{equation}\label{asararatawte}
|\tilde{\Phi}^*g_{\mathbb{R}^l}-(f^0)^*g_{\mathbb{R}^l}|\le C(n) (\mathbf{Lip}f)^2 \cdot \epsilon.
\end{equation}
and since  $\epsilon$ is arbitrary, where we used Lemma \ref{lem:1} to get (\ref{asararatawte})
\end{proof}
Let us give an application of Theorem \ref{prop:rademacher} to the \textit{Korevaar-Schoen energy} of a map.
We follow the terminology from \cite{GT2}.
\begin{definition}[Korevaar-Schoen energy]\label{defksenery}
Let $h:A \to Y$ be a Borel map and let $r \in (0, \infty)$.
\begin{enumerate} 
\item Define the \textit{energy density at scale $r$ of $h$ at $x \in A$}, denoted by $\mathrm{ks}_{Y, r}(h)(x)$, by
\begin{equation}
\mathrm{ks}_{Y, r}(h)(x):=\left(\frac{1}{\meas_X(B_r(x))}\int_{B_r(x) \cap A}\frac{\dist_Y(h(x), h(y))^2}{r^2}\di \meas_X(y)\right)^{1/2}.
\end{equation}
\item Define the \textit{Korevaar-Schoen energy at scale $r$}, denoted by $\mathcal{E}_{A, Y, r}^{KS}(h)$, by
\begin{equation}
\mathcal{E}_{A, Y, r}^{KS}(h):=\int_A\left(\mathrm{ks}_{Y, r}(h)\right)^2\di \meas_X.
\end{equation}
\item Define the \textit{Korevaar-Schoen energy}, denoted by $\mathcal{E}_{A, Y}^{KS}(h)$, by
\begin{equation}
\mathcal{E}_{A, Y}^{KS}(h):=\limsup_{r \to 0^+}\mathcal{E}_{A, r}^{KS}(h).
\end{equation}
\end{enumerate}
\end{definition}
Compare the following corollary with \cite[Th.4.14]{GT2}.
\begin{corollary}[Compatibility with Korevaar-Schoen energy for Lipschitz map]\label{cor:ksks}
Let $f:A \to Y$ be a Lipschitz map. Assume that $f(A) \subset \mathcal{R}_Y(\{(t_i, \tau_i)\}_i)$ holds for some $\{(t_i, \tau_i)\}_i$. Then we have as $r \to 0^+$
\begin{equation}\label{ksconv}
(n+2)\mathrm{ks}_{Y, r}(f)^2 \to e_Y(f)
\end{equation}
in $L^1(A, \meas_X)$. In particular we have
\begin{equation}
\mathcal{E}_{A, Y}(f)=\frac{n+2}{2}\mathcal{E}_{A, Y}^{KS}(f).
\end{equation} 
\end{corollary}
\begin{proof}
Let us take $x \in A$ satisfying the conclusions of Theorem \ref{prop:rademacher}, where we will use the same notation as in Theorem \ref{prop:rademacher}.
Then the uniform convergence of $f$ to $f^0$ implies
\begin{equation}
\lim_{r \to 0^+}\mathrm{ks}_{Y, r}(f)(x)^2=\frac{1}{\omega_n}\int_{B_1(0_n)}|f^0(z)|^2\di \mathcal{H}^n=\frac{1}{n+2}\cdot \mathrm{tr}((f^0)^*g_{\mathbb{R}^N}).
\end{equation}
On the other hand the $L^2_{\mathrm{loc}}$-strong convergence of $f^*g_Y$ to $(f^0)^*g_{\mathbb{R}^N}$ with Proposition \ref{weakriem} yields
\begin{align}
\frac{1}{\meas_X(B_{r_i}(x))}\int_{B_{r_i}(x)}e_Y(f)\di \meas_X&=\frac{1}{\meas_X(B_{r_i}(x))}\int_{B_{r_i}(x)}\langle f^*g_Y, g_X\rangle \di \meas_X \nonumber \\
&\to \frac{1}{\omega_n}\int_{B_1(0_n)}\langle (f^0)^*g_{\mathbb{R}^N}, g_{\mathbb{R}^n}\rangle \di \mathcal{H}^n=\mathrm{tr}((f^0)^*g_{\mathbb{R}^N}),
\end{align}
which completes the proof of (\ref{ksconv}) because of the dominated convergence theorem.
\end{proof}
\subsection{Nonlinear analogue of Cheeger's differentiability theorem}
We are now in position to give a nonlinear analogue of Cheeger's differentiability theorem \cite[Th.6.1]{Cheeger}.
\begin{theorem}[Compatibility, II]\label{propcompatr}
Let $f:U \to Y$ be a weakly smooth map. Assume that $f(U \setminus D) \subset \mathcal{R}_Y(\{(t_i, \tau_i)\}_i)$ for some $\{(t_i, \tau_i)\}_i$ and some $\meas_X$-negligible set $D$.
Then the following two conditions are equivalent.
\begin{enumerate}
\item We have
\begin{equation}
\liminf_{i \to \infty}t_i^{(N+2)/2}\mathcal{E}_{U, Y, t_i}(f)<\infty.
\end{equation}
\item The map $f$ is a Sobolev map.
\end{enumerate}
In particular $f$ is a $0$-Sobolev map if and only if $f$ is a Sobolev map.
Moreover if these conditions are satisfied, then
\begin{equation}\label{ddccrrtt}
G_f(x) =\left(\left| f^*g_Y\right|_B(x)\right)^{1/2}\quad \text{for $\meas_X$-a.e. $x \in U$}.
\end{equation}
In particular we have
\begin{equation}\label{ddccrrttww}
G_f(x) =\mathrm{Lip}(f|_{\tilde{D}})(x)\quad \text{for $\meas_X$-a.e. $x \in \tilde{D}$}
\end{equation}
whenever the restriction of $f$ to a Borel subset $\tilde{D}$ of $U$ is Lipschitz. 
\end{theorem}
\begin{proof}
Since Proposition \ref{propcompat} proves the implication from (2) to (1), let us prove the converse implication. Assume that (1) holds.
Proposition \ref{prop:lipschitzlusin} allows us to find a sequence of Borel subsets $\{D_j\}_j$ of $U$ such that $\meas_X(U\setminus \bigcup_jD_j)=0$ holds and that each $f|_{D_j}$ is Lipschitz. Fix a sufficiently small $\epsilon \in (0, 1)$.
Recalling Theorem \ref{prop:finitedimen}, fix
\begin{itemize}
\item an integer $i \in \mathbb{N}$,
\item a sequence of points $\{y_k\}_{k \in \mathbb{N}} \subset \mathcal{R}_Y(\epsilon/6, t_i, \tau_i)$ and a sequence of positive numbers $\{r_k\}_{k \in \mathbb{N}} \subset (0, 1)$ satisfying that, 
\begin{itemize}
\item $\mathcal{R}_Y(\epsilon/6, t_i, \tau_i) \subset \bigcup_kB_{r_k}(y_k)$,
\item $\tilde{\Phi}_{t_i}^l:B_{r_k}(y_k) \cap \mathcal{R}_Y(\epsilon/6, t_i, \tau_i) \to \mathbb{R}^l$ is a $(1\pm \epsilon )$-bi-Lipschitz embedding for any $l$ satisfying
\begin{equation}
c_Nt_i^{(N+2)/2}\sum_{j=l+1}^{\infty}e^{-2\lambda_j^Yt_i}\|\dist \phi_j^Y\|_{L^{\infty}}^2<\frac{\epsilon}{6}.
\end{equation}
\end{itemize}
\end{itemize}
Fix $\phi \in \mathrm{Lip}(Y, \dist_Y)$.
Then for all $j, k \in \mathbb{N}$, thanks to Propositions \ref{prop:den} and \ref{prop:geometric}, we have
\begin{align}\label{jjnhytu}
&\mathrm{Lip}\left(\phi \circ f\right)|_{D_j \cap (U \setminus D) \cap f^{-1}(B_{r_k}(y_k) \cap \mathcal{R}_Y(\epsilon/2, t_i, \tau_i)) }(x)\nonumber \\
&=\mathrm{Lip}\left(\phi \circ (\tilde{\Phi}_{t_i}^{l})^{-1} \circ \tilde{\Phi}_{t_i}^{l}\circ f\right)|_{D_j \cap (U \setminus D) \cap f^{-1}(B_{r_k}(y_k) \cap \mathcal{R}_Y(\epsilon/2, t_i, \tau_i)) }(x) \nonumber \\
&\le (1-\epsilon)^{-1}\mathbf{Lip} \phi \cdot \mathrm{Lip}(\tilde{\Phi}_{t_i}^{l} \circ f)(x) \nonumber \\
&\le (1-\epsilon)^{-1}\mathbf{Lip} \phi \cdot \left(\left| (\tilde{\Phi}_{t_i}^{l} \circ f)^*g_{\mathbb{R}^l}\right|_B(x)\right)^{1/2} \nonumber \\
&\le (1-\epsilon)^{-1}\mathbf{Lip} \phi \cdot \left(\left| (\tilde{\Phi}_{t_i}^Y \circ f)^*g_{\ell^2}\right|_B(x)\right)^{1/2}
\end{align}
for $\meas_X$-a.e. $x \in D_j \cap (U \setminus D) \cap f^{-1}(B_{r_k}(y_k) \cap \mathcal{R}_Y(\epsilon/6, t_i, \tau_i))$. Thus (\ref{jjnhytu}) is satisfied for $\meas_X$-a.e. $x \in D_j \cap (U \setminus D) \cap f^{-1}(\mathcal{R}_Y(\epsilon/6, t_i, \tau_i))$ because $k$ is arbitrary.

For fixed $s \in (0, 1)$ and $m \in \mathbb{N}$, let
\begin{equation}
\phi^{s, m}:=\sum_i^me^{-\lambda_i^Ys}\left(\int_Y\phi \cdot \phi_i^Y\di \mathcal{H}^N\right)\phi_i^Y.
\end{equation}
For fixed $s \in (0, 1)$, we see that $\{\phi^{s, m}\}_m$ is equi-Lipschitz and that for any sufficiently large $m$, we have
\begin{align}
|\nabla \phi^{s, m}|^2(y)&=\sum_{i, j}^me^{-(\lambda_i^Y +\lambda_j^Y)s}\left(\int_Y\phi \cdot \phi_i^Y\di \mathcal{H}^N\right) \cdot \left(\int_Y\phi \cdot \phi_j^Y\di \mathcal{H}^N\right)\langle \nabla \phi_i^Y, \nabla \phi_j^Y\rangle (y) \nonumber \\
&\le \sum_{i, j} e^{-(\lambda_i^Y +\lambda_j^Y)s}\left(\int_Y\phi \cdot \phi_i^Y\di \mathcal{H}^N\right) \cdot \left(\int_Y\phi \cdot \phi_j^Y\di \mathcal{H}^N\right)\langle \nabla \phi_i^Y, \nabla \phi_j^Y\rangle (y) +\epsilon \nonumber \\
&=|\nabla h_s\phi|^2(y)+\epsilon,\quad \text{for $\mathcal{H}^N$-a.e. $y \in Y$},
\end{align}
where we used (\ref{eq:eigenfunction}).
In particular we have $\mathbf{Lip}\phi^{s, m} \le \mathbf{Lip}h_s \phi+\epsilon$.
Thus applying (\ref{jjnhytu}) for $\phi^{s, m}$ instead of $\phi$ with our assumption yields that $\phi^{s, m} \circ f \in H^{1, 2}(U, \dist_X, \meas_X)$ holds with
\begin{align}
|\nabla (\phi^{s, m}\circ f)|(x) &\le (1-\epsilon)^{-1}\mathbf{Lip}\phi^{s, m} \cdot \left(\left| (\tilde{\Phi}_{t_i}^Y \circ f)^*g_{\ell^2}\right|_B(x)\right)^{1/2} \nonumber \\
&\le  (1-\epsilon)^{-1}\cdot (\mathbf{Lip}h_s \phi+\epsilon) \cdot \left(\left| (\tilde{\Phi}_{t_i}^Y \circ f)^*g_{\ell^2}\right|_B(x)\right)^{1/2} \nonumber \\
&\le  (1-\epsilon)^{-1} \cdot (e^{-Ks}\mathbf{Lip} \phi +\epsilon) \cdot \left(\left| (\tilde{\Phi}_{t_i}^Y \circ f)^*g_{\ell^2}\right|_B(x)\right)^{1/2}
\end{align}
for $\meas_X$-a.e. $x \in D_j \cap (U \setminus D) \cap f^{-1}(\mathcal{R}_Y(\epsilon/6, t_i, \tau_i))$, where we used (\ref{absbayraywsrai}).  
In particular for $\meas_X$-a.e. $x \in D_j \cap (U \setminus D) \cap f^{-1}(\bigcap_{l \ge i}\mathcal{R}_Y(\epsilon/6, t_l, \tau_l))$, we have
\begin{equation}\label{eq:gra}
|\nabla (\phi^{s, m}\circ f)|(x)\le  (1-\epsilon)^{-1} \cdot (e^{-Ks}\mathbf{Lip} \phi +\epsilon) \cdot \left(\left| (\tilde{\Phi}_{t_l}^Y \circ f)^*g_{\ell^2}\right|_B(x)\right)^{1/2}
\end{equation}
for any $l \ge i$.
Thus combining Proposition \ref{prop:def}, 
letting $m \to \infty$, $s \to 0^+$ and then $l \to \infty$ in a weak form of (\ref{eq:gra})
\begin{align}
\int_E|\nabla (\phi^{s, m}\circ f)|^2\di \meas_X&\le \int_E  (1-\epsilon)^{-2} \cdot (e^{-Ks}\mathbf{Lip} \phi +\epsilon)^2 \cdot \left| (\tilde{\Phi}_{t_l}^Y \circ f)^*g_{\ell^2}\right|_B\di \meas_X, \nonumber \\
&\quad \forall E \subset D_j \cap (U \setminus D) \cap f^{-1}\left(\bigcap_{l \ge i}\mathcal{R}_Y(\epsilon/6, t_l, \tau_l)\right),
\end{align}
show
\begin{equation}
\int_E|\nabla (\phi \circ f)|^2\di \meas_X \le \int_E (1-\epsilon)^{-2} \cdot (\mathbf{Lip} \phi +\epsilon)^2\cdot \left|f^*g_Y\right|_B\di \meas_X.
\end{equation}
Therefore we have for $\meas_X$-a.e. $x \in D_j \cap (U \setminus D) \cap f^{-1}(\bigcap_{l \ge i}\mathcal{R}_Y(\epsilon/6, t_l, \tau_l))$
\begin{equation}\label{asbrauwauruasuj}
|\nabla (\phi \circ f)|(x) \le  (1-\epsilon)^{-1}\cdot  (\mathbf{Lip} \phi +\epsilon) \cdot \left(\left|f^*g_Y\right|_B(x)\right)^{1/2},
\end{equation}
which completes the proof of (2) with 
\begin{equation}\label{edftgv}
G_f(x) \le \left(\left|f^*g_Y\right|_B(x)\right)^{1/2},\quad \text{for $\meas_X$-a.e. $x \in U$}
\end{equation}
because $\epsilon, i$ and $j$ are arbitrary in (\ref{asbrauwauruasuj}). 

Finally let us prove the reverse inequality of (\ref{edftgv}). In order to simplify our notation, put $A:=D_i$ and the restriction of $f$ to $A$ is also denoted by the same notation $f$.
We can find $x \in A$ and a convergent sequence $r_i \to 0^+$ as in Theorem \ref{prop:rademacher} (we will use the same notations as in Theorem \ref{prop:rademacher} below). Moreover with no loss of generality we can assume that $x$ is a $2$-Lebesgue point of $G_f$.
Let us denote the map $f^0:\mathbb{R}^n \to \mathbb{R}^N$ by
\begin{equation}
f^0(x_1, x_2, \ldots, x_n)=(x_1, \ldots, x_n)M
\end{equation}
for some $n \times N$-matrix $M$. Put 
\begin{equation}\label{asdfaoieabs}
(X_i, \dist_{X_i}, \meas_{X_i}, x):=\left(X, r_i^{-1}\dist_X, \meas_X(B_{r_i}(x))^{-1}\meas_X, x\right) \stackrel{\mathrm{pmGH}}{\to} (\mathbb{R}^n, \dist_{\mathbb{R}^n}, \omega_n^{-1}\mathcal{H}^n, 0_n) 
\end{equation}
and 
\begin{equation}
 (Y_i, \dist_{Y_i}, f(x)):=(Y, r_i^{-1}\dist_Y, f(x)) \stackrel{\mathrm{pGH}}{\to} (\mathbb{R}^N, \dist_{\mathbb{R}^N}, 0_N).
\end{equation}
Thanks to \cite[Cor.4.12]{AmbrosioHonda2} with no loss of generality we can assume that there exists a sequence of harmonic maps $H_i=(h_{i, j})_j:B_2^{\dist_{X_i}}(x) \to \mathbb{R}^n$ such that $H_i$ converge uniformly to $\mathrm{id}_{\mathbb{R}^n}$ on $B_2(0_n)$ with respect to the convergence (\ref{asdfaoieabs}). Then we define the map $F_i:B_2^{\dist_{X_i}}(x) \to \mathbb{R}^N$ by 
\begin{equation}
F_i=(f_{i,1},\ldots, f_{i, N}):=(h_{i, 1},\ldots,h_{i, n}) \cdot M.
\end{equation}
Note that Theorem \ref{prop:rademacher} ensures
\begin{equation}
\int_{B_1^{\dist_{X_i}}(x)}\left| f^*g_{Y_i}-F_i^*g_{\mathbb{R}^N}\right|\di \meas_{X_i}  \to 0,
\end{equation}
in particular
\begin{equation}\label{ppnnfhbgppnnfhbg}
\int_{B_1^{\dist_{X_i}}(x)}\left| \left|f^*g_{Y_i}\right|_B-\left|F_i^*g_{\mathbb{R}^N}\right|_B\right|\di \meas_{X_i}  \to 0.
\end{equation}
Let us prove that
\begin{equation}\label{ppnnfhbg}
\int_{B_1^{\dist_{X_i}}(x)}\left|F_i^*g_{\mathbb{R}^N}\right|_B\di \meas_{X_i}  \to \frac{1}{\omega_n}\int_{B_1(0_n)}|(f^0)^*g_{\mathbb{R}^N}|_B\di \mathcal{H}^n=:C(M).
\end{equation}
Put
\begin{equation}
\epsilon_i:=\sup_{j, k}\int_{B_2^{\dist_{X_i}}(x)}\left| \delta_{jk}-\langle \nabla h_{i, j}, \nabla h_{i, k}\rangle \right|\di \meas_{X_i} \to 0
\end{equation}
and
\begin{equation}
B_i:=\left\{y \in B_1^{\dist_{X_i}}(x); \sup_{r\le 1}\frac{1}{\meas_{X_i}(B_r^{\dist_{X_i}}(y))}\int_{B_r^{\dist_{X_i}}(y)}\left| \delta_{jk}-\langle \nabla h_{i, j}, \nabla h_{i, k}\rangle \right|\di \meas_{X_i}<(\epsilon_i)^{1/2}, \forall j, \forall k\right\}.
\end{equation}
Then the maximal function theorem (e.g. \cite[Th.2.2]{Heinonen}) shows that
\begin{equation}\label{eq:quaiti}
\meas_{X_i}(B_1^{\dist_{X_i}}(x) \setminus B_i) \le C(\tilde{K}, \tilde{N}) (\epsilon_i)^{1/2}
\end{equation}
holds under assuming that $(X, \dist_X, \meas_X)$ is an $\RCD (\tilde{K}, \tilde{N})$ space.
Let us now recall that for any symmetric bilinear form $L: V \times V \to \mathbb{R}$ over an $n$-dimensional Hilbert space $(V, \langle \cdot, \cdot \rangle)$ and any basis $\{v_i\}_{i=1}^n$ of $V$ with $|\langle v_i, v_j\rangle -\delta_{ij}|<\epsilon$ for all $i, j$, we have
\begin{equation}\label{ansb73728}
|L|_B=\sup_{\sum_i(a_i)^2=1}L\left(\sum_ia_iv_i, \sum_ia_iv_i\right) \pm C(n)|L|_B\epsilon,
\end{equation}
where we used a notation 
\begin{equation}
a=b \pm \epsilon \Longleftrightarrow |a-b|\le \epsilon
\end{equation}
(compare with \cite[(5.36)]{AHPT}). 
Applying (\ref{ansb73728}) for $\{\nabla h_{i, j}\}_j$ on $B_i$ shows for $\meas_{X_i}$-a.e. $z \in B_i$;
\begin{equation}\label{labet}
|F_i^*g_{\mathbb{R}^N}|_B(z) =\sup_{\sum_j(a_j)^2=1} F_i^*g_{\mathbb{R}^N}\left(\sum_ja_j\nabla h_{i, j}, \sum_ja_j\nabla h_{i, j} \right)\pm C(n) |F_i^*g_{\mathbb{R}^N}|_B(z) (\epsilon_i)^{1/2}.
\end{equation}
On the other hand, by the definition of $B_i$ for any $a_j \in \mathbb{R}$ with $\sum_j(a_j)^2=1$ we have
\begin{align}\label{ojnh}
&\left|F_i^*g_{\mathbb{R}^N}\left(\sum_ja_j\nabla h_{i, j}, \sum_ja_j\nabla h_{i, j} \right)\right|(z) \nonumber \\
&=\left|(f^0)^*g_{\mathbb{R}^N}\left(\sum_ja_j\nabla x_j, \sum_ja_j\nabla x_j\right)\right|(z) \pm C(\tilde{K}, \tilde{N}, \mathbf{Lip}f) (\epsilon_i)^{1/2}
\end{align}
for $\meas_{X_i}$-a.e. $z \in B_i$.
In particular combining (\ref{labet}) with (\ref{ojnh}) yields for $\meas_{X_i}$-a.e. $z \in B_i$
\begin{align}
|F_i^*g_{\mathbb{R}^N}|_B(z)&=\sup_{\sum_j(a_j)^2=1}(f^0)^*g_{\mathbb{R}^N}\left(\sum_ja_j\nabla x_j, \sum_ja_j\nabla x_j\right)\pm C(\tilde{K}, \tilde{N}, \mathbf{Lip}f) (\epsilon_i)^{1/2} \nonumber \\
&=C(M)\pm C(\tilde{K}, \tilde{N}, \mathbf{Lip}f) (\epsilon_i)^{1/2}.
\end{align} 
Therefore by (\ref{eq:quaiti}) we have
\begin{align}
\int_{B_1^{\dist_{X_i}}(x)}\left|F_i^*g_{\mathbb{R}^N}\right|_B\di \meas_{X_i}&=\int_{B_i}\left|F_i^*g_{\mathbb{R}^N}\right|_B\di \meas_{X_i} + \int_{B_1^{\dist_{X_i}}(x) \setminus B_i}\left|F_i^*g_{\mathbb{R}^N}\right|_B\di \meas_{X_i} \nonumber \\
&=\meas_{X_i}(B_i)C(M) \pm C(\tilde{K}, \tilde{N}, \mathbf{Lip}f) (\epsilon_i)^{1/2} \nonumber \\
&\to C(M),
\end{align}
which completes the proof of (\ref{ppnnfhbg}).

In particular by (\ref{ppnnfhbgppnnfhbg})
\begin{equation}\label{nbaybshf}
\frac{1}{\meas_X(B_{r_i}(x))}\int_{B_{r_i}(x)}\left| \left|f^*g_{Y_i}\right|_B-C(M)\right|\di \meas_{X}  \to 0.
\end{equation}
Let us take a linear function $\phi:\mathbb{R}^n \to \mathbb{R}$ such that $|\nabla \phi|\equiv 1$ and that 
\begin{equation}
(f^0)^*g_{\mathbb{R}^N}(\nabla \phi, \nabla \phi) \equiv C(M).
\end{equation}
Fix $\epsilon \in (0, 1)$ and take $\psi \in C^{\infty}_c(\mathbb{R}^n)$ with $\psi \equiv 1$ on $B_2(0_n)$ and $|\nabla \psi|+|\Delta \psi|\le \epsilon$ on $\mathbb{R}^n$.
Applying (the proof of) \cite[Prop.1.10.2]{AmbrosioHonda} for $\psi\phi$, there exists a sequence of $\phi_i \in \mathrm{Test}F(X_i, \dist_{X_i}, \meas_{X_i})$ such that $\phi_i$ converge uniformly to $\phi$ on $B_2(0_n)$ with respect to the convergence (\ref{asdfaoieabs}) and that $\limsup_{i \to \infty}\mathbf{Lip}\phi_i\le (1+\epsilon)\mathbf{Lip}\phi=1+\epsilon$.
Then since
\begin{equation}\label{pphnhfyttrb}
|\nabla (\phi_i \circ f)|(z) \le \mathbf{Lip} \phi_i \cdot G_f(z),\quad \text{for $\meas_X$-a.e. $z \in B_2^{\dist_{X_i}}(x)$}, 
\end{equation}
letting $i \to \infty$ in the weak form of (\ref{pphnhfyttrb}) easily yields
\begin{equation}\label{98uuui9}
C(M)=|\nabla (\phi \circ f^0)|^2 \le (1+\epsilon)^2\overline{G_f}(x)^2
\end{equation}
where we used our assumption that $x$ is a $2$-Lebesgue point of $G_f$ (recall Definition \ref{lebesguepoi} for the definition of $\overline{G_f}(x)$).  Combining (\ref{nbaybshf}) with (\ref{98uuui9}) shows
\begin{equation}
\left(\overline{\left|f^*g_Y\right|_B}(x)\right)^{1/2}\le (1+\epsilon)\overline{G_f}(x)
\end{equation}
if $x$ is also a $2$-Lebesgue point of $|f^*g_Y|_B$, which completes the proof of (\ref{ddccrrtt}) because $\epsilon$ is arbitrary.

Finally it follows from Propositions \ref{proppullb} and (\ref{ddccrrtt}) that (\ref{ddccrrttww}) holds.
\end{proof}
Thanks to Theorem \ref{propcompatr}, in the sequel, we can focus on Sobolev maps instead of $0$-Sobolev maps. 
The following proposition gives an asymptotic behavior of $t^{(N+2)/2}\mathcal{E}_{U, Y, t}(f)$ as $t \to 0$.
\begin{proposition}\label{prop:l1}
Let  $f:U \to Y$ be a Sobolev map with $f(U \setminus D) \subset \mathcal{R}_Y(\{(t_i, \tau_i)\}_i)$ for some $\{(t_i, \tau_i)\}_i$ and some $\meas_X$-negligible set $D$. Then $(\tilde{\Phi}_{t_i}^Y\circ f)^*g_{L^2}$ $L^1$-strongly converge to $f^*g_Y$ on $U$ with for all $i \in \mathbb{N}$ and $\meas_X$-a.e. $x \in A_i:=\bigcap_{j \ge i}\mathcal{R}_Y(\epsilon, t_j, \tau_j)$,
\begin{equation}\label{eqeq}
|(\tilde{\Phi}_{t_i}^Y \circ f)^*g_{L^2}-f^*g_Y|(x)\le C(n)\epsilon \min \left\{ |f^*g_Y|(x), |(\tilde{\Phi}_{t_i}^Y \circ f)^*g_{L^2}|(x)\right\}.
\end{equation}
In particular we see that 
$c_Nt_i^{(N+2)/2}e_{Y, t_i}(f)$ $L^1$-strongly converge to $e_{Y}(f)$ on $U$.
\end{proposition}
\begin{proof}
Recall that we already know $\meas_X$-a.e. pointwise convergence $|(\tilde{\Phi}_{t_i}^Y \circ f)^*g_{L^2}-f^*g_Y| \to 0$ after passing to a subsequence (see the proof of Proposition \ref{prop:l17773}). Thus the desired $L^1$-convergence is justified by this with (\ref{abasurbauawurb}) and the dominated convergence theorem.
Moreover (\ref{eqeq}) is a direct consequence of Proposition \ref{prop:cauchy}.
\end{proof}

Finally let us give a precise description of the asymptotics of Korevaar-Schoen energy densities by the pull-back.
\begin{theorem}\label{koseergy}
Let $f:U \to Y$ be a Sobolev map with $f(U \setminus D) \subset \mathcal{R}_Y(\{(t_i, \tau_i)\}_i)$ for some $\{(t_i, \tau_i)\}_i$ and some $\meas_X$-negligible subset $D$ of $U$. 
Then $(n+2)\mathrm{ks}_{Y, r}(f)^2$ $L^1$-strongly converge to $e_Y(f)$ on $U$ as $r \to 0^+$. In particular we have
\begin{equation}
\mathcal{E}_{U, Y}(f)=\frac{n+2}{2}\mathcal{E}_{U, Y}^{KS}(f).
\end{equation} 
\end{theorem}
\begin{proof}
Note that Lebesgue differentiation theorem with Proposition \ref{prop:lipschitzlusin} and Corollary \ref{cor:ksks} easily yields
\begin{equation}\label{aboriaioraiwraio}
(n+2)\mathrm{ks}_{Y, r}(f)(x)^2 \to e_Y(f)(x), \quad \text{for $\meas_X$-a.e. $x \in U$.}
\end{equation}
Let $G_f=0$ outside $U$.
Note that since the function $z \mapsto \dist_Y(f(x), z)$ is $1$-Lipschitz for fixed $x \in X$, we have
\begin{equation}\label{bariiairbausb}
\mathrm{ks}_{Y, r}(f)(x)^2 \le \frac{1}{\meas_X(B_r(x))}\int_{B_r(x)}G_f^2\di \meas_X=:G_{f, r}(x)^2.
\end{equation}
Then letting $\phi_r(z):=\int_{B_r(z)}\frac{1}{\meas_X(B_r(x))}\di \meas_X(x)$, Fubini's theorem shows
\begin{align}\label{asssrrff}
\int_UG_{f, r}^2\di \meas_X&=\int_U\left(\frac{1}{\meas_X(B_r(x))}\int_{B_r(x)}G_f(z)^2\di \meas_X(z)\right)\di \meas_X(x) \nonumber \\
&=\int_U\int_{X}\frac{1}{\meas_X(B_r(x))}1_{B_r(x)}(z)G_f(z)^2\di\meas_X(z)\di \meas_X(x) \nonumber \\
&=\int_U\int_{X}\frac{1}{\meas_X(B_r(x))}1_{B_r(z)}(x)G_f(z)^2\di\meas_X(x)\di \meas_X(z) \nonumber \\
&=\int_UG_f(z)^2\left(\int_{B_r(z)}\frac{1}{\meas_X(B_r(x))}\di \meas_X(x)\right)\di \meas_X(z) \nonumber \\
&=\int_UG_f(z)^2\phi_r(z) \di \meas_X(z)
\end{align}
Since the Bishop-Gromov inequality yields
\begin{equation}
\sup_{ r \in (0, 1)}\| \phi_r\|_{L^{\infty}}<\infty
\end{equation}
and $\phi_r(z) \to 1$ holds for $\meas_X$-a.e. $z \in X$ because of Theorem \ref{thmRCD decomposition},
the dominated convergence theorem with (\ref{asssrrff}) shows
\begin{equation}
\int_UG_{f, r}^2\di \meas_X \to \int_UG_f^2\di \meas_X.
\end{equation}
Recall the following general fact;
\begin{itemize}
\item Let $(W, \meas_W)$ be a measure space and let $f_i, g_i, f, g \in L^1(W, \meas_W) (i=1, 2, \ldots)$. Assume that $f_i(w), g_i(w) \to f(w), g(w)$ hold for $\meas_W$-a.e. $w \in W$, respectively, that $|f_i|(w) \le g_i(w)$ holds for $\meas_W$-a.e. $w \in W$, and that $\lim_{i \to \infty}\|g_i\|_{L^1}=\|g\|_{L^1}$. Then $f_i \to f$ in $L^1(W, \meas_W)$. 
\end{itemize}
See \cite[Lem.2.4]{AmbrosioHondaTewodrose} for the proof.

Applying this fact for $G_{f, r}, G_f$ with (\ref{bariiairbausb}) and (\ref{aboriaioraiwraio}) shows that 
$(n+2)\mathrm{ks}_{Y, r}(f)^2$ $L^1$-strongly converges to $e_Y(f)$ on $U$.
\end{proof}
\subsection{Uniformly weakly smooth set and compactness}
In this section we discuss uniformly weekly smooth set in the following sense.
\begin{definition}[Uniformly weakly smooth set]
Let $t_i \to 0^+$ be a convergent sequence in $(0, 1)$ and let $\{\tau_i\}_i$ be a sequence in $(0, \infty)$.
\begin{enumerate}
\item A sequence of subsets $\{A_l\}_l$ of $Y$ is said to be \textit{uniformly weakly smooth for $\{(t_i, \tau_i)\}_i$} if 
for any $\epsilon \in (0, 1)$ there exists $i \in \mathbb{N}$ such that 
\begin{equation} 
A_l \subset \bigcap_{j \ge i}\mathcal{R}_Y(\epsilon, t_j, \tau_j),\quad \forall l \in \mathbb{N}
\end{equation}
holds.
\item A subset $A$ is said to be \textit{uniformly weakly smooth for $\{(t_i, \tau_i)\}_i$} if the constant sequence $\{A\}$ is uniformly weakly smooth for $\{(t_i, \tau_i)\}_i$.
\end{enumerate}
\end{definition}
Let us give a compactness result for Sobolev maps under the unifom weak smoothness of the image.
\begin{theorem}[Compactness]\label{corcompactness}
Let $R \in (0, \infty]$, let $x \in X$, let $f_i:B_R(x) \to Y$ be a sequence of Sobolev maps. Assume that the following two conditions hold:
\begin{enumerate}
\item The sequence $\{f_i(B_R(x) \setminus D_i)\}_i$ is uniformly weakly smooth for some $\{(t_i, \tau_i)\}_i$ and some $\meas_X$-negligible subsets $D_i$ of $B_R(x)$.
\item We have
\begin{equation}\label{eqwwd}
\sup_{i}\mathcal{E}_{B_R(x), Y}(f_i)<\infty.
\end{equation}
\end{enumerate}
Then after passing to a subsequence there exists a Sobolev map $f:B_R(x) \to Y$ such that $f(B_R(x) \setminus D)$ is uniformly weakly smooth for $\{(t_i, \tau_i)\}_i$ for some $\meas_X$-negligible subset $D$ of $B_R(x)$, that $f_i$ converge to $f$ for $\meas_X$-a.e. $x \in B_R(x)$ and that 
\begin{equation}\label{eq:lowesemicoasawas}
\liminf_{i \to \infty}\int_{B_R(x)}\phi_i e_{Y}(f_i)\di \meas_X \ge \int_{B_R(x)}\phi e_{Y}(f)\di \meas_X
\end{equation}
for any $L^2_{\mathrm{loc}}$-strongly convergent sequence $\phi_i \to \phi$ with $\phi_i \ge 0$ and $\sup_i\|\phi_i\|_{L^{\infty}}<\infty$.
In particular 
\begin{equation}\label{oorr}
\liminf_{i \to \infty}\mathcal{E}_{B_R(x), Y}(f_i)\ge \mathcal{E}_{B_R(x), Y}(f).
\end{equation}
\end{theorem}
\begin{proof}
Note that (\ref{eqeq}) with (\ref{eqwwd}) shows
\begin{equation}\label{baswabtauis}
C:=\sup_{i,l}t_i^{(N+2)/2}\mathcal{E}_{B_R(x), Y, t_i}(f_l)<\infty.
\end{equation}
By Theorem \ref{propcompact} after passing to a subsequence there exists a map $f:B_R(x) \to Y$ such that $f_j$ converge to $f$ for $\meas_X$-a.e. $x \in B_R(x)$ and that $f$ is a $t_i$-Sobolev map for any $i$. In particular by the first assumption we see that $f(B_R(x) \setminus D)$ is uniformly weakly smooth for $\{(t_i, \tau_i)\}_i$ and some $\meas_X$-negligible set $D$. 
Since Theorem \ref{propcompact} yields
\begin{equation}\label{137282910}
\liminf_{j \to \infty}\mathcal{E}_{B_R(x), Y, t_i}(f_j)\ge \mathcal{E}_{B_R(x), Y, t_i}(f),\quad \forall i \in \mathbb{N},
\end{equation}
letting $i \to \infty$ in (\ref{137282910}) with (\ref{baswabtauis}) proves that $f$ is a $0$-Sobolev map. Thus Theorem \ref{propcompatr} shows that $f$ is a Sobolev map.

Then (\ref{eqeq}) shows
\begin{equation}
\int_{B_R(x)}\left| c_Nt_i^{(N+2)/2}e_{Y, t_i}(f)-e_Y(f)\right| \di \meas_X\le C(n)\epsilon \int_{B_R(x)}\left| c_Nt_i^{(N+2)/2}e_{Y, t_i}(f)\right| \di \meas_X\le C(n)\epsilon \cdot C \cdot c_N,
\end{equation}
whenever $f(z) \in \mathcal{R}_Y(\epsilon, t_i, \tau_i)$ holds for $\meas_X$-a.e. $z \in B_R(x)$.
Thus it follows from (\ref{eq:lowesemico11}) that (\ref{eq:lowesemicoasawas}) is satisfied.
\end{proof}
As a corollary of the above results we obtain a variant of $\Gamma-$convergence for $\mathcal{E}_{U, Y, t_i}$ to $\mathcal{E}_{U, Y}$;
\begin{corollary}[A variant of variational convergence]
Let $t_i \to 0^+$ be a convergent sequence in $(0, 1)$ and let $\{\tau_i\}_i$ be a sequence in $(0, \infty)$. Then the energies $c_Nt_i^{(N+2)/2}\mathcal{E}_{U, Y, t_i}$ converge to $\mathcal{E}_{U, Y}$ in the following sense.
\begin{enumerate}
\item We have
\begin{equation}
\liminf_{i \to \infty}c_Nt_i^{(N+2)/2}\mathcal{E}_{U, Y, t_i}(f_i) \ge \mathcal{E}_{U, Y}(f)
\end{equation}
for any $\meas_X$-a.e. convergent sequence $f_i:U \to Y$ to $f:U \to Y$ satisfying that $\{f_i(U \setminus D_i)\}_i$ is uniformly weakly smooth for $\{(t_i, \tau_i)\}_i$ and some $\meas_X$-negligible subsets $D_i$ of $U$. 
\item We have
\begin{equation}
\lim_{i \to \infty}c_Nt_i^{(N+2)/2}\mathcal{E}_{U, Y, t_i}(f)=\mathcal{E}_{U, Y}(f)
\end{equation}
for any $0$-Sobolev map $f:U \to Y$.
\end{enumerate}
\end{corollary}
\begin{remark}
With the help of Theorem \ref{corcompactness}, in the case of weakly smooth targets, one can show the full variational convergence (i.e. $\Gamma-$convergence) of the approximate energy to our new energy. Since we will not make use of it here, we do not pursue into this direction. 
\end{remark}
\begin{proof}
Let us check only (1) because (2) is a direct consequence of Proposition \ref{prop:l1}.
Applying Vitali's covering theorem, there exists a pairwise disjoint sequence of closed balls $\{\overline{B}_{r_i}(x_i)\}_i$ such that $\overline{B}_{5r_i}(x_i) \subset U$ holds for any $i$ and that 
\begin{equation}\label{vitali}
U \setminus \bigsqcup_{i=1}^k\overline{B}_{r_i}(x_i) \subset \bigcup_{i=k+1}^{\infty}\overline{B}_{5r_i}(x_i),\quad \forall k \in \mathbb{N}
\end{equation}
holds. Fix $k \in \mathbb{N}$ and take $f_i, f$ as in the assumption. Theorem \ref{corcompactness} yields
\begin{align}\label{asnauarbwawa}
\liminf_{i \to \infty}c_Nt_i^{(N+2)/2}\mathcal{E}_{U, Y, t_i}(f_i)&\ge \liminf_{i \to \infty}\sum_{j=1}^k\left(c_Nt_i^{(N+2)/2}\mathcal{E}_{B_{r_j}(x_j), Y, t_i}(f_i)\right) \nonumber \\
&\ge \sum_{j=1}^k\mathcal{E}_{B_{r_i}(x_i), Y}(f) \nonumber \\
&=\mathcal{E}_{\sqcup_{i=1}^kB_{r_i}(x_i), Y}(f).
\end{align}
Then letting $k \to \infty$ in (\ref{asnauarbwawa}) completes the proof of (1) because of
\begin{equation}
\lim_{k \to \infty}\sum_{i=k+1}^{\infty}\meas_X(B_{5r_i}(x_i))\le C\lim_{k \to \infty}\sum_{i=k+1}^{\infty}\meas_X(B_{r_i}(x_i)) =0.
\end{equation}
\end{proof}

\subsection{Special case}\label{sub:special}
In this section let us consider a Borel map $f:A \to Y$ with $f_{\sharp}(\meas_X\res_A) \ll \mathcal{H}^N$.
\begin{corollary}\label{cor:isom}
Let $B$ be a Borel subset of $Y$ and let $F:B \to Y$ be a locally isometric embedding as metric spaces, namely, for any $y \in B$ there exists $r \in (0, 1)$ such that $\dist_Y(F(z), F(w))=\dist_Y(z, w)$ holds for all $z, w \in B_r(y) \cap B$. Then 
\begin{equation}\label{eqloca}
F^*g_Y=g_Y.
\end{equation}
In particular $\mathcal{E}_{B, Y}(F)=N\mathcal{H}^N(B)/2$.
\end{corollary}
\begin{proof}
Applying Proposition \ref{prop:iso} as $f=F, h=F^{-1}$ completes the proof.
\end{proof}
In Corollary \ref{cor:isom}, recall that in general, the equality (\ref{eqloca}) does not imply the local isometry of $F$. See Remark \ref{asahoarbasj}.

Finally we introduce the following result which is a combination of previous results.
\begin{theorem}\label{asbairanwi}
Let  $f:U \to Y$ be a weakly smooth map with $f_{\sharp}(\meas_X\res_U)\ll \mathcal{H}^N$ and 
\begin{equation}
\liminf_{t \to 0^+}t^{(N+2)/2}\mathcal{E}_{U, Y, t}(f)<\infty.
\end{equation}
Then we have the following.
\begin{enumerate}
\item The map $f$ is a Sobolev map.
\item The normalized $t$-energy densities $c_Nt^{(N+2)/2}e_{Y, t}(f)$ and the normalized Korevaar-Schoen energy densities $(n+2)\mathrm{ks}_{Y, t}(f)^2$ $L^1$-strongly converge to $e_Y(f)$ on $U$ as $t \to 0^+$.
\item We have
\begin{equation}
G_f(x)=\mathrm{Lip} (f|_{\tilde{D}})(x)=\left| f^*g_Y\right|_B(x),\quad \text{for $\meas_X$-a.e. $x \in \tilde{D}$}
\end{equation}
whenever the restriction of $f$ to a Borel subset $\tilde{D}$ of $U$ is Lipschitz. 
\end{enumerate}
\end{theorem}
\begin{proof}
This is a direct consequence of Propositions \ref{propdensereg}, \ref{prop:l1}, Corollary \ref{cor:ksks}, Theorems \ref{propcompatr} and \ref{koseergy}
\end{proof}
\section{A generalization of Takahashi's theorem}\label{takahashitheoremsec}
Let us fix a finite dimensional $\RCD$ space $(X, \dist_X, \meas_X)$. We start this section by giving the definition of the $L^1$-Laplacian. 
\begin{definition}[$D_1(\Delta_X)$]
Let us denote by $D_1(\Delta_X)$ the set of all functions $\phi \in H^{1, 2}(X, \dist_X, \meas_X)$ satisfying that there exists a unique $\psi \in L^1(X, \meas_X)$, denoted by $\Delta_X\phi$,  such that 
\begin{equation}\label{eqtest}
\int_X\langle \dist \phi, \dist \tilde{\psi} \rangle \di \meas_X=-\int_X\psi \tilde{\psi} \di \meas_X
\end{equation}
for any Lipschitz function $\tilde{\psi}$ on $X$ with compact support.
\end{definition}
Note that it is easy to check that (\ref{eqtest}) also holds for all $\phi \in D_1(\Delta_X)$ and any $\tilde{\psi} \in H^{1, 2}(X, \dist_X, \meas_X) \cap L^{\infty}(X, \meas_X)$ because letting $s \to 0^+$ and then $R\to \infty$ in the equality
\begin{equation}
\int_X\langle \dist \phi, \dist (f_R \cdot h_s(\tilde{\psi})) \rangle \di \meas_X=-\int_X\psi \cdot f_R \cdot h_s(\tilde{\psi}) \di \meas_X, \quad \forall s \in (0, \infty)
\end{equation}
complete the proof, where $f_R$ is a cut-off Lipschitz function satisfying $f_R|_{B_R(x)} \equiv 1, \supp f_R \subset B_{R^2}(x)$ and $|\nabla f_R|\le R^{-1}$.
\begin{proposition}\label{prop:harmonicchara}
Let $(\mathbb{S}^k(1), \dist_{\mathbb{S}^k(1)}, \mathcal{H}^k)$ be the $k$-dimensional standard unit sphere and let $f=(f_i)_i:X \to \mathbb{S}^k(1)$ be a Sobolev map.
Then the following four conditions are equivalent.
\begin{enumerate}
\item For any Lipschitz map $\phi :X \to \mathbb{R}^{k+1}$ with compact support, we have
\begin{equation}
\frac{\di}{\di t}\Big|_{t=0}\mathcal{E}_{X, \mathbb{S}^k(1)}\left( \frac{f+t\phi}{|f+t\phi |}\right)=0.
\end{equation}
\item  Each $f_i$ is in $D_1(\Delta_X)$ with
\begin{equation}
\Delta_Xf_i+e_{\mathbb{S}^k(1)}(f)f_i=0,\quad \forall i \in \{1, \ldots, k+1\}.
\end{equation}
\item We have
\begin{equation}\label{eaabsbatayrbays}
\frac{\di}{\di t}\Big|_{t=0}\mathcal{E}_{X, \mathbb{S}^k(1)}(f_t)=0
\end{equation}
for any map $(t, x) \mapsto f_t(x)=(f_{t, i}(x))_i \in \mathbb{S}^k(1)$ satisfying that $f_0=f$, that $f_{t, i} \in H^{1, 2}(X, \dist_X, \meas_X)$ and that the map $t \mapsto f_{t, i}$ is continuous at $0$ in $H^{1, 2}(X, \dist_X, \meas_X)$ with
\begin{equation}
\lim_{t \to 0}\int_X\left(\left| \frac{f_t-f}{|t|^{1/2}}\right|^2e_{\mathbb{S}^k(1)}(f) + \sum_{i=1}^{k+1}\left|\dist \left(\frac{f_{t, i}-f_i}{|t|^{1/2}}\right)\right|^2\right)\di \meas_X=0.
\end{equation}
\item The equality (\ref{eaabsbatayrbays}) holds for any map $(t, x) \mapsto f_t(x)=(f_{t, i}(x))_i \in \mathbb{S}^k(1)$ satisfying that $f_0=f$, that $f_{t, i} \in H^{1, 2}(X, \dist_X, \meas_X)$ and that the map $t \mapsto f_{t, i}$ is differentiable at $0$ in $H^{1, 2}(X, \dist_X, \meas_X)$ with
\begin{equation}
\frac{\di}{\di t}\Big|_{t=0}f_{t, i} \in L^{\infty}(X, \meas_X).
\end{equation}
\end{enumerate} 
\end{proposition}
\begin{proof}
First let us prove the implication from (1) to (2). Assume that (1) holds.
Let $\phi=(\phi_1, \ldots, \phi_{k+1}):X \to \mathbb{R}^{k+1}$ be a Lipschitz map with compact support. Note that $|f+t\phi |>0$ holds for any sufficiently small $t \in (0, 1)$, in particular, we have $(f_i+t\phi_i)/|f+t\phi | \in H^{1, 2}(X, \dist_X, \meas_X)$ which implies that the map $x \mapsto (f+t\phi)/|f+t\phi|$ is a Sobolev map.  Since
\begin{equation}\label{eq:es}
\frac{1}{|f+t\phi|}=1-f\cdot \phi t +o(t)
\end{equation}
as $t \to 0^+$, by a direct calculation with (\ref{eq:es}), we have
\begin{equation}
\frac{\di}{\di t}\Bigl|_{t=0}\left(\int_X\left| \dist \left( \frac{f_i+t\phi_i}{|f+t\phi|}\right)\right|^2\di \meas_X \right) =2\int_X\langle \dist f_i, \dist (\phi_i-f_i f \cdot \phi)\rangle\di \meas_X.
\end{equation}
Then by Corollary \ref{prop:iso2}, we have
\begin{align}\label{eq:laplace}
0=\frac{\di}{\di t}\Big|_{t=0}\mathcal{E}_{X, \mathbb{S}^k(1)}\left( \frac{f+t\phi}{|f+t\phi |}\right)&=\frac{1}{2}\sum_{i=1}^{k+1}\frac{\di}{\di t}\Bigl|_{t=0}\left(\int_X\left| \dist \left( \frac{f_i+t\phi_i}{|f+t\phi|}\right)\right|^2\di \meas_X \right) \nonumber \\
&=\sum_{i=1}^{k+1}\int_X\langle \dist f_i, \dist (\phi_i-f_i f \cdot \phi)\rangle\di \meas_X.
\end{align}
Since
\begin{equation}
\sum_{i=1}^{k+1}\int_X\langle \dist f_i,  f_i\dist (f \cdot \phi) \rangle\di \meas_X={\color{blue}\frac{1}{2}}\sum_{i=1}^{k+1}\int_X\langle \dist f_i^2, \dist (f \cdot \phi)\rangle \di \meas_X=\frac{1}{2}\int_X\langle \dist |f|^2, \dist (f \cdot \phi)\rangle \di \meas_X=0,
\end{equation}
(\ref{eq:laplace}) is equivalent to
\begin{equation}
\sum_{i=1}^{k+1}\int_X\langle \dist f_i, \dist \phi_i\rangle \di \meas_{X}=\int_Xe_{\mathbb{S}^k(1)}(f)f\cdot \phi \di \meas_X
\end{equation}
which proves that (2) holds because $\phi$ is arbitrary.

Next let us prove the implication from (2) to (3). Assume that (2) holds. Let us take $f_s$ as in (3).
Then applying Corollary \ref{prop:iso2} again shows as $t \to 0$
\begin{align}
&\frac{{\color{blue}2}}{t}\left(\mathcal{E}_{X, \mathbb{S}^k(1)}(f_t)-\mathcal{E}_{X, \mathbb{S}^k(1)}(f)\right)\nonumber \\
&=2\sum_{i=1}^{k+1}\int_X\left\langle \dist \left(\frac{f_{t, i}-f_i}{t}\right), \dist f_i\right\rangle\di \meas_X +\sum_{i=1}^{k+1}\int_X\left\langle \dist \left(\frac{f_{t, i}-f_i}{t}\right), \dist \left(f_{t, i}-f_i\right)\right\rangle\di \meas_X\nonumber \\
&{\color{blue}=2\sum_{i=1}^{k+1}\int_X \left(\frac{f_{t, i}-f_i}{t}\right) \cdot e_{\mathbb{S}^k(1)}(f) \cdot f_i \di \meas_X  +\sum_{i=1}^{k+1}\int_X\left|\dist \left(\frac{f_{t, i}-f_i}{|t|^{1/2}}\right)\right|^2\di \meas_X}\nonumber \\
&=\int_X\left|\frac{f_t-f}{|t|^{1/2}}\right|^2\cdot e_{\mathbb{S}^k(1)}(f)\di \meas_X +\sum_{i=1}^{k+1}\int_X\left|\dist \left(\frac{f_{t, i}-f_i}{|t|^{1/2}}\right)\right|^2\di \meas_X\nonumber \\
&\to 0,
\end{align}
which completes the proof of (3), {\color{blue}where we used the elementary fact that $\sum_{i=1}^{k+1}|f_{t, i}-f_i|^2=2-2\sum_{i=1}^{k+1}f_{t, i}f_i$ in order to get the third equality above.}

Similarly under assuming (2), as $t \to 0$,
\begin{align}
&\frac{{\color{blue}2}}{t}\left(\mathcal{E}_{X, \mathbb{S}^k(1)}(f_t)-\mathcal{E}_{X, \mathbb{S}^k(1)}(f)\right)\nonumber \\
&=2\sum_{i=1}^{k+1}\int_X\left\langle \dist \left(\frac{f_{t, i}-f_i}{t}\right), \dist f_i\right\rangle\di \meas_X +\sum_{i=1}^{k+1}\int_X\left\langle \dist \left(\frac{f_{t, i}-f_i}{t}\right), \dist \left(f_{t, i}-f_i\right)\right\rangle\di \meas_X\nonumber \\
&\to 2\sum_{i=1}^{k+1}\int_X\left\langle \dist \left(\frac{\di}{\di t}\Big|_{t=0}f_t\right)_i, \dist f_i\right\rangle\di \meas_X \nonumber \\
&=2\sum_{i=1}^{k+1}\int_X\left(\frac{\di}{\di t}\Big|_{t=0}f_t\right)_i \cdot e_{\mathbb{S}^k(1)}(f)f_i\di \meas_X \nonumber \\
&=\int_X\left(\frac{\di}{\di t}\Big|_{t=0}|f_t|^2\right) \cdot e_{\mathbb{S}^k(1)}(f)\di \meas_X=0,
\end{align}
which proves (4), where we used a fact that (\ref{eqtest}) holds for any $\tilde{\psi} \in H^{1, 2}(X, \dist_X, \meas_X) \cap L^{\infty}(X, \meas_X)$.

Since the remaining implications, from (3) to (1) and from (4) to (1), are trivial because the map $t \mapsto (f+t\phi)/|f+t\phi|$ satisfies the assumptions of (3) and (4), we conclude. 
\end{proof}
Based on the above proposition, let us give the following definition;
\begin{definition}[Harmonic/minimal map into sphere]
Let $f:X \to \mathbb{S}^k(1)$ be a Sobolev map.
\begin{enumerate}
\item{(Harmonicity)} The map $f$ is said to be \textit{harmonic} if one of the four conditions in Proposition \ref{prop:harmonicchara} is satisfied (thus all hold),
\item{(Minimality)} The map $f$ is said to be \textit{minimal} if it is isometric {\color{blue}(namely $f^*g_{\mathbb{S}^k(1)}=g_X$)} and harmonic.
\end{enumerate}
\end{definition}
We are now in a position to introduce the main result in this section (Theorem \ref{takahashitheorem}). The following gives a generalization of Takahashi's theorem \cite[Th.3]{takahashi} to $\RCD$ spaces. 
\begin{theorem}[Generalization of Takahashi's theorem]\label{theoremtakahashi}
Assume that $(X, \dist_X, \meas_X)$ is a compact $\RCD(K, N)$ space. Let $f=(f_i)_i:X \to \mathbb{S}^k(1)$ be an isometric Sobolev map. Then the following three conditions are equivalent.
\begin{enumerate}
\item The map $f$ is harmonic (thus it is minimal).
\item The equality (\ref{eaabsbatayrbays}) holds for any map $(t, x) \mapsto f_t(x) \in \mathbb{S}^k(1)$ satisfying that $f_0=f$, that $f_{t, i} \in H^{1, 2}(X, \dist_X, \meas_X)$ and that the map $t \mapsto f_{t, i}$ is differentiable at $0$ in $H^{1, 2}(X, \dist_X, \meas_X)$.
\item We see that $\meas_X=c\mathcal{H}^n$ for some $c \in (0, \infty)$, that $(X, \dist_X, \mathcal{H}^n)$ is a non-collapsed $\RCD(K, n)$ space and that $f$ is an eigenmap with $\Delta_Xf_i+nf_i=0$ for any $i$, where $n$ denotes the essential dimension of $(X, \dist_X, \meas_X)$.
\end{enumerate}
In particular if the above conditions hold, then $X$ is bi-H\"older homeomorphic to an $n$-dimensional closed manifold, $f$ is $1$-Lipschitz and for all $\epsilon \in (0,1)$ and $x \in X$, there exists $r \in (0, 1)$ such that $f|_{B_r(x)}$ is a $(1 \pm \epsilon)$-bi-Lipschitz embedding.
\end{theorem}
\begin{proof}
First of all, note that the Proposition \ref{propsobtolip} shows that $f$ is Lipschitz, and that $\sum_i|\dist f_i|^2=e_Y(f)=|g_X|^2=n$ holds.

The implication from (1) to (2) follows from (3) of Proposition \ref{prop:harmonicchara}. The converse implication is justified by
the same reason in the proof of the implication from (3) to (1) in Proposition \ref{prop:harmonicchara}. Therefore we have the equivalence between (1) and (2).

The remaining equivalences are justified by \cite[Th.1.2]{honda20} (see also Proposition \ref{asataras}).
\end{proof}

\section{Behavior of energies with respect to measured Gromov-Hausdorff convergence}\label{secenergies}
Let us fix 
\begin{itemize}
\item a pointed measured Gromov-Hausdorff convergent sequences of pointed $\RCD(\hat{K}, \hat{N})$ spaces
\begin{equation}\label{eq:4}
(X_i, \dist_{X_i}, \meas_{X_i}, x_i) \stackrel{\mathrm{pmGH}}{\to} (X, \dist_X, \meas_X, x),
\end{equation}
\item a measured Gromov-Hausdorff convergent sequence of compact $\RCD(K, N)$ spaces
\begin{equation}\label{eq:mgh}
(Y_i, \dist_{Y_i}, \meas_{Y_i}) \stackrel{\mathrm{mGH}}{\to} (Y, \dist_Y, \meas_Y).
\end{equation}
\end{itemize}
We start this section by giving the following technical lemma  (recall Definition \ref{lebesguepoi}).
\begin{lemma}\label{pointwiseconv}
Let $R \in (0, \infty)$, let $f_i \in H^{1, 2}(B_R(x_i), \dist_{X_i}, \meas_{X_i})$ be an $H^{1, 2}$-bounded sequence, let $D_i$ be a sequence of $\meas_{X_i}$-negligible subsets of $B_R(x_i)$, and let $f \in L^2(B_R(x), \meas_X)$ be the $L^2$-strong limit of $f_i$ on $B_R(x)$. Then we have the following.
\begin{enumerate}
\item For any $\epsilon \in (0, 1)$, after passing to a subsequence, there exist a compact subset $A$ of $B_R(x)$ and a sequence of compact subsets $A_i$ of $B_R(x_i) \cap \mathrm{Leb}_1f_i$ such that $\meas_X(B_R(x) \setminus A)+\meas_{X_i}(B_R(x_i) \setminus A_i)<\epsilon$ holds for any $i$, that $A_i$ Gromov-Hausdorff converge to $A$ with respect to (\ref{eq:4}), that $\{f_i|_{A_i}\}_i$ is equi-Lipschitz, and that $\overline{f}_i(y_i) \to \overline{f}(y)$ holds whenever $y_i \in A_i \to y \in A \cap \mathrm{Leb}_1f$ (in particular $\overline{f}|_A$ is Lipschitz).
\item After passing to a subsequence, there exist a Borel subset $B$ of $B_R(x)$ and a sequence of Borel subsets $B_i$ of $B_R(x_i)$ such that $\meas_X(B_R(x)\setminus B)=0$ holds, that $B_i \subset B_R(x_i) \cap \mathrm{Leb}_1(f_i) \setminus D_i$ holds, $\meas_{X_i}(B_R(x_i) \setminus B_i) \to 0$ holds and that for any $y \in B$ there exists a sequence of $y_i \in B_i$ such that $y_i \to y$ holds and that $\overline{f}_i(y_i) \to \overline{f}(y)$ holds.
\end{enumerate}
\end{lemma}
\begin{proof}
Let us first check (1). 
Multiplying suitable cut-off functions to $f_i$, it is enough to check the assertion under assuming $f_i \in H^{1, 2}(X_i, \dist_{X_i}, \meas_{X_i}), f \in H^{1, 2}(X, \dist_X, \meas_X)$ with $C_0:=\sup_i\|f_i\|_{H^{1, 2}(X_i)}<\infty$.
Fix a sufficiently large $L \in [1, \infty)$, let 
\begin{equation}
\tilde{A}_{i}:=\left\{ z_i \in B_R(x_i); \sup_{r \in (0, \infty)}\frac{1}{\meas_{X_i}(B_r(z_i))}\int_{B_r(z_i)}|\nabla f_i|^2\di \meas_{X_i}\le L^2\right\}.
\end{equation}
The maximal function theorem shows that 
\begin{equation}
\meas_{X_i}(B_{R}(x_i) \setminus \tilde{A}_{i})\le \frac{C(\tilde{K}, \tilde{N}, R, C_0)}{L^2}\meas_{X_i}(B_{R}(x_i)).
\end{equation}
For any $i \in \mathbb{N}$, find a compact subset $A_{i}$ of $(B_{R-L^{-2}}(x_i) \cap \tilde{A}_{i} \cap \mathrm{Leb}_1(f_i)) \setminus D_i$ with
\begin{equation}
\meas_{X_i}(\tilde{A}_{i} \setminus A_{i})\le \frac{C(\tilde{K}, \tilde{N}, R, C_0)}{L^2}\meas_{X_i}(B_{R}(x_i)).
\end{equation}
After passing to a subsequence, with no loss of generality we can assume that there exists a compact subset $A$ of $\overline{B}_{R-L^{-2}}(x)$ such that $A_{i}$ Gromov-Hausdorff converge to $A$ with respect to (\ref{eq:4}) (see for instance subsection 2.2 of \cite{honda11}).
Then Vitali's covering theorem yields 
\begin{equation}
\limsup_{i \to \infty}\meas_{X_i}(A_{i})\le \meas_X(A),
\end{equation} 
in particular,
\begin{equation}
\frac{\meas_X(A)}{\meas_X(B_{R}(x))} \ge 1-\frac{C(\tilde{K}, \tilde{N}, R, C_0)}{L^2}.
\end{equation}
Take $y \in A \cap \mathrm{Leb}_1(f)$, fix a sufficiently small $\epsilon \in (0, 1)$ and find $r \in (0, L^{-1}\epsilon )$ with
\begin{equation}\label{asriairairairai}
\left|\overline{f}(y)-\frac{1}{\meas_X(B_r(y))}\int_{B_r(y)}f\di \meas_X\right|\le \epsilon.
\end{equation}
Let $y_i \in A_{i}$ converge to $y$. Then a Poincare inequality \cite[Th.1]{Rajala} with the definition of $A_{i}$ shows
\begin{align}\label{nauirbaiurbai}
&\left| \frac{1}{\meas_{X_i}(B_{s}(y_i))}\int_{B_s(y_i)}f_i\di \meas_{X_i} - \frac{1}{\meas_{X_i}(B_{2s}(y_i))}\int_{B_{2s}(y_i)}f_i\di \meas_{X_i}\right| \nonumber \\
&\le  \frac{1}{\meas_{X_i}(B_s(y_i))}\int_{B_s(y_i)}\left|f_i -\frac{1}{\meas_{X_i}(B_{2s}(y_i))}\int_{B_{2s}(y_i)}f_i\di \meas_{X_i} \right|\di \meas_{X_i} \nonumber \\
&\le \frac{C(\tilde{K}, \tilde{N})}{\meas_{X_i}(B_{2s}(y_i))}\int_{B_{2s}(y_i)}\left|f_i -\frac{1}{\meas_{X_i}(B_{2s}(y_i))}\int_{B_{2s}(y_i)}f_i\di \meas_{X_i} \right|\di \meas_{X_i} \nonumber \\
&\le C(\tilde{K}, \tilde{N})sL, \quad \forall s \in (0, r].
\end{align}
In particular letting $s:=2^{-i}r$ in (\ref{nauirbaiurbai}) and then taking the sum with respect to $i$ yield
\begin{equation}\label{asbrioabrioaiora}
\left| \overline{f}(y_i)-\frac{1}{\meas_{X_i}(B_r(y_i))}\int_{B_r(y_i)}f_i\di \meas_{X_i}\right| \le C(\tilde{K}, \tilde{N})rL \le C(\tilde{K}, \tilde{N})\epsilon.
\end{equation}
Thus combining (\ref{asriairairairai}) with (\ref{asbrioabrioaiora}) and the arbitrariness of $\epsilon$ implies
\begin{equation}
\overline{f}_i(y_i)\to \overline{f}(y)
\end{equation}
which completes the proof of (1) because $L$ is arbitrary.

It follows from (1) and a diagonal argument with $\epsilon \to 0^+$ that (2) holds.
\end{proof} 
We are now in a position to introduce a compactness result for approximate Sobolev maps with respect to the measured Gromov-Hausdorff convergence.
\begin{theorem}\label{prop:lowersemiapp}
Let $t_i \to t$ be a convergent sequence in $(0, \infty)$ and let $R \in (0, \infty]$.
If a sequence of $t_i$-Sobolev maps $f_i:B_R(x_i) \to Y_i$ satisfies 
\begin{equation}\label{absruairairubaiurai}
\liminf_{i \to \infty}\mathcal{E}_{B_R(x_i), Y_i, t_i}(f_i)<\infty,
\end{equation}
then after passing to a subsequence, there exists a $t$-Sobolev map $f:B_R(x) \to Y$ such that $\psi_i \circ f_i$ $L^2_{\mathrm{loc}}$-strongly converge to converge to $\psi \circ f$ on $B_R(x)$ for any uniformly convergent sequence of  equi-Lipschitz functions $\psi_i$ on $Y_i$ to $\psi$ on $Y$ and that 
\begin{equation}\label{eq:lowesemico2}
\liminf_{i \to \infty}\int_{B_R(x_i)}\phi_i e_{Y_i, t_i}(f_i)\di \meas_{X_i} \ge \int_{B_R(x)}\phi e_{Y, t}(f)\di \meas_X
\end{equation}
for any $L^2_{\mathrm{loc}}$-strongly convergent sequence $\phi_i \to \phi$ with $\phi_i \ge 0$ and $\sup_i\|\phi_i\|_{L^{\infty}}<\infty$.
In particular
\begin{equation}\label{eq:lowesemico3}
\liminf_{i \to \infty}\mathcal{E}_{B_R(x_i), Y_i, t_i}(f_i)\ge \mathcal{E}_{B_R(x), Y, t}(f).
\end{equation}
\end{theorem}
\begin{proof}
Thanks to Theorem \ref{thm:spectral} with the gradient estimates (\ref{eq:eigenfunction}), with no loss of generality we can assume that $\phi_i^{Y_j}$ converge uniformly to $\phi_i^Y$ with $\lambda_i^{Y_j} \to \lambda_i^Y$. By (\ref{absruairairubaiurai}), we have
\begin{equation}
\sup_j\|e^{-\lambda_i^{Y_j}t_j}\phi_i^{Y_j}\circ f_j\|_{H^{1, 2}(B_R(x_j))}<\infty, \quad \forall i \in \mathbb{N}.
\end{equation}
Thus  by Theorem \ref{bbbg}, after passing to a subsequence with a diagonal argument, for any $i \in \mathbb{N}$, there exists $F_i \in H^{1, 2}(B_R(x), \dist_X, \meas_X)$ such that $e^{-\lambda_i^{Y_j}t_j}\phi_i^{Y_j}\circ f_j$ $L^2_{\mathrm{loc}}$-strongly converge to $F_i$ on $B_R(x)$ and that $\dist (e^{-\lambda_i^{Y_j}t_j}\phi_i^{Y_j}\circ f_j)$ $L^2$-weakly converge to $\dist F_i$ on $B_R(x)$. Then since Lemma \ref{pointwiseconv} ensures that $F(\tilde{x}) \in \Phi_t^{\ell^2}(Y)$ holds for $\meas_X$-a.e. $\tilde{x} \in B_R(x)$, where $F:=(F_i)_i$, by letting $f:=(\Phi_t^{\ell^2})^{-1}\circ F$, 
it follows from the same argument as in the proof of Theorem \ref{propcompact} that the desired conlusions hold except for the convergence of $\psi_i \circ f_i$.

From the above argument we know that $\phi_i^{Y_j}\circ f_j$ $L^2_{\mathrm{loc}}$-strongly converge to $\phi_i^Y\circ f$ for any $i$. In particular thanks to (\ref{eq:expansion1}) with (\ref{eq:eigenfunction}), we see that $(h_s^{Y_i}\psi_i) \circ f$  $L^2_{\mathrm{loc}}$-strongly converge to $(h_s^Y\psi) \circ f$ for any $s \in (0, \infty)$.
Since $\lim_{s \to 0^+}\sup_i\|h_s^{Y_i}\psi_i-\psi_i\|_{L^{\infty}} =0$ holds because of (\ref{absbayraywsrai}) (see also \cite[Prop.1.4.6]{AmbrosioHonda}), we have the desired convergence of $\psi_i \circ f_i$ to $\psi \circ f$.
\end{proof}
\begin{proposition}\label{prop:stability}
Assume that $(Y_i, \dist_i, \meas_i)$ and $(Y, \dist_Y, \meas_Y)$ are non-collapsed, namely, $\meas_{Y_i}=\mathcal{H}^N$ and $\meas_Y=\mathcal{H}^N$ are satisfied.
Let $\epsilon_i \to \epsilon, t_i \to t, \tau_i \to \tau$ be convergent sequences in $(0, \infty)$. Then if a sequence of points $y_i \in \mathcal{R}_{Y_i}(\epsilon_i, t_i, \tau_i)$ converge to a point $y \in Y$ with respect to (\ref{eq:mgh}), then $y \in \mathcal{R}_Y(\epsilon, t, \tau)$.
\end{proposition}
\begin{proof}
Since for all $r \in (0, \tau]$ and $i \in \mathbb{N}$,
\begin{equation}\label{eq:3}
\frac{1}{\mathcal{H}^N(B_r(y_i))}\int_{B_r(y_i)}|g_{Y_i}-c_Nt_i^{(N+2)/2}g_{t_i}^{Y_i}|\di \mathcal{H}^N\le\epsilon_i,
\end{equation}
letting $i \to \infty$ in (\ref{eq:3}) with \cite[Th.5.19]{AHPT} and Proposition \ref{weakriem} shows
\begin{equation}
\frac{1}{\mathcal{H}^N(B_r(y))}\int_{B_r(y)}|g_{Y}-c_Nt^{(N+2)/2}g_t^{Y}|\di \mathcal{H}^N\le \epsilon,
\end{equation}
which completes the proof.
\end{proof}
\begin{definition}[Uniformly weakly smooth subsets with respect to mGH convergence]
We say that a sequence of Borel subsets $A_i$ of $Y_i$ is said to be \textit{uniformly weakly smooth for $\{(t_i, \tau_i)\}_i$} if for any $\epsilon \in (0, 1)$ there exists $i \in \mathbb{N}$ such that 
\begin{equation}
A_l \subset \bigcap_{j \ge i}\mathcal{R}_{Y_l}(\epsilon, t_j, \tau_j), \quad \forall l \in \mathbb{N}
\end{equation}
holds.
\end{definition}
Let us introduce a compactness result for Sobolev maps with respect to the measured Gromov-Hausdorff convergence.
\begin{theorem}\label{hybfgaty}
Assume that $(Y_i, \dist_i, \meas_i)$ and $(Y, \dist_Y, \meas_Y)$ are non-collapsed.
Let $R \in (0, \infty]$ and let $f_i:B_R(x_i) \to Y_i$ be a Sobolev map. In addition, assume that the following two conditions hold.
\begin{enumerate}
\item The sequence $\{f_i(B_R(x_i) \setminus D_i)\}_i$ is uniformly weakly smooth for some $\{(t_i, \tau_i)\}_i$ and some $\meas_{X_i}$-negligible subsets $D_i$ of $B_R(x_i)$. 
\item We have
\begin{equation}
\liminf_{i \to \infty}\mathcal{E}_{B_R(x_i), Y_i}(f_i)<\infty.
\end{equation}
\end{enumerate}
Then after passing to a subsequence there exists a Sobolev map $f:B_R(x) \to Y$ such that $f(B_R(x) \setminus D)$ is uniformly smooth for $\{(t_i, \tau_i)\}_i$ for some $\meas_X$-negligible set $D$, that $\psi_i \circ f_i$ $L^2_{\mathrm{loc}}$-strongly converge to $\psi \circ f$ on $B_R(x)$ for any uniformly convergent sequence of equi-Lipschitz functions $\psi_i$ on $Y_i$ to $\psi$ on $Y$ and that 
\begin{equation}\label{eq:lowesemico4}
\liminf_{i \to \infty}\int_{B_R(x_i)}\phi_i e_{Y_i}(f_i)\di \meas_{X_i} \ge \int_{B_R(x)}\phi e_{Y}(f)\di \meas_X
\end{equation}
for any $L^2_{\mathrm{loc}}$-strongly convergent sequence $\phi_i \to \phi$ with $\phi_i \ge 0$ and $\sup_i\|\phi_i\|_{L^{\infty}}<\infty$.
In particular
\begin{equation}\label{eq:lowesemico5}
\liminf_{i \to \infty}\mathcal{E}_{B_R(x_i), Y_i}(f_i)\ge \mathcal{E}_{B_R(x), Y}(f).
\end{equation}
\end{theorem}
\begin{proof}
The proof is essentially same to that of Theorem \ref{corcompactness}. It is trivial that after passing to a subsequence, with no loss of generality we can assume $\sup_i\mathcal{E}_{B_R(x_i), Y_i}(f_i)<\infty$.
Thus by (\ref{eqeq}) we know
\begin{equation}
\sup_{i, l}t_i^{(N+2)/2}\mathcal{E}_{B_R(x_l), Y_l, t_i}(f_l)<\infty.
\end{equation}
Thus Theorem \ref{prop:lowersemiapp} with Lemma \ref{pointwiseconv} and Proposition \ref{prop:stability} allows us to prove that after passing to a subsequence, there exists a map $f:B_R(x) \to Y$ such that $f$ is a $t_i$-Sobolev map for any $i$ and that $f(B_R(x) \setminus D)$ is uniformly smooth associated with $\{(t_i, \tau_i)\}_i$ for some $\meas_X$-negligible set $D$. 
The remaining statements follows from the proofs of Theorems \ref{corcompactness} and \ref{prop:lowersemiapp} with Theorem \ref{spectral2}.
\end{proof}
\begin{remark}
In Theorems \ref{prop:lowersemiapp} and \ref{hybfgaty}, if we consider the case when $(X_i, \dist_{X_i}, \meas_{X_i}, x_i) \equiv (X, \dist_X, \meas_X, x)$ and $(Y_i, \dist_{Y_i}, \mathcal{H}^N)\equiv (Y, \dist_Y, \mathcal{H}^N)$, then thanks to the dominated convergence theorem, the $L_{\mathrm{loc}}^2$-strong convergence of $\psi_i \circ f_i$ for any $\psi_i$ is equivalent to the $\meas_X$-a.e. pointwise convergence of $f_i$, up to passing to a subsequence (recall the proofs of Proposition \ref{propcompact} and of Corollary \ref{corcompactness}).
\end{remark}
Finally let us end this section by giving a compactness result for Lipschitz maps defined on Borel subsets. 
\begin{theorem}\label{thmcompacatara}
Assume that $(Y_i, \dist_i, \meas_i)$ and $(Y, \dist_Y, \meas_Y)$ are non-collapsed.
Let $A_i, A$ be Borel subsets of $X_i, X$, respectively, let $L \in (0, \infty)$ and let $f_i:A_i \to Y_i$ be a sequence of $L$-Lipschitz maps.
In addition, assume that the following two conditions hold.
\begin{enumerate}
\item The sequence $\{f_i(A_i \setminus D_i)\}_i$ is uniformly weakly smooth for some $\{(t_i, \tau_i)\}_i$ and some $\meas_{X_i}$-negligible subsets $D_i$ of $A_i$.
\item The functions $1_{A_i}$ $L^2_{\mathrm{loc}}$-strongly converge to $1_A$.
\end{enumerate}
Then after passing to a subsequence there exist a Borel subset $\tilde{A}$ of $A$ and an $L$-Lipschitz map $f:\tilde{A} \to Y$ such that $\meas_X(A \setminus \tilde{A})=0$ holds, that $f(\tilde{A})$ is uniformly smooth associated with $\{(t_i, \tau_i)\}_i$, that $\psi_i \circ f_i$ $L^2_{\mathrm{loc}}$-strongly converge to $\psi \circ f$ on $\tilde{A}$ for any uniformly convergent sequence of  equi-Lipschitz functions $\psi_i$ on $Y_i$ to $\psi$ on $Y$ and that 
\begin{equation}\label{eq:lowesemico43}
\liminf_{i \to \infty}\int_{A_i}\phi_i e_{Y_i}(f_i)\di \meas_{X_i} \ge \int_{A}\phi e_{Y}(f)\di \meas_X
\end{equation}
for any $L^2_{\mathrm{loc}}$-strongly convergent sequence $\phi_i \to \phi$ with $\phi_i \ge 0$ and $\sup_i\|\phi_i\|_{L^{\infty}}<\infty$.
In particular
\begin{equation}\label{eq:lowesemico53}
\liminf_{i \to \infty}\mathcal{E}_{A_i, Y_i}(f_i)\ge \mathcal{E}_{A, Y}(f).
\end{equation}
\end{theorem}
\begin{proof}
For any $i, j$, applying Macshane's lemma for $\phi_i^{Y_j}\circ f_j$, there exist a Lipschitz function $F_{j, i}: X_j \to \mathbb{R}$ such that $\mathbf{Lip}F_{j, i} \le \mathbf{Lip}\phi_i^{Y_j} \cdot L$ holds and that $F_{j, i} \equiv \phi_i^{Y_j}\circ f_j$ holds on $A_j$. Thus by Theorem \ref{bbbg}, after passing to a subsequence for any $i$ there exists a Lipschitz function $F_i:X \to \mathbb{R}$ such that $F_{j, i}$ uniformly converge to $F_i$ on any bounded subset of $X$ and that $\dist F_{j, i}$ $L^2_{\mathrm{loc}}$-weakly converge to $\dist F_i$. In particular (recalling  (\ref{eq:eigenfunction})) letting  $F_j^t:=(e^{-\lambda_i^{Y_j}t}F_{j, i})_i:X_j \to \ell^2$ uniformly converge  to $F^t:=(e^{-\lambda_j^Yt}F_i)_i:X \to \ell^2$ on any bounded subset of $X$. Then Lemma \ref{pointwiseconv} ensures that there exists a Borel subset $\tilde{A}$ of $A$ such that $\meas_X(A \setminus \tilde{A})=0$ holds, that $F^t(A) \in \Phi_t^{\ell^2}(Y)$ holds and that the map $f:=(\Phi_t^{\ell^2})^{-1} \circ F^t$ on $\tilde{A}$ does not depend on $t$ and it is an $L$-Lipschitz function, where we used the pointwise convergence of $f_j$ to $f$ for the last statement. Then the remaining statements follow from Proposition \ref{prop:stability}, the proofs of Theorems \ref{propcompact} and \ref{prop:lowersemiapp}, and a fact that $1_{A_j}\dist F_{j, i}$ $L^2_{\mathrm{loc}}$-weakly converge to $1_A\dist F_i$.
\end{proof}
\section{Appendix: locally bi-Lipschitz embeddability by heat kernel in general case}\label{biliplp}
In this appendix we will generalize several results proved in Section \ref{secembed} to general finite dimensional $\RCD$ spaces. For any two $\meas_X$-a.e. symmetric tensors $T_i (i=1, 2) \in L^2((T^*)^{\otimes 2}(A, \dist_X, \meas_X))$ over a Borel subset $A$ of an $\RCD$ space $(X, \dist_X, \meas_X)$, we say that $T_1 \le T_2$ holds for $\meas_X$-a.e. $x \in A$ if
\begin{equation}
T_1(V, V) \le T_2(V, V),\quad \text{for $\meas_X$-a.e. $x \in A$}
\end{equation}
for any $V \in L^{\infty}(T(A, \dist_X, \meas_X))$. Recall (\ref{anasuarbaj}) for the notation of $\Psi$ and recall  $g_t:=\Phi_t^*g_{L^2}$.
The following gives a variant of Proposition \ref{prop:quantapp}.
\begin{proposition}\label{aborbairaibbiaroibwoia}
Let $\epsilon \in (0, 1)$, $K \in \mathbb{R}, N \in [1, \infty)$ and $C_1, C_2, d, t, \epsilon \in (0, \infty)$ with $C_1 \le C_2$ and let $(X, \dist_X, \meas_X, x)$ be a pointed compact $\RCD(K, N)$ space with $\mathrm{diam}(X, \dist_X)\le d$. If there exist a Borel subset $A$ of $X$ and $r \in (0, d]$ such that
\begin{equation}
C_1g_X\le g_t\le C_2g_X,
\end{equation}
holds for $\meas_X$-a.e. $z \in A$, that
\begin{equation}
\frac{\meas_X(B_r(x) \cap A)}{\meas_X(B_r(x))}\ge 1-\epsilon;
\end{equation}
holds and that $(X, r^{-1}\dist_X, \meas_X(B_r(x))^{-1}\meas_X, x)$ is $\epsilon$-pointed measured Gromov-Hausdorff close to $(\mathbb{R}^n, \dist_{\mathbb{R}^n}, \omega_n^{-1}\mathcal{H}^n, 0_n)$,
then we have
\begin{equation}
C_1\dist_X(y, z) -\Psi C_1 r \le \|\Phi_t(y)-\Phi_t(z)\|_{L^2}\le C_2\dist_X(y, z)+\Psi C_2 r,\quad \forall y,\,\forall z \in B_r(x), 
\end{equation}
where $\Psi:=\Psi(\epsilon, r; K, N, C_1, C_2, d, t)$.
\end{proposition}
\begin{proof}
The proof is done by a contradiction. If not, then there exist $\tau \in (0, 1)$, a sequence of pointed $\RCD(K, N)$ spaces $(X_i, \dist_{X_i}, \meas_{X_i}, x_i)$, a convergent sequence $r_i \to 0^+$, a sequence of Borel subsets $A_i$ of $X_i$ and sequences of $y_i, z_i \in B_{r_i}(x_i)$ such that 
\begin{itemize}
\item we have
\begin{equation}\label{asatraowearb}
C_1g_{X_i}\le g_t^{X_i}\le C_2g_{X_i},\quad \text{for $\meas_{X_i}$-a.e. $z \in A_i$},
\end{equation} 
\item we have
\begin{equation}\label{aburbausaisb}
\frac{\meas_{X_i}(B_{r_i}(x_i) \cap A_i)}{\meas_{X_i}(B_{r_i}(x_i))} \to 1,
\end{equation}
\item we have
\begin{equation}\label{eucldieb}
\left( X_i, r_i^{-1}\dist_{X_i}, \meas_{X_i}(B_{r_i}(x_i))^{-1}\meas_{X_i}, x_i\right) \stackrel{\mathrm{pmGH}}{\to} (\mathbb{R}^n, \dist_{\mathbb{R}^n}, \omega_n^{-1}\mathcal{H}^n, 0_n),
\end{equation}
\item either
\begin{equation}\label{0088jhu}
\|\Phi_{t}^{X_i}(y_i)-\Phi_t^{X_i}(z_i)\|_{L^2}<C_1\dist_{X_i}(y_i, z_i)-C_1\tau r_i
\end{equation}
or
\begin{equation}\label{fabseaurabwu}
C_2\dist_{X_i}(y_i, z_i)+C_2\tau r_i < \|\Phi_{t}^{X_i}(y_i)-\Phi_t^{X_i}(z_i)\|_{L^2}
\end{equation}
is satisfied. 
\end{itemize}
Let us consider functions on $(X_i, r_i^{-1}\dist_{X_i})$ defined by
\begin{equation}
\overline{\phi}_{i, j}:=\frac{e^{-\lambda^{X_i}_jt}}{r_i}\left( \phi^{X_i}_j-\frac{1}{\meas_{X_i}(B_{r_i}(x_i))}\int_{B_{r_i}(x_i)}\phi_j^{X_i}\di \meas_{X_i}\right).
\end{equation}
Then by an argument similar to the proof of Proposition \ref{prop:quantapp}, after passing to a subsequence, there exists a Lipschitz map $\overline{\Phi}:=(\overline{\phi}_j)_j:\mathbb{R}^n \to \ell^2$ such that the maps $\overline{\Phi}_t^{\ell^2}:=(\overline{\phi}_{i, j})_j:(X_i, r_i^{-1}\dist_{X_i}) \to \ell^2$ uniformly converge to $\overline{\Phi}$ on any bounded subset of $\mathbb{R}^n$ and that each $\overline{\phi}_j$ is linear.
Thanks to (\ref{asatraowearb}) and (\ref{aburbausaisb}) with \cite[Th.1.10.2]{AmbrosioHonda}, it is easily checked that 
\begin{equation}\label{ppsbaywrh}
C_1g_{\mathbb{R}^n}\le \overline{\Phi}^*g_{\ell^2}\le C_2g_{\mathbb{R}^n}
\end{equation}
holds on $B_1(0_n)$. Thus by the lineality of $\overline{\phi}_j$, (\ref{ppsbaywrh}) holds on $\mathbb{R}^n$ with
\begin{equation}\label{aasbfay7666}
C_1\dist_{\mathbb{R}^n}(z, w)\le \|\overline{\Phi}(z)-\overline{\Phi}(w)\|_{\ell^2}\le C_2\dist_{\mathbb{R}^n}(z, w),\quad \forall z,\,\forall w \in \mathbb{R}^n.
\end{equation}
On the other hand, after passing to a subsequene, we have $y_i \to y, z_i \to z$ with respect to (\ref{eucldieb}) for some $y, z \in \overline{B}_1(0_n)$. Then (\ref{0088jhu}) and (\ref{fabseaurabwu}) imply that either
\begin{equation}
\|\overline{\Phi}(y)-\overline{\Phi} (z)\|_{\ell^2}\le C_1\dist_{\mathbb{R}^n}(y, z)-C_1\tau
\end{equation}
or
\begin{equation}
C_2\dist_{\mathbb{R}^n}(y, z)+C_2\tau \le \|\overline{\Phi}(y)-\overline{\Phi}(z)\|_{\ell^2}
\end{equation}
is satisfied, which contradicts (\ref{aasbfay7666}).
\end{proof}
We are now in a position to introduce the desired bi-Lipschitz embeddability for $\Phi_t$. Recall that $t\meas_X(B_{\sqrt{t}}( \cdot))g_t$ $L^p$-strongly converge to $\overline{c}_ng$ on $X$ for any $p \in [1, \infty)$ (see subsection \ref{subsec:pullbak}).
\begin{theorem}[Bi-Lipschitz embedding]
Let $(X, \dist_X, \meas_X)$ be a finite dimensional compact $\RCD$ space. Then for any $\epsilon \in (0, 1)$ there exists $t_0 \in (0, 1)$ such that for any $t \in (0, t_0]$ there exists a compact subset $X_{\epsilon, t}$ of $X$ such that $\meas_X(X \setminus X_{\epsilon, t})\le \epsilon$ holds and that $\Phi_t|_{X_{\epsilon, t}}$ is a bi-Lipschitz embedding.
\end{theorem}
\begin{proof}
Let us denote by $n$ the essential dimension of $(X, \dist_X, \meas_X)$ and assume that $(X, \dist_X, \meas_X)$ is an $\RCD(K, N)$ space with $\mathrm{diam}(X, \dist_X)\le d<\infty$ and $\meas_X(X)=1$.
Fix a sufficiently small $\delta \in (0, 1)$. Thanks to Theorem \ref{thm:RN}, there exist $\tau, \tau_1, \tau_2 \in (0, \infty)$ with $\tau_1 \le \tau_2$ and a Borel subset $A_1$ of $X$ such that
\begin{equation}
\tau_1 \le \frac{\meas_X(B_r(x))}{r^n}\le \tau, \quad  \forall r \in (0, \tau),\,\,\forall x \in A_1
\end{equation}
and that $\meas_X(X \setminus A_1)\le \delta$ holds.
Find $t_0 \in (0, 1)$ with
\begin{equation}
\int_X\left|\overline{c}_ng-t\meas_X(B_{\sqrt{t}}(\cdot ))g_t\right|\di \meas_X<\delta,\quad \forall t \in (0, t_0].
\end{equation}
Fix $t \in (0, t_0]$ and put
\begin{equation}
A_2:=\left\{ x \in X; \sup_{r>0}\frac{1}{\meas_X(B_r(x))}\int_{B_r(x)}\left| \overline{c}_ng-t\meas_X(B_{\sqrt{t}}(\cdot ))g_t\right| \di \meas_X \le \delta^{1/2}\right\}.
\end{equation}
Then the maximal function theorem shows
\begin{equation}\label{asdabuatb}
\meas_X\left(X \setminus A_2\right)\le \frac{C(K, N, d)}{\delta^{1/2}}\int_X\left|\overline{c}_ng-t\meas_X(B_{\sqrt{t}}(\cdot ))g_t\right|\di \meas_X\le C(K, N, d)\delta^{1/2}.
\end{equation}
Note that we can find $C_1, C_2 \in (0, \infty)$ with  
\begin{equation}
C_1g_X\le g_t \le C_2g_X, \quad \text{for $\meas_X$-a.e. $x \in A_2$}.
\end{equation}
Then fix a sufficiently small $\epsilon \in (0, 1)$. Thanks to Theorem \ref{thm:RN} again, there exist $\tau_3 \in (0, 1)$ and a Borel subset $A_3$ of $A_1 \cap A_2$ such that $\meas_X((A_1 \cap A_2) \setminus A_3)<\epsilon$ holds and that the rescaled space $(X, r^{-1}\dist_X, \meas_X(B_r(x))^{-1}\meas_X, x)$ is $\epsilon$-pointed measured Gromov-Hausdorff close to $(\mathbb{R}^n, \dist_{\mathbb{R}^n}, \omega_n^{-1}\mathcal{H}^n, 0_n)$ for all $x \in A_3$ and $r \in (0, \tau_3)$. Applying (\ref{lebdesntity}), there exist $\tau_4 \in (0, \epsilon]$ and a Borel subset $A_4$ of $A_3$ such that $\meas_X(A_3 \setminus A_4)<\epsilon$ and 
\begin{equation}
\frac{\meas_X(B_r(x) \cap A_3)}{\meas_X(B_r(x))}\ge 1-\epsilon, \quad \forall x \in A_4,\,\,\forall r \in (0, \tau_4]
\end{equation}
hold. Then applying Proposition \ref{aborbairaibbiaroibwoia} as $A=A_{3}$ and $x \in A_{4}$ shows that 
\begin{equation}\label{aahhhsuab}
C_1\dist_X(y, z)-\Psi C_1 r\le \|\Phi_t(y)-\Phi_t(z)\|_{L^2}\le C_2\dist_X(y, z)+\Psi C_2 r
\end{equation}
holds for all $r \in (0, \tau_4]$ and $y, z \in B_r(x)$, where $\Psi:=\Psi(\epsilon; K, N, C_1, C_2, d, t)$ (recall $r<\tau_4 \le \epsilon$). In particular for all $x \in A_{4}$ and $y, z \in B_{\tau_4/4}(x)\cap A_4$ with $y \neq z$, letting $r:=\dist_X(y, z)$ with (\ref{aahhhsuab}) yields
\begin{equation}
C_1(1-\Psi)\dist_X(y, z)\le \|\Phi_t(y)- \Phi_t(z)\|_{L^2}\le C_2(1+\Psi)\dist_X(y, z) 
\end{equation}
which proves that $\Phi_t|_{A_4}$ is a locally bi-Lipschitz embedding. Replacing $A_4$ by a compact subset $A_5$ of $A_4$ with $\meas_X(A_4\setminus A_5)<\epsilon$, we can easily prove that $\Phi_t|_{A_5}$ is a bi-Lipschitz embedding. Thus we conlude because $\delta, \epsilon$ are arbitrary.
\end{proof}
Similarly we are able to prove the following finite dimensional reduction of the above result. We omit the proof. Compare with  Theorem \ref{prop:finitedimen}.
\begin{theorem}
Let $(X, \dist_X, \meas_X)$ be a finite dimensional compact $\RCD$ space. Then for any $\epsilon \in (0, 1)$ there exists  $t_0 \in (0, 1)$ such that for any $t \in (0, t_0]$ there exists a compact subset $X_{\epsilon, t}$ of $X$ such that $\meas_X(X \setminus X_{\epsilon, t})\le \epsilon$ holds and that $\Phi_t^l|_{X_{\epsilon, t}}$ is a bi-Lipschitz embedding for any sufficiently large $l$. 
\end{theorem}

\subsection*{Conflict of interests:} The authors declare no conflict of interest.

\end{document}